\documentclass[reqno,a4paper,12pt]{amsart}
\usepackage[all]{xy}
\usepackage{amsfonts}
\usepackage[mathcal]{eucal}
\usepackage{eufrak}
\usepackage{amssymb}
\usepackage{amsmath}
\usepackage{color}
\usepackage[pagebackref,colorlinks]{hyperref}
\newtheorem{thm}{Theorem}[section]
\newtheorem{cor}[thm]{Corollary}
\newtheorem{lem}[thm]{Lemma}
\newtheorem{prop}[thm]{Proposition}
\theoremstyle{definition}

\theoremstyle{remark}
\newtheorem{rem}[thm]{Remark}
\theoremstyle{question}
\newtheorem{que}[thm]{Question}
\theoremstyle{definition}
\newtheorem{exa}[thm]{Example}
\theoremstyle{definition}
\newtheorem{conj}[thm]{Conjecture}
\numberwithin{equation}{section}
\newcommand{\norm}[1]{\left\Vert#1\right\Vert}
\newcommand{\set}[1]{\left\{#1\right\}}
\newcommand{\ra}{\rightarrow}

\newcommand{\mi}{\setminus}

\newcommand{\ov}{\overline}

\newcommand{\dlambda}{\delta}

\newcommand{\lbr}{[}
\newcommand{\rbr}{]}

\newcommand{\conggr}{\cong_{\mathrm{gr}}}

\newcommand{\Ga}{\Gamma}
\newcommand{\ga}{\gamma}
\newcommand{\al}{\alpha}

\newcommand{\de}{\delta}

\newcommand{\ti}{\widetilde}

\newcommand\CK[1][1]{\operatorname{CK}_{#1}}
\newcommand\ZK[1][1]{\operatorname{ZK}_{#1}}

\newcommand\SH{\operatorname{SH^0}}

\newcommand\NK[1][1]{\operatorname{NK}_{#1}}
\newcommand\SK[1][1]{\operatorname{SK}_{#1}}
\newcommand{\Nrd}[1][{}]{{\operatorname{Nrd}_{#1}}}
\newcommand{\Trd}[1][{}]{{\operatorname{Trd}_{#1}}}
\newcommand{\Tr}[1][{}]{{\operatorname{Tr}_{#1}}}

\newcommand{\Aut}{\operatorname{Aut}}

\newcommand{\chr}{\operatorname{char}}
\newcommand{\ind}{\operatorname{ind}}

\newcommand{\Gal}{\operatorname{Gal}}

\newcommand{\gr}{\text{gr}}

\newcommand{\Supp}{\operatorname{Supp}}

\begin{document}
\author{R. Hazrat}%
\address{Centre for Research in Mathematics, University of Western Sydney, Australia}%
\email{r.hazrat@uws.edu.au}%
\author{M. Mahdavi-Hezavehi}%
\address{Department of Mathematical Sciences, Sharif University of Technology, Tehran, Iran}%
\email{mahdavih@sharif.edu}%
\author{M. Motiee}%
\address{Faculty of Basic Sciences, Babol University of Technology, Babol, Iran}%
\email{motiee@nit.ac.ir}%

\title[Multiplicative groups of division rings]
{Multiplicative groups of division rings}
\address{}%
\email{}%

\thanks{}%

\begin{abstract}
Exactly 170 years ago, the construction of the real quaternion algebra by William Hamilton was announced in the Proceedings of the Royal Irish Academy.
It became the first example of non-commutative division rings and a major turning point of algebra. To this day,  the multiplicative group structure of quaternion algebras have not completely been understood. This article is a long survey of the recent developments on the multiplicative group structure of division rings.
\end{abstract}
\maketitle
\tableofcontents
\section{Introduction and review of old results}\label{sec1birk}
Let $D$ be a division ring, $D^*=D\setminus \{0\}$ its unit group
and $Z(D)$ its center. Evidently, the first departure to connect the
algebraic properties of $D$ to that of $D^*$ was done by Wedderburn
in 1905. By his theorem, if $D$ has a finite cardinal number then
$D^*$ is abelian and hence $D$ is a field (See \cite[p.
179]{Kar88}). Since Wedderburn's result classifies a class of
division rings in terms of a group theoretic (more exactly set
theoretic) property, one way to generalize this observation is to
attribute a group theoretic property to $D^*$ and then to explore
which division rings enjoy this property. This point of view also
includes the detection of the exclusivity of subgroups and quotients
of $D^*$. Namely, from Wedderburn's theorem one can easily conclude
that if $\operatorname{char}(D)\neq 0$ then every finite subgroup of
$D^*$ is a cyclic group. In 1953, Herstein determined some finite
groups that can be embedded as a subgroup in a division ring
\cite{Her53}. But, the characterization of all finite groups that
can occur as a subgroup of a division ring was done by Amitsur in
\cite{Am55}. According to Amitsur's results, a finite group $G$ sits
in the unit group of a division ring if and only if $G$ is one of
the following types:
\begin{enumerate}
  \item[(A1)] A group that all of its Sylow subgroups are cyclic.
  \item[(A2)] $C_m\rtimes Q_{2^{n+1}}$, where $C_m$ is a cyclic group of odd order $m$, an
  element of order $2^n$ of $Q_{2^{n+1}}$
  centralizes $C_m$ and an element of order 4 inverts $C_m$.
  \item[(A3)] $Q_8\times M$, where $M$ is a group of type (A1).
  \item[(A4)] The binary octahedral group of order 48.
  \item[(A5)] $\operatorname{SL}(2,3)\times M$, where $M$ is a group of type (A1).
  \item[(A6)] The group $\operatorname{SL}(2,5)$.
\end{enumerate}
where $Q_{2^{n+1}}$ is the generalized quaternion group of order
$2^{n+1}$ for each $n\in\mathbb{N}$. Moreover, in \cite{Fau69} and
\cite[\S2]{ShWe86} successful efforts have been made to study the
structure of locally finite subgroups of a division ring. We recall
that a group $G$ is called locally finite if every finitely
generated subgroup of $G$ is finite. A remarkable result concerning
the quotients of $D^*$ is due to Kaplansky which asserts that if
$D^*$ is center-by-periodic then it is commutative \cite{Kap51}.
This contains Noether-Jacobson Theorem as a special case. (Recall
that by Noether-Jacobson Theorem, if $D$ is an algebraic division
algebra over some subfield $F$, then $D^*$ contains an element $a$
not in $F$ that is separable over $F$ \cite[p.~244]{Lam01}.) Before
1950, an interesting problem concerning the group structure of $D$
was to figure out how far $D^*$ is from commutativity.
Thus, it was reasonable for one to explore what happens when $D^*$
is nilpotent or more specifically is soluble. Guided by this viewpoint, the
roles of the multiplicative commutators in the structure of $D$ were
made prominent. In this direction in \cite{Hua49} Hua has proved
that if for a natural number $r\geq 2$, all $r$-mixed commutators
(elements of $D^*$ that are defined inductively by
$[a_1,a_2]=a_1^{-1}a_2^{-1}a_1a_2$ and
$[a_1,a_2,\ldots,a_r]=[a_1,[a_2,\ldots,a_r]$) of $D^*$ lie in some
division subring, then $D$ is a field. Moreover, he proved that if
an element of $D^*$ commutes with all $r$-mixed commutators, then it
is central. Finally, in 1950 he proved that if $D^*$ is soluble,
then $D$ is a field \cite{Hua50}. At the other extreme, the roles of
additive and multiplicative commutators in the structure of $D^*$
were extensively studied in 1970s by some authors. In
\cite{HerProSch75} Herstien, Procesi and Schacher showed that if
every additive commutator (an element of the form $xy-yx$ for some
$x,y\in D^*$) is torsion modulo $Z(D)$ then $[D:Z(D)]\leq 4$. Putcha
and Yaqub also proved that if every multiplicative commutator is
killed by a power of 2 in $D^*/Z(D^*)$ then $D^*$ is commutative
\cite{PutYaq74}. However, the general case of this problem was
investigated by Herstein in \cite{Her78} by proving that if either
every $[a,b]$ has a finite order in $D^*$ or $D$ is of finite
dimension over $Z(D)$ and every $[a,b]$ is torsion modulo $Z(D)$
then $D$ is a field. An alternative interesting result of
\cite{Her78} asserts that every subnormal periodic subgroup of $D^*$
is central. Afterwards, in \cite{Her80} he showed that whenever
$Z(D)$ is countable, then (i) if $[a,b]^{n}\in Z(D)$ for all $a,b\in
D^*$, for some $n\in \mathbb{N}$ then $D$ is a field, and (ii) If
$N$ is a subnormal subgroup of $D^*$ that is periodic modulo $Z(D)$
then $N$ is central. The last result is a special type of another
absorbing old problem determining how much subnormal subgroups of
$D^*$ reflect the multiplicative structure of $D^*$. In other words,
how ``big'' subnormal subgroups are in $D^*$. One of the earlier
known results depending on this subject had been obtained by
Herstein in \cite{Her56} by showing that every noncentral element of
$D^*$ possesses infinitely many conjugates. Then, in 1957 Scott
proved that $D^*/Z(D^*)$ has no nontrivial abelian normal subgroup
or equivalently every abelian normal subgroup of $D^*$ is central
\cite{Sco57}. He also extended the results of \cite{Her56} by
demonstrating that the number of conjugates of every noncentral
element in $D^*$ is equal to $|D|$, the cardinal of $D$. In 1960, in
\cite{Huz60} Huzurbazar provided a generalization of Scott's
theorem. His generalization states that every locally nilpotent
subnormal subgroup of $D^*$ is central. But, the most important
result concerning the structure of subnormal subgroups was obtained
by Stuth in 1964 asserting that (i) If $G$ is a noncentral subnormal
subgroup of $D^*$ and $x^G$ is the conjugacy  class of the
noncentral element $x\in D^*$ in $G$, then the division subring
generated by $x^G$ is $D$, (ii) Every soluble subnormal subgroup of
$D^*$ is central \cite{Stu64}. Another aspect in the study of the
unit groups of division rings is to achieve analogous properties of
general linear groups over fields. Namely, Bachmuth asked whether
Tits' Alternative is valid for general linear groups over division
rings. In \cite{Lic77} Lichtman answered it negatively by showing
that every division ring can be embedded into a division ring $K$ so
that $K^*$ contains a finitely generated subgroup that is not
soluble-by-finite and does not contains a noncyclic free subgroup.
Then he asked whether $D^*$ has a noncyclic free subgroup. A more
general question posed as a conjecture by Goncalves and Mandel in
\cite{GonMan86}. More exactly, they conjectured that every
noncentral subnormal subgroup of a division ring contains a
noncyclic free subgroup. They also proved that the conjecture is
true in several cases (see \cite[p. 208]{Kar88}). Also, in a
preceding work, Goncalves had shown that the conjecture has a
positive answer whenever $D$ is of finite dimension over its center.
Moreover, an earlier result of Lichtman asserts that if a subnormal
subgroup $G$ of $D^*$ contains a nonabelian  nilpotent subgroup then
$G$ has a noncyclic free subgroup \cite{Lic78}. Note that an
affirmative answer to this conjecture implies a great part of the
above mentioned results of Hua, Herstein and Stuth. However, in this
article our aim is to review and give the proofs of some recent
works on the structure of the unit groups of division rings. We
would like to point out that our viewpoint is focused on the
application of the above results in discovering the role of some
special subgroups in the structure of a division ring.

\subsection{Conventions.} We now collect some of the terminology and notation that will be used throughout
the paper:

If $R$ is a ring, we write $Z(R)$ for the center of $R$; $R^*$ for
the group of units of $R$; $\operatorname{M}_n(R)$ for all $n\times
n$ matrices over $R$; and $\operatorname{GL}_n(R)$ for the group of
all invertible $n\times n$ matrices over $R$. For a subgroup $G$ of
$R^*$, the $R$-linear hull of $G$ is always indicated by $R[G]$,
i.e., \[R[G]:=\big \{\sum_{j=1}^mr_jg_j|r_j\in R,\ g_j\in G\ \textrm{and}\
m\in \mathbb{N}\big \}.\] Moreover, if $D$ is a division ring with center
$F$ and $G$ a subgroup of $D^*$ then we denote the division
subring generated by $F$ and $G$ by $F(G)$. It is not hard to see
that $F(G)=F[G]$ whenever all the elements of $G$ are algebraic over
$F$. Let $F$ be a field. We use $F_{\operatorname{alg}}$ for an
algebraic closure of $F$. For an $F$-algebra $A$, we write $[A:F]$
for the dimension of $A$ over $F$ as an $F$-vector space. By an
$F$-central simple algebra we mean an $F$-algebra $A$ where $Z(A)=F$
and $[A:F]<\infty$. Moreover, an $F$-central division algebra is an
$F$-central simple algebra that is also a division ring. Moreover,
sometimes we use the phrase ``division algebra'' for division rings
that are finite dimensional over their centers. Now, let $A$ be an
$F$-central simple algebra and $K$ be a maximal subfield of $A$. It
is known that if $L$ is an algebraic extension of $K$ then
$A\otimes_FL$ is isomorphic to a matrix algebra over $L$. This
immediately implies that $[A:F]$ is a square. We refer to
$\sqrt{[A:F]}$ as the degree of $A$ and denote it by $\deg(A)$.
Also, by Wedderburn's Theorem $A$ is isomorphic to a matrix ring
over an $F$-central division algebra $D$ where $D$ is unique up to
isomorphism. The (Schur) index of $A$ is
$\operatorname{ind}(A)=\deg(D)$. Let $A$ be an $F$-central simple
algebra and $L$ be a splitting field of $A$ which means that there
exists an isomorphism $\phi:A\otimes_FL\rightarrow
\operatorname{M}_n(L)$ (Evidently $n=\deg(A)$). From some classical
results in the theory of central simple algebras, it can be seen
that $\det(\phi(a\otimes1))\in F^*$ for each $a\in A^*$. Moreover,
the value of $\det(\phi(a\otimes1))$ is independent of the choice of
the splitting field $L$ and the isomorphism $\phi$ (cf. \cite[p.
145]{Dra83}). Thus, we can define a map
\begin{align*}
\operatorname{Nrd}_{A/F}:A^*&\longrightarrow F^*\\
 a &\longmapsto \det(\phi(a\otimes1)),
 \end{align*}
  without ambiguity. This map is called the
``reduced norm''. For a field $F$, we use $Br(F)$ for the Brauer
group of $F$. Also, for an $F$-central simple algebra $A$ we denote
its equivalence class in $Br(F)$ by $[A]$. If $K/F$ is a Galois
extension of fields, $\Gal(K/F)$ stands for the Galois group of $K$
over $F$. If $K/F$ is a Galois extension that is cyclic and $\sigma$
is the generator of $\Gal(K/F)$, and if $b\in F^*$ we write
$(K/F,\sigma,b)$ for the cyclic $F$-algebra generated by $K$ and $x$
with relations $xc=\sigma(c)x$ for all $c\in K$ and $x^n=b$. Note
that such an algebra is an $F$-central simple algebra of degree $n$
and also we have $(K/F,\sigma,b)=\oplus_{i=1}^{n-1}Kx^i$. For every
natural number $n$ the notation $\mu_n(F)$ is frequently used for
the group of $n$-th roots of unity in the field $F$. Now, let $G$ be
a group. For each subset $X$ of $G$ we use the symbol $\langle
X\rangle$ for the subgroup generated by $X$. Also, for a subgroup
$H$ of $G$ we write $N_{G}(H)$ and $C_{H}(G)$ for the normalizer and
centralizer of $H$ in $G$, respectively. If $a,b\in G$, we denote the
commutator $a^{-1}b^{-1}ab$ by $[a,b]$. Also, we put $G'$ for the
commutator subgroup of $G$, i.e., $G'=\langle [a,b]|a,b\in
G\rangle$. The commutator subgroups of $G$ of higher orders that are
defined inductively are also denoted by $G^{(n)}=(G^{(n-1)})'$. $G$
is called ``soluble'' if $G^{(n)}=\{1\}$ for some natural number
$n$. The notation $a^b$ also stands for $b^{-1}ab$. If $H$ and $K$
are two subgroups of $G$, then the subgroup $\langle [h,k]|h\in H,\
k\in K\rangle$ is always denoted by $[H,K]$. Let $\zeta_0G=G$ and
for every $i>0$ define the subgroup $\zeta_iG$ by
$\zeta_iG=[G,\zeta_{i-1}G]$. Thus, in this manner we obtain a chain
of subgroups $$\zeta_0G\supseteq \zeta_1G\supseteq\ldots\supseteq\zeta_iG\supseteq\ldots$$
of $G$ which is called the descending central series of $G$. We say
that $G$ is ``nilpotent'' if $\zeta_mG=\{1\}$ for some natural
number $m$.


\section{Soluble-by-finite subgroups of division algebras}
This section deals with the properties of soluble-by-finite
subgroups of division algebras. We will see that the subgroups of
this type, arise naturally in some problems concerning the structure
of $D^*$.

\subsection{Crossed product division algebras}
As is well known, all of the incipient examples of (finite
dimensional) division algebras had a particular common property; all
of them contained a maximal subfield that was Galois over the
center. Such a division algebra is called a \textit{crossed
product}. The existence of a non-crossed product division algebra
remained as an open problem till 1972 when Amitsur presented the
first example of a non-crossed product division algebra in
\cite{Am72}. Also, more examples of non-crossed product division
algebras are produced by some other authors. For a survey on this
subject see \cite[\S 5]{Wad99}. But, perhaps the first important
result connecting the crossed product problem to the multiplicative
group of a division algebra, is the following remarkable theorem of
Albert (cf. Proposition in \cite[p. 286]{Pie82} or \cite[p.
87]{Dra83}):
\begin{thm}[Albert's Criterion]\label{thm1}
Let $D$ be an $F$-central division algebra of prime degree $p$. If
$D^*/F^*$ contains a non-trivial element of order $p$, then $D$ is a
cyclic division algebra, i.e., $D$ has a maximal subfield $K$, such
that $K/F$ is a cyclic extension.
\end{thm}
Motivated by the above result, in \cite{MaTi03} the authors
investigate the relation between the structure of the unit group of
a division algebra $D$ and the cyclicity conditions when the index
of $D$ is a prime number. The next theorem, is the main result of
\cite{MaTi03}.
\begin{thm}\label{thm2}
Let $D$ be an $F$-central division algebra of prime degree $p$. Then
the following conditions are equivalent:
\begin{enumerate}
  \item $D$ is a cyclic division algebra;
  \item $D^*$ contains a non-abelian soluble subgroup;
  \item $D^*/F^*$ contains a non-trivial torsion subgroup;
  \item $D^*/F^*$ contains an element of order $p$.
\end{enumerate}
\end{thm}

To prove Theorem \ref{thm2}, we need the following lemma.
\begin{lem}\label{lem}
Let $D$ be a finite dimensional division algebra over its center
$F$. If $G$ is a soluble subgroup of $D^*$, then $G$ contains an
abelian normal subgroup $A$ of finite index.
\end{lem}
\begin{proof}
We know that $D\otimes_FF_{\operatorname{alg}}\cong
\operatorname{M}_n(F_{\operatorname{alg}})$ where $n=\deg(D)$. Thus,
we may view $D^*$ as a subgroup of
$\operatorname{GL}_n(F_{\operatorname{alg}})$. Now, by
Lie-Kochlin-Mal'cev Theorem \cite[p. 434]{Rob82}, $G$ contains a
normal subgroup $A$ such that $A'$ is unipotent, i.e., for each $a\in A'$ the element
$1-a$ is nilpotent . But, since the only
unipotent subgroup of $D^*$ is $\set{1}$, we must have $A'=\set{1}$.
This means that $A$ is abelian and the proof is complete.
\end{proof}
Lemma \ref{lem} is our major tool in the subsequent works. Now, we
are going to present a proof for Theorem \ref{thm2}.
\begin{proof}[Proof of Theorem 2]
$(1)\Rightarrow(2)$. Let $K$ be the maximal subfield of $D$ where
$K/F$ is a cyclic extension. Thus, by the Skolem-Noether Theorem, there
is an element $z\in D^*$ such that $D=\oplus_{j=1}^{p-1}Kz^j$ and
for each $k\in K$, $kz=z\sigma(k)$, where $\sigma$ is a generator of
$\Gal(K/F)$ . Put $S=K^*\langle z\rangle$. $S$ is non-abelian,
because $\sigma$ is non-trivial. On the other hand, $z^p\in K^*$ as
$\sigma^p=1_K$. Therefore, $S/K^*$ is a cyclic group and hence $S$
is non-abelian soluble.

$(2)\Rightarrow(3)$. Let $G$ be a non-abelian soluble subgroup of
$D^*$. By Lemma \ref{lem}, we can choose a normal abelian subgroup
$A$ of $G$ with $|G:A|<\infty$. Take $A$ maximal and put $H=C_G(A)$.
If $H\neq A$, then $H$ is non-abelian, because $A$ is maximal. Thus,
$F[H]=D$ and so $A\subseteq F^*$. Therefore, $F^*H/F^*$ is a
non-trivial torsion subgroup of $D^*/F^*$. Now, let $H=A$. Thus, $A$
is non-central in $G$ and hence $K=F[A]$ is a maximal subfield of
$D$. If $g\in G\setminus A$, then $g$ induces a non-trivial
$F$-automorphism $\sigma$ of $K$ by $k\mapsto gkg^{-1}$. Since
$\sigma$ has a finite order, then $g^m\in K^*$ for some $m\neq1$.
But, $\sigma(g^m)=gg^mg^{-1}=g^m$. So $g^m\in F$, because $F$ is the
fixed field of $\sigma$. Now, $F^*\langle g\rangle/F^*$ is a
non-trivial torsion subgroup of $G$.

$(3)\Rightarrow(4)$. Let $T$ be a subgroup of $D^*$ such that
$T/F^*$ is torsion. Denote by $D^1$ the kernel of the reduced norm
map $\operatorname{Nrd}_D:D^*\rightarrow F^*$. Consider the two
following cases:

\textbf{Case 1}. $T\cap D^1\subseteq F^*$. If we choose an element
$u\in T\setminus F^*$, then $u^p\operatorname{Nrd}_D(u)^{-1}\in
T\cap D^1\subseteq F^*$ and hence $u^p\in F^*$.

\textbf{Case 2}. $T\cap D^1\nsubseteq F^*$. If this is the case,
then there is an element $1\neq u\in T\cap D^1$ such that $u\not\in
F^*$. But, since $u\in T$, we have $u^m\in F^*$ for some $m$ as
$T/F^*$ is torsion. Therefore, $u^{pm}=\operatorname{Nrd}_D(u)^m=1$.
Since $u\neq1$, we conclude that $K=F(u)$ is a cyclic extension of
$F$ with $[D:K]=p$. Now, by the Skolem-Noether Theorem, there is a
$z\in D^*$ such that the inner automorphism $d\mapsto zdz^{-1}$
restricts to a non-trivial $F$-automorphism of $K$. Now, a similar
argument as in the proof of $(2)\Rightarrow(3)$ implies that $z^p\in
F^*$.

$(4)\Rightarrow(1)$. This follows from Albert's Criterion.
\end{proof}
Let $F$ be a field containing a primitive $n$-th root of unity
$\omega_n$ for some natural number $n$. Recall that a symbol algebra
is an algebra of the form $A=\oplus_{i,j=0}^{n-1}Fu^iv^j$ where
$u,v$ are symbols with the requirements $u^n,v^n\in F$ and
$uv=\omega_nvu$ (cf. \cite[\S 11]{Dra83}). When $n$ is prime and
$A$ is not a split algebra, one can easily see that $A$ is a cyclic
division algebra. Thus, over a given field $F$, the class of symbol
algebras of prime degree $p$ is a subclass of cyclic division
algebras of degree $p$. So, it can be of interest to realize when a
division algebra of prime degree is a symbol algebra.
\begin{thm}[\cite{MaMo12}]\label{thm3}
Let $D$ be an $F$-central division algebra of prime degree $p$. Then
$D$ is a symbol algebra if and only if $D^*$ contains a non-abelian
nilpotent subgroup.
\end{thm}
\begin{proof}
If $D$ is a symbol algebra, then there are elements $u,v\in D^*$
such that $u^p,v^p\in F^*$, $uv=\omega_pvu$ and
$D=\oplus_{i,j=1}^{p-1}Fu^iv^j$. Put $G=F^*\langle u,v\rangle$.
Since $G/F^*$ is abelian, we conclude that $G'\subseteq F^*\subseteq
Z(G)$. Hence, $G$ is nilpotent. Conversely, let $G$ be a non-trivial
nilpotent subgroup of $D^*$. By Theorem \ref{thm2}, we conclude that
$D$ is a cyclic division algebra. Thus, if we show that $F^*$
contains a primitive $p$-th root of unity, then we are done. To
prove this, suppose that the length of the lower central series of
$G$ is equal to $t$, i.e.,
$$\zeta_0G\supset\zeta_1G\supset ...\supset\zeta_{t-1}G\supset\zeta_tG=1,$$
where $\zeta_0G=G$  and $\zeta_iG=[G,\zeta_{i-1}G]$ for $1\leq i\leq
t$. Therefore, $1\neq\zeta_{t-1}G\subseteq G'\cap Z(G)\subseteq G'\cap F^*$. Thus, for
each $1\neq a\in \zeta_{t-1}G$ we must have
$a^p=\operatorname{Nrd}_D(a)=1$, as desired.
\end{proof}
Now, let us look more closely at the group theoretic structure of a
crossed product division algebra. Suppose $D$ is a crossed product
division algebra with maximal Galois subfield $K$.  By
Skolem-Noether Theorem, for every $\sigma\in \Gamma=\Gal(K/F)$ there
is an $e_{\sigma}\in D^*$ such that
\begin{enumerate}
  \item [(i)] $A=\oplus_{\sigma\in\Gamma}Ke_{\sigma}$ and $\sigma(k)=e_{\sigma}ke_{\sigma}^{-1}$
  ($k\in K$ and $\sigma\in\Gamma$);
  \item [(ii)] There is a 2-cocycle $f\in Z^2(\Gamma,K^*)$ such that
  $e_{\sigma}e_{\tau}=f(\sigma,\tau)e_{\sigma\tau}$ ($\sigma,\tau\in \Gamma$).
\end{enumerate}
Now, put $N=N_{D^*}(K^*)$. From (i), we have $e_{\sigma}\in N$ for
each $\sigma\in \Gamma$. Thus $F[N]=D$, i.e., $N$ is an
\textit{absolutely irreducible}  subgroup of $D^*$. On the other
hand, every element of $N$ induces an $F$-automorphism of $K$ by
conjugation. Therefore, $N/C_{D^*}(K^*)\cong \Gamma$. But,
$C_{D^*}(K)=K$ as $K$ is a maximal subfield. This implies that
$N/K^*\cong\Gamma$. Consequently, $N$ is an absolutely irreducible
abelian-by-finite (equivalently soluble-by-finite) subgroup of
$D^*$. Now, the following question arises naturally:
\begin{que}\label{que}
Let $D$ be a division algebra such that its unit group has an
absolutely irreducible soluble-by-finite subgroup $G$. Is $D$ a
crossed product?
\end{que}
The first significant result concerning Question \ref{que}, was 
obtained in \cite{Ma01} where the author showed that if $G$ is a soluble
maximal subgroup (this means that $G$ is a maximal subgroup of $D^*$
that is soluble), then the question has a positive answer.
Afterwards, this question became the subject of a series of papers
including \cite{EbKiMa05a}, \cite{EbKiMa05b}, \cite{KiMa05} and
\cite{KeMa07}. In \cite{EbKiMa05a}, it was proved that if $G$ is
soluble, then $D$ is a quasi-crossed product division algebra in the
sense that it contains a tower of subfields $F=Z(D)\subsetneq
K\subseteq L$ such that $K/F$ is Galois, $L$ is a maximal subfield
and $L/K$ is an abelian Galois extension. Also, the results of
\cite{EbKiMa05b} and \cite{KiMa05} assert that $D$ is a crossed
product division algebra if one of the following conditions holds:
\begin{enumerate}
  \item [(i)] $G$ is either supersoluble or nilpotent;
  \item [(ii)] $G$ has no subgroups isomorphic to $\operatorname{SL}(2,5)$ and
  $D$ has a prime power degree.
\end{enumerate}
However, a natural question related to (ii) is what happens when $G$
has a subgroup isomorphic to $\operatorname{SL}(2,5)$? Answering
this question requires suitable information concerning the structure
of soluble-by-finite subgroups of division algebras. The most
important systematic investigation of the structure of
soluble-by-finite subgroups in the general case was fulfilled in an
interesting paper of Shirvani (\cite{Sh05}), where he has proved
that:
\begin{thm}\label{thm}
Let $D$ be a division algebra of degree $m$. Then:
\begin{enumerate}
  \item Every soluble-by-finite subgroup of $D^*$ is contained in a maximal soluble-by-finite
  subgroup of $D^*$;
  \item If $G$ is a soluble-by-finite subgroup of $D^*$, then $G$ contains an abelian
  normal subgroup $A$ with $|G:A|=cmt$, where $c=1,6,12$ or $30$, and $t$ is a proper divisor
  of $m$. Moreover, if $\operatorname{char}(D)>0$, or if $m$ is odd, or if $G$ is torsion-free,
  then $c=1$.
\end{enumerate}
Consequently, every soluble-by-finite subgroup of $D^*$ has a normal
abelian subgroup of index dividing $60\deg(D)^2$.
\end{thm}
Using this information he showed that if $\deg(D)$ is such that
every group of order $\deg(D)^2$ is nilpotent, or if $G$ is either
metabelian or torsion-free, then $D$ is a crossed product division
algebra. The proofs of Shirvani's results are very technical and
include heavy computations. But, in \cite{We06}, Wehrfritz gave an
altenative proof of part (2) of Theorem \ref{thm} based on the
contents of his earlier results in \cite{We90a} and\cite{We90c}.
Thereby, he also slightly improved Shirvani's bound for $|G:A|$. In
fact, Wehrfritz results, show that one can replace the bound
$60\deg(D)^2$ by $30\deg(D)^2$. In the sequel, we try to simplify
Shirvani's idea to achieve the major part of the above mentioned
results.
\begin{lem}\label{lem4}
Let $G$ be an irreducible subgroup of $D^*$. If $A$ is a normal
abelian subgroup of $G$, then:
\begin{enumerate}
  \item $[D:F[C_G(A)]]=|G:C_G(A)|$;
  \item $C_D(F[A])=F[C_G(A)]$;
  \item $F[A]/F$ is Galois with $\Gal(F[A]/F)\cong
G/C_G(A)$.
\end{enumerate}
In particular, if $C_G(A)=A$ then $F[A]$ is a Galois maximal
subfield of $D$ and hence $D$ is a crossed product division algebra.
\end{lem}
\begin{proof}
$(1)$ Put $H=C_G(A)$. Let $T$ be a left transversal of $H$ in $G$.
Clearly, $T$ is a set of generators of $D$ over $F[H]$. Thus, it is
enough to show that the elements of $T$ are linearly independent
over $F[H]$. Let $\sum_{j=1}^ka_jt_j=0$ be a relation of minimal
length where $a_j\in F[H]$ and $t_j\in T$. We may assume that
$t_1=1$. Now, for every $a\in A$ we have
$0=a^{-1}(\sum_{j=1}^ka_jt_j)a=\sum_{j=1}^ka_jt_j^a=
\sum_{j=1}^ka_ja^{-1}[t_j^{-1},a^{-1}]at_j=\sum_{j=1}^ka_j[t_j^{-1},a^{-1}]t_j$.
If we subtract this relation from $\sum_{j=1}^ka_jt_j=0$, we obtain
$\sum_{j=2}^k(a_j-a_j[t_j^{-1},a^{-1}])t_j=0$. Since this is a
shorter relation than $\sum_{j=1}^ka_jt_j=0$, we must have
$a_j-a_j[t_j^{-1},a^{-1}]=0$ for all $2\leq j\leq k$. Thus,
$[t_j^{-1},a^{-1}]=1$ and hence $t_j\in C_G(A)=H$ for all $j$, which
is a contradiction.

$(2)$ It is clear that $F[C_G(A)]\subseteq C_D(F[A])$. Conversely,
Let $\set{g_1,\ldots,g_n}\subseteq G$ be a $F$-basis for $D$.
Suppose $\sum u_{j_k}g_{j_k}\in C_D(F[A])$ with $0\neq u_{j_k}\in
F$. For each $a\in A$, we have $\sum au_{j_k}g_{j_k}=\sum
u_{j_k}g_{j_k}a= \sum au_{j_k}[g_{j_k}^{-1},a^{-1}]g_{j_k}$.
Therefore, $au_{j_k}=au_{j_k}[g_{j_k}^{-1},a^{-1}]$ for every $j_k$.
But, $u_{j_k}\neq0$ and hence $[g_{j_k}^{-1},a^{-1}]=1$. Thus every
$g_{j_k}$ lies in $C_G(A)$. This implies that $C_D(F[A])\subseteq
F[C_G(A)]$ and the result follows.

$(3)$ Consider the map
$$\sigma:G/H\rightarrow \Gal(F[A]/F)\ ,\ xH\mapsto (u\mapsto x^{-1}ux).$$
Clearly, $\sigma$ is injective. On the other hand, by the equality
$C_D(F[A])=F[C_G(A)]$ and the Centralizer Theorem (see \cite[p.
42]{Dra83}), we have $[F[A]:F]=[D:F[H]]=|G/H|$. So, $F[A]/F$ is
Galois and $\sigma$ is an isomorphism.

The last statement follows from $(3)$.
\end{proof}
\begin{cor}\label{cormaini}
Let $D$ be an $F$-central division algebra and $G$ be an absolutely
irreducible subgroup of $D^*$. If one of the following conditions
holds, then $D$ is a crossed product division algebra.
\begin{enumerate}
  \item $G$ is abelian-by-nilpotent;
  \item $G$ is metabelin.
\end{enumerate}
\end{cor}
\begin{proof}
(1) Let $A$ be a normal abelian subgroup of $G$ such that $G/A$ is
nilpotent. Take a maximal abelian normal subgroup $B\supseteq A$.
Since $G/B$ is nilpotent, $C_G(B)/B\cap Z(G/B)\neq 1$ unless
$C_G(B)=B$. However, if $C_G(B)\neq B$, then we can choose an
element $x\in C_G(B)\setminus B$ such that $\langle
x,B\rangle/B\subseteq Z(G/B)$. Therefore, $\langle
x,B\rangle\supsetneq B$ is a normal abelian subgroup of $G$ which is
a contradiction, because $B$ is maximal. Thus, $C_G(B)=B$ and the
result follows by Lemma \ref{lem4}.

(2) If $G$ is metabelian, then $G'$ and $G/G'$ are abelian. Thus,
$G$ is abelian-by-nilpotent and the result follows from (1).
\end{proof}
\begin{lem}\label{lem5}
Let $D$ be a division algebra and $K$ be a subfield of $D$. Suppose
that $G$ is a subgroup of $D^*$ centralizing $K$. If
$G'Z(G)\subseteq K^*$, then $[K[G]:K]=|G:G\cap K^*|$.
\end{lem}
\begin{proof}
Put $A=G\cap K^*$. Let $T$ be a left transversal of $A$ in $G$.
Clearly, $T$ generates $K[G]$ as an $K$-algebra. Thus, if we show
that the elements of $T$ are linearly independent over $K$, then we
are done. Let $\sum_{j=1}^ka_jt_j=0$ be a relation of minimal length
where $0\neq a_j\in K$ and $t_j\in T$. We may assume that $t_1=1$.
Now, since $G'\subseteq K\subseteq C_D(G)$, for every $g\in G$ we
have
$0=g^{-1}(\sum_{j=1}^ka_jt_j)g=\sum_{j=1}^ka_jt_j^g=\sum_{j=1}^ka_jg^{-1}[t_j^{-1},g^{-1}]gt_j=
\sum_{j=1}^ka_j[t_j^{-1},g^{-1}]t_j$. Subtraction from
$\sum_{j=1}^ka_jt_j=0$, yields
$\sum_{j=2}^k(a_j-a_j[t_j^{-1},g^{-1}])t_j=0$. Since this is a
shorter relation than $\sum_{j=1}^ka_jt_j=0$ (note that since
$G'\subseteq K^*$, each $a_j[t_j^{-1},g^{-1}]$ lies in $K$), we must
have $a_j-a_j[t_j^{-1},g^{-1}]=0$ for all $2\leq j\leq k$.
Therefore, $[t_j^{-1},g^{-1}]=1$. So, $t_j\in Z(G)\subseteq K^*$ and
hence $t_j\in A$ for all $2\leq j\leq k$, which is a contradiction.
\end{proof}
\begin{prop}\label{prop6}
Suppose that $D$ is an $F$-central division algebra and that $X$ and
$Y$ are subgroups of $D^*$. If $Z(F[X])=F$ and $Y\subseteq C_D(X)$,
then $F[XY]=F[X]\otimes_FF[Y]$.
\end{prop}
\begin{proof}
Clearly $F[XY]$ is an $F$-central division algebra. Thus, we have
$F[XY]=F[X]\otimes_FC_{F[XY]}(F[X])$. But, $F[Y]\subseteq
C_{F[XY]}(F[X])$, because $Y\subseteq C_D(X)$. Hence
$F[X]\otimes_FF[Y]$ is a $F$-subalgebra of $F[XY]$. Now, let
$\set{x_1\ldots x_r}$ be a $F$-basis for $F[X]$ and $\set{y_1\ldots
y_s}$ be a $F[X]$-basis for $F[XY]$. One may easily check that
$\set{x_iy_j|1\leq i\leq r,1\leq j\leq s}$ is a $F$-basis for
$F[XY]$. So, $\dim_FF[XY]=rs$. But, $y_j$'s are linearly independent
over $F$. Therefore, $\dim_FF[Y]\geq s$. Thus, we must have
$\dim_F(F[X]\otimes_FF[Y])=\dim_FF[X]\dim_FF[Y] \geq
rs=\dim_FF[XY]$. This implies that $F[XY]=F[X]\otimes_FF[Y]$.
\end{proof}
The following result is a special case of \cite[Th. 2.6]{Sh05}.
\begin{thm}\label{thm7}
Let $D$ be an $F$-central division algebra and $G$ be an irreducible
subgroup of $D^*$. Suppose that $N$ is a normal subgroup of $G$ and
that $Z(F[N])=F$. Then there exist subgroups $G_1$ and $G_2$ of
$D^*$ such that:
\begin{enumerate}
  \item $F^*\subseteq G_i$, $N\unlhd G_1$, $F[N]=F[G_1]$ and $G\subseteq G_1G_2$;
  \item $D=F[G_1]\otimes_FF[G_2]$;
  \item For every $i$, there exists an epimorphism $\phi_i:G\rightarrow G_i/F^*$.
  Moreover, $N\subseteq \ker\phi_2$.
\end{enumerate}
\end{thm}
\begin{proof}
$(1)$ Since $N\unlhd G$, every $g\in G$ induces an $F$-automorphism
$u\mapsto g^{-1}ug$ of $F[N]$. Thus, by the Skolem-Noether Theorem,
there is an element $f_1(g)\in F[N]$ that has the same effect as $g$
on $F[N]$ (and of course on $N$). So, for every $u\in F[N]$, we have
$g^{-1}ug=f_1(g)^{-1}uf_1(g)$ and hence $gf_1(g)^{-1}\in C_D(F[N])$.
Put $f_2(g)=gf_1(g)^{-1}$. But, it is not hard to check that
$f_1(g)$ and $f_2(g)$ are unique modulo $F^*$. Thus, for all $g,h\in
G$ and $i=1,2$, $f_i(gh)\in f_i(g)f_i(h)F^*$. Now, set
$G_i=F^*\langle f_i(g):\ g\in G\rangle$. If $g\in G$, from the
equality $g=f_1(g)f_2(g)$ we conclude that $G\subseteq G_1G_2$.
Also, if $u\in N$, then $u=f_1(u)k$ for some $k\in F^*$. Therefore
$N\leq G_1$. Moreover, since $f_1(g)$ has the same action as $g$ on
$N$, we have $N\unlhd G_1$. Finally, because $N\subseteq
G_1\subseteq F[N]$, we have $F[N]\subseteq F[G_1]\subseteq F[N]$.
Hence $F[N]=F[G_1]$.

$(2)$ By construction we have $G_2\subseteq C_D(G_1)$. Now, from
Proposition \ref{prop6}, we conclude that
$F[G]=F[G_1G_2]=F[G_1]\otimes_FF[G_2]$.

$(3)$ For every $i=1,2$, the map $\phi_i:G\rightarrow G_i/F^*,
g\mapsto f_i(g)F^*$ is the required epimorphism. Furthermore, since
$f_2(n)=nf_1(n)^{-1}\in F^*$ for all $n\in N$, we have $N\subseteq
\ker \phi_2$.
\end{proof}
Recall that by the work of Amitsur \cite[Th. 2]{Am55}, if $G$ is  a
finite subgroup of the unit group of a division ring, then $G$ is
one of the following type:
\begin{enumerate}
  \item [(i)] All Sylow subgroups of $G$ are cyclic;
  \item [(ii)] All of odd Sylow subgroups of $G$ are cyclic and the Sylow
  $2$-subgroup of $G$ is a generalized quaternion group $Q_{2^{n+1}}$ of order $2^{n+1}$
  for some natural number $n\geq2$.
\end{enumerate}
Here, we must point out that the generalized quaternion group of
order $2^{n+1}$ is a group generated by two element $x,y$ satisfying
$y^{2^{n-1}}=x^2, x^4=1$ and $xyx^{-1}=y^{-1}$. The above
classification is all the material about finite subgroups of a
division ring that we need in our approach.
\begin{lem}\label{lem8}
Let $D$ be a division algebra and $N$ be a finite subgroup of $D^*$.
If $N$ is nilpotent of class at most $2$, then $N$ is either cyclic
or $N\cong Q_8\times C$ where $C$ is a cyclic group.
\end{lem}
\begin{proof}
Since $N$ is nilpotent, it is isomorphic to the direct product of
its all Sylow subgroups. Here, two cases may occur:

\textbf{Case 1}. All Sylow subgroups of $N$ are cyclic. Clearly, in
this case $N$ is cyclic.

\textbf{Case 2}. The Sylow $2$-subgroup of $N$ is $Q_{2^{n+1}}$ and
other Sylow subgroups are cyclic. If this is the case, then $N\cong
Q_{2^{n+1}}\times C$ where $C$ is a cyclic group. Thus, we must
prove that $n=2$. Suppose that $x,y$ are the generators of
$Q_{2^{n+1}}$. Since $Q_{2^{n+1}}$ is nilpotent of class 2, we have
$Q_{2^{n+1}}'\subseteq Z(Q_{2^{n+1}})$. Therefore,
$y^{2}=[y^{-1},x^{-1}]\in Z(Q_{2^{n+1}})$ and so $[x^{-1},y^2]=1$.
But, $[x^{-1},y^2]=y^4$. Thus, we must have $y^4=1$. This implies
that $n=2$, as desired.
\end{proof}
To proceed, we need also the following theorem which is a direct
consequence of  \cite[Th. 2.1.11]{ShWe86}.
\begin{thm}\label{th}
If $D$ is an $F$-central division algebra containing a subgroup
$S\cong\operatorname{SL}(2,5)$, then $Z(F[S])=F$ if $\sqrt{5}\in F$
and $Z(F[S])=F(\sqrt{5})$ if $\sqrt{5}\not\in F$. Therefore,
$[Z(F[S]):F]\leq2$.
\end{thm}
Now, we can prove the following important result.
\begin{thm}\label{thm14}
Let $D$ be an $F$-central division algebra and $G$ be a
soluble-by-finite absolutely irreducible subgroup of $D^*$. Then:
\begin{enumerate}
  \item If $G$ contains a normal subgroup
  $S\cong \operatorname{SL}(2,5)$, then
  there are soluble-by-finite subgroups $G_1,G_2$ of $D^*$ such that $F[G_1]=F[Q]$
  (where $Q\cong Q_8$), $D=F[Q]\otimes_FF[G_2]$
  and $G_2$ has no normal subgroup isomorphic to $Q_8$ or $\operatorname{SL}(2,5)$.
  Moreover, if $Z(F[S])=F$ then $S\unlhd G_1$ and if
  $[Z(F[S]):F]=2$, then $Q\unlhd G_1$.
  \item If $G$ has a normal subgroup $Q\cong Q_8$, then
  there are soluble-by-finite subgroups $G_1,G_2$ of $D^*$ such that $F[G_1]=F[Q]$,
  $D=F[Q]\otimes_FF[G_2]$
  and $G_2$ has no normal subgroup isomorphic to $Q_8$ or $\operatorname{SL}(2,5)$.
  Moreover, $Q\unlhd G_1$.
  \item If $G$ has no normal subgroup isomorphic to $Q_8$ or $\operatorname{SL}(2,5)$, then
  it contains a maximal abelian normal subgroup $A$ of index dividing $\deg(D)^2$. Moreover,
  $F[A]/F$ is Galois with $\Gal(F[A]/F)\cong G/C_G(A)$, $C_G(A)'\subseteq A$
  and so $C_G(A)$ is nilpotent of class at most $2$.
\end{enumerate}
\end{thm}
\begin{proof}
(1) First suppose that $G$ has a normal subgroup
$S\cong\operatorname{SL}(2,5)$. Put $L=Z(F[S])$. Now, by Theorem
\ref{th}, two cases may occur:

\textbf{Case 1}. $L=F$. So $F[S]=F[Q]$ where $Q$ is the
$Q_8$-subgroup of $S$. Therefore, $\deg(F[S])=2$. Now, by Theorem
\ref{thm7}, there exist subgroups $G_1$ and $G_2$ of $D^*$ such that
$D=F[Q]\otimes_FF[G_2]$. Also, $G\subseteq G_1G_2$ and $S\unlhd
G_1$. If $G_2$ contains a normal subgroup isomorphic to $Q_8$ or
$\operatorname{SL}(2,5)$, again by Theorem \ref{thm7}, we can
decompose $F[G_2]$ as $F[G_2]=F[Q]\otimes_FF[G_3]$ for some subgroup
$G_3$. But, this decomposition implies that
$\operatorname{M}_4(F)=F[Q]\otimes_FF[Q]$ is a subalgebra of $D$
which is a contradiction. So, we conclude that $G_2$ contains no
normal subgroup isomorphic to $Q_8$ or $\operatorname{SL}(2,5)$.
Moreover, since $G_1,G_2$ are homomorphic images of $G$, they are
soluble-by-finite.

\textbf{Case 2}. $[L:F]=2$. Since $G$ normalizes $F[S]$, it also
normalizes $L^*$. Therefore, $M=L^*G$ is a soluble-by-finite
subgroup of $D^*$ and of course it is absolutely irreducible.
Moreover, $S\lhd M$. Let $H=C_M(L)$. It is not hard to see that
$SC_M(S)\subseteq H$. But, $\overline{M}=M/C_M(S)$ is isomorphic to
a subgroup of $\operatorname{Aut}(S)\cong\operatorname{Sym}(5)$. On
the other hand, $SC_M(S)/C_M(S)\cong S/Z(S)$ and hence
$|SC_M(S):C_M(S)|=60$. Moreover, by Lemma \ref{lem4}, we have
$|M:H|=2$. Thus, $|\overline{M}|=|M:H||H:C_M(S)|\geq
|M:H||SC_M(S):C_M(S)|=120$. Therefore, $\overline{M}\cong
\operatorname{Sym}(5)$ and $H=SC_M(S)$. Let $Q$ be the
$Q_8$-subgroup of $S$ and let $T\supseteq C_M(S)$ be such that
$\overline{T}$ is a Sylow $2$-subgroup of $\overline{M}$ containing
$\overline{Q}$. Thus $QC_M(S)\subseteq T$ and so $F[S]=L[Q]\subseteq
F[T]$. On the other hand $T$ properly contains $C_M(S)$ and
$|M:SC_M(S)|=2$, Thus, we conclude that $M=ST$. Then
$D=F[M]=F[ST]=F[S][T]=F[T]$. Also, since
$|\overline{T}:\overline{Q}|=2$, we conclude that that $Q\lhd T$.
Now, by Theorem \ref{thm7} we can find a decomposition of $D$ as
$D=F[G_1]\otimes_F[G_2]$ where $M\subseteq G_1G_2$, $F[G_1]=F[Q]$
and $Q\lhd G_1$. Moreover, a similar argument as in Case 1, shows
that $G_2$ has no normal subgroup isomorphic to $Q_8$ or
$\operatorname{SL}(2,5)$.

(2) This follows from Theorem \ref{thm7} and a similar argument as
in Case 1 of (1).

(3) By Lemma \ref{lem}, $G$ contains a normal abelian subgroup $A$
of finite index. Take $A$ maximal and put $H=C_G(A)$. If $H=A$, then
$F[A]$ is a maximal subfield and hence by Lemma \ref{lem4}, we have
$|G:A|=[D:F[A]]=\deg(D)$ and the result follows. Thus, we assume
that $H\neq A$. Clearly, $Z(H)=A$. Therefore, $H$ is center-by-finte
and so $H'$ is finite. But, by our assumption $G$ has no normal
subgroup isomorphic to $\operatorname{SL}(2,5)$. This forces that
$H'$ is soluble. Now, we claim that $H'$ is abelian. Suppose that
$H'$ is non-abelian. Let $H^j$ denote the $j$-th term of the derived
series of $H$ and $t$ be the smallest integer such that $H^t=1$. So,
$H^{t-2}$ is non-abelian metabelian. Also, $H^{t-1}$ is abelian and
normal in $G$. Since $H^{t-1}\subseteq C_G(A)$, $AH^{t-1}$ is a
normal abelian subgroup of $G$ and hence $AH^{t-1}=A$, because $A$
is maximal. Therefore, $H^{t-1}$ is central in $H$ and so $H^{t-2}$
is nilpotent of class $2$. Now, by Lemma \ref{lem8}, $H^{t-2}$ is
either cyclic or $Q_8\times C$ where $C$ is cyclic. But, $H^{t-2}$
is non-abelian, so it can not be cyclic. Moreover, if
$H^{t-1}=Q_8\times C$, then $Q_8\lhd G$ which is impossible. This
contradiction shows that $H'$ is abelian and hence $H'\subseteq A$,
because $A$ is maximal. Now, put $K=F[A]$. By Lemma \ref{lem5} and
Lemma \ref{lem4}, we have $|H:K^*\cap H|=[K[H]:K]=\deg(F[H])^2$. On
the other hand, $A\subseteq K^*\cap H \lhd G$. Since $A$ is maximal,
we conclude that $A=K^*\cap H$. Thus, $|H:A|=\deg(F[H])^2$. Now,
Lemma \ref{lem4} implies that
$|G:A|=|G:H||H:A|=[F[G]:F[H]]\deg(F[H])^2$. Finally, the Double
Centralizer Theorem gives us $[F[G]:F[H]]\deg(F[H])=\deg(F[G])$ and
consequently $|G:A|$ divides $\deg(D)^2$.
\end{proof}
\begin{cor}\label{cor15}
Let $D$ be an $F$-central division algebra and $G$ be a
soluble-by-finite subgroup of $D^*$. If $G$ has no normal subgroup
isomorphic to $\operatorname{SL}(2,5)$ (in particular, if $G$ is
soluble), then $G$ contains a normal abelian subgroup of index
dividing $6\deg(D)^2$.
\end{cor}
\begin{proof}
Without loss of generality, we may assume that $G$ is absolutely
irreducible. If $G$ has no normal $Q_8$-subgroup, the result follows
immediately by (3) of Theorem \ref{thm14}. Now, suppose that $G$ has
a normal subgroup $Q\cong Q_8$. Using (2) of Theorem \ref{thm14}, we
can find a decomposition $D=F[G_1]\otimes_FF[G_2]$ for some
soluble-by-finite subgroups $G_1,G_2$ where $G\subseteq G_1G_2$,
$Q\lhd G_1$ and $G_2$ has no normal $\operatorname{SL}(2,5)$ or
$Q_8$-subgroup. Now, by the former case, $G_2$ contains a normal
abelian subgroup $A$ of index dividing $\deg(D)^2/4$ (note that
$\deg(F[G_2])=2$). Also, $G_1\subseteq N=N_{F[Q]^*}(Q)$. Since $Q$
is an absolutely irreducible subgroup of $F[Q]$, we have
$C_N(Q)=F^*$. Therefore, $N/F^*$ is isomorphic to a subgroup of
$\operatorname{Aut}(Q)\cong \operatorname{Sym}(4)$ and so
$|G_1:F^*|$ divides $24$ (recall that $F^*\subseteq G_1$). Hence,
$|G_1G_2:F^*A|$ divides $24\deg(D)^2/4=6\deg(D)^2$. So, $|G:G\cap
F^*A|$ divides $6\deg(D)^2$, as desired.
\end{proof}

\begin{thm}
Let $D$ be an $F$-central division algebra such that $D^*$ contains
an absolutely irreducible soluble-by-finite subgroup $G$. Then:
\begin{enumerate}
  \item $D$ is a quasi-crossed product division algebra;
  \item If $\deg(D)$ is such that every finite group of order $\deg(D)^2$ is nilpotent,
  or if $G$ has no element of order dividing $\deg(D)$, then $D$ is a crossed
  product division algebra
  over a maximal subfield.
\end{enumerate}
\end{thm}
\begin{proof}
$(1)$ First, suppose that $G$ has no subgroup isomorphic to $Q_8$ or
$\operatorname{SL(2,5)}$. By Theorem \ref{thm14}, $G$ contains an
abelian normal subgroup $A$ of finite index such that $F[A]/F$ is
Galois. Also, $H=C_G(A)$ is nilpotent. Thus, if $B\supseteq A$ is a
maximal normal abelian subgroup of $C_G(A)$ of finite index, Then by
Corollary \ref{cormaini}, we have $C_H(B)=B$ and $F[B]/F[A]$ is Galois
with $\Gal(F[B]/F[A])\cong H/C_H(B)$ (note that by Lemma \ref{lem4},
$Z(F[H])=F[A]$). Also, $H/C_H(B)$ is abelian, because $H'\subseteq
A\subseteq B$. Furthermore, by the Double Centralizer Theorem $F[B]$
is a maximal subfield of $D$. So, we have a tower of subfields
$F\subseteq F[A]\subseteq F[B]$ such that $F[B]$ is a maximal
subfield of $D$, $F[A]/F$ is Galois and $F[B]/F[A]$ is abelian
Galois, as desired.

Now, if $G$ contains a normal subgroup isomorphic to $Q_8$ or
$\operatorname{SL}(2,5)$, Then by Theorem \ref{thm14},
$D=F[Q]\otimes_FF[G_2]$ where $G_2$ is soluble-by-finite that has no
subgroups isomorphic to $Q_8$ or $\operatorname{SL}(2,5)$. But, as
above, there exists a tower of subfields $F\subseteq K\subseteq L$
of $F[G_2]$ such that $L$ is a maximal subfield of $F[G_2]$, $K/F$
is Galois and $L/K$ is abelian Galois. Now, one can easily check
that $F\subseteq K\otimes_FF(\sqrt{-1})\subseteq
L\otimes_FF(\sqrt{-1})$ is the required tower of subfields for $D$.

$(2)$ Suppose that every finite group of order $\deg(D)^2$ is
nilpotent. Let $G$ contain no subgroups isomorphic to $Q_8$ or
$\operatorname{SL}(2,5)$ and let $A$ be a maximal normal abelian
subgroup of $G$ such that $|G:A|$ divides $\deg(D)^2$ (Theorem
\ref{thm14}). Then, $G$ is abelian-by-nilpotent and hence by
Corollary \ref{cormaini}, $D$ is a crossed product division algebra. Now,
suppose that $G$ contains a copy of $Q_8$ or
$\operatorname{SL}(2,5)$ as a normal subgroup. By Theorem
\ref{thm14}, $D$ is decomposed as $D=F[Q]\otimes_FF[G_2]$ where
$G_2$ has no normal subgroups $Q_8$ or
$\operatorname{SL}(2,5)$-subgroup. Thus, as above $F[G_2]$ is a
crossed product division algebra. Hence, $D$ is a crossed product
division algebra.

Finally, consider the case where $G$ has no elements of order
dividing $\deg(D)$. Clearly, $G$ has no $Q_8$ or
$\operatorname{SL}(2,5)$-subgroup. Let $A$ be a maximal normal
abelian subgroup of $G$ with $|G:A|<\infty$. Put $H=C_G(A)$. by
Lemma \ref{lem4}, $F[H]$ is an $F[A]$-central division algebra.
Moreover, Theorem \ref{thm14} implies that $H'\subseteq F[H]'\cap
F[A]$. So, if $a\in H'$ then
$a^{\deg(F[H])}=\operatorname{Nrd}_{F[H]}(a)=1$. Therefore,
$\operatorname{ord}(a)$ divides $\deg(D)$ which leads to $a=1$.
Thus, $H'=1$ and so $H$ is abelian. Since $A$ is maximal, we
conclude that $H=A$ and the result follows.
\end{proof}
As we mentioned before, by the works of Shirvani and Wehrfritz, if
$G$ is a soluble-by-finite subgroup of a division algebra $D$, then
$G$ has a normal abelian subgroup of index dividing $30\deg(D)^2$.
But in Corollary \ref{cor15}, the case in which $G$ has no normal
$\operatorname{SL}(2,5)$-subgroups is considered. Thus, if we would
like to achieve the above bound, we must consider the case
$\operatorname{SL}(2,5)\lhd G$. If $Z(F[S])=F$, a similar method as
in Corollary \ref{cor15} is enough to find the required bound. For,
since $Z(F[S])=F$, Theorem \ref{thm14} gives a decomposition
$D=F[G_1]\otimes_FF[G_2]$. Since $Z(F[S])=F$, we must have
$C_{N}(S)=F^*$ where $N=N_{D^*}(S)$. Thus $N/F^*\cong
\operatorname{Aut}(S)\cong\operatorname{Sym}(5)$. But, since
$G_1\subseteq N$ we conclude that $|G_1:F^*|$ divides 120. On the
other hand, by Corollary \ref{cor15}, $G_2$ has a normal abelian
subgroup $A$ of index dividing $\deg(D)^2/4$. Therefore, $|G:G\cap
F^*A|$ divides $120\deg(D)^2/4=30\deg(D)^2$. But, for the case
$[Z(F[S]):F]=2$ we need a little more work. In fact we need the
following stronger version of Theorem \ref{thm14}.
\begin{thm}\label{thm17}
Let $D$ be a division algebra and $G$ be a soluble-by-finite
absolutely irreducible subgroup of $D^*$. If $G$ has no normal
subgroups isomorphic to $Q_8$ or $\operatorname{SL}(2,5)$, then it
contains a characteristic abelian subgroup of index dividing
$\deg(D)^2$.
\end{thm}
To prove Theorem \ref{thm17}, we need to recall the concept of a
general crossed product algebra and a useful lemma of \cite{We90b}.
Let $R$ be a ring, $S$ a subring of $R$ and $G$ a subgroup of the
group of units of $R$ normalizing $S$. If $R=S[G]$ and if $N=G\cap
S$ is a normal subgroup of $G$ with $R=\oplus_{t\in T} St$ for some
(and hence any) transversal $T$ of $N$ to $G$, we say that $R$ is a
crossed product algebra of $S$ by $G/N$. Note that Lemma \ref{lem4},
essentially says that if $D$ is a division algebra and $A$ is a
normal abelian subgroup of an absolutely irreducible subgroup $G$ of
$D^*$, then $D$ is a crossed product division algebra of $F[C_G(A)]$
by $G/C_G(A)$.
\begin{lem}\label{lem18}
Let $R=S[G]$ be a ring where $S$ is a subring of $R$ and $G$ is a
subgroup of $R^*$ normalizing $S$. Suppose that $R$ is a crossed
product of $S[N_i]$ by $G/N_i$ for normal subgroups $N_i$ of $G$.
Set $N=\cap N_i$. Then $R$ is a crossed product of $S[N]$ by $G/N$.
\end{lem}
\begin{proof}
See \cite[p. 213]{We90b}.
\end{proof}
Now, Lemma \ref{lem18} helps us prove Theorem \ref{thm17}:
\begin{proof}[Proof of Theorem \ref{thm17}]
By Theorem \ref{thm14} and Lemma \ref{lem4}, $G$ has a maximal
abelian normal subgroup $A$ of finite index such that $H=C_G(A)$ is
nilpotent of class at most 2 and $D$ is a crossed product of $F[H]$
by $G/H$. Now, it is not hard to see that for each $\phi\in
\operatorname{Aut}(G)$, $D$ is a crossed product of $F[H^{\phi}]$ by
$G/H^{\phi}$. Put $\mathcal{H}=\cap_{\phi}H^{\phi}$ where $\phi$
runs through $\operatorname{Aut}(G)$. Clearly $\mathcal{H}$ is a
characteristic subgroup of $G$ that is nilpotent of class at must 2.
From Lemma \ref{lem18}, it follows that $D$ is a crossed product of
$F[\mathcal{H}]$ by $G/\mathcal{H}$. Now,
$\mathcal{A}=Z(\mathcal{H})$ is a characteristic subgroup of $G$
with $\mathcal{H}'\subseteq \mathcal{A}$. But,
$|\mathcal{H}:\mathcal{A}|=[F[\mathcal{H}]:F[\mathcal{A}]]$ by Lemma
\ref{lem5}. Thus
$|G:\mathcal{A}|=|G:\mathcal{H}||\mathcal{H}:\mathcal{A}|=
[D:F[\mathcal{H}]][F[\mathcal{H}]:F[\mathcal{A}]]=[D:F[\mathcal{A}]]$,
as required.
\end{proof}
Now, let $G$ be an absolutely irreducible soluble-by-finite subgroup
of $D^*$. Suppose that $G$ has a normal
$\operatorname{SL}(2,5)$-subgroup $S$ and that $[Z(F[S]):F]=2$. In
the proof of Theorem \ref{thm14}(1) we saw that $G$ is contained in
a soluble-by-finite subgroup $M$ such that $S\lhd M$. Moreover,
$|M:SC_M(S)|=2$ and $|SC_M(S):C_M(S)|=60$. It is not hard to see
that $C_M(S)$ has no normal $\operatorname{SL}(2,5)$ or
$Q_8$-subgroup and $\deg(F[C_M(S)])$ divides $\deg(D)/2$. Thus,
$C_M(S)$ has a characteristic abelian subgroup $A$ of index dividing
$\deg(D)^2/4$. Now, since $A$ is characteristic in $C_M(S)$ and
$C_M(S)\lhd M$, we conclude that $A\lhd M$. Finally,
$|M:A|=|M:SC_M(S)||SC_M(S):C_M(S)||C_M(S):A|=2\times60 \times
\deg(D)^2/4=30\deg(D)^2$. Therefore, $|G:G\cap A|$ divides
$30\deg(D)^2$. Thus, we have proved that
\begin{thm}\label{tt}
Let $D$ be a division algebra and $G$ be a soluble-by-finite
subgroup of $D^*$. Then $G$ has a normal abelian subgroup of index
dividing $30\deg(D)^2$.
\end{thm}
Here, we should point out that the above results concerning the
structure of soluble-by-finite subgroups of division algebras were
nicely obtained by Wehrfritz for the soluble-by-finite subgroups of
$\operatorname{GL}_n(D)$, where $D$ is a division algebra in
\cite{Weh07}. His results demonstrates that if $G$ is a
soluble-by-finite subgroup of $\operatorname{GL}_n(D)$, then $G$ has
an abelian normal subgroup of finite index dividing
$b(n)\deg(D)^{2n}$, where $b(n)$ is an integer valued function which
depends only on $n$.
By an interesting result of \cite{KeMa07}, the existence of a finite
absolutely irreducible subgroup in a division algebra $D$ guarantees
that $D$ is a crossed product division algebra. Here, we give a very
short proof for this fact.
\begin{thm}
Let $D$ be an $F$-central division algebra such that $D^*$ contains
an irreducible finite subgroup $G$. Then $D$ is a crossed product
division algebra.
\end{thm}
\begin{proof}
From the first section, recall that $G$ is one of the groups of
types (A1)-(A6). We consider all the above cases separately.

(A1) Since $G$ contains no subgroup isomorphic to $Q_8$ or
$\operatorname{SL}(2,5)$, by Theorem \ref{thm14}, it contains a
maximal normal abelian subgroup $A$ such that $C_G(A)$ is nilpotent
of class at most 2. Therefore, $C_G(A)$ is abelian by Lemma
\ref{lem8}. But, $A$ is maximal. So, $C_G(A)=A$ and the result
follows by Lemma \ref{lem4}.

(A2) Let $x\in Q_{2^{n+1}}$ have order $2^n$. Since $x$ centralizes
$C_m$, the subgroup $A=C_m\times \langle x\rangle$ is abelian and of
course $|G:A|=2$. Thus $A\lhd G$ and clearly $A$ is non-central.
Therefore, $C_G(A)=A$, as desired.

(A3) By Proposition \ref{prop6}, we have $D=F[Q_8]\otimes_FF[M]$.
Now, by the case (A1), $F[M]$ is a crossed product division algebra
which implies that $D$ is a crossed product division algebra.

(A4) Let $Q$ be the normal $Q_8$-subgroup of $G$. By Theorem
\ref{thm7}, $D$ has a decomposition as $D=F[Q]\otimes_FF[G_2]$,
where $G_2/F^*$ is a homomorphic image of $G/Q$. Thus, $|G_2:F^*|$
divides $6$. Therefore, if $G_2$ is non-abelian, then we must have
$G_2/F^*\cong\operatorname{Sym}(3)$. Put $A=F^*\langle a\rangle$,
where $a\in G_2$ has order 3 in $G_2/F^*$. Clearly, $A$ is an
abelian subgroup of $G_2$ with $|G_2:A|=2$. Thus, $A\lhd G_2$ and
$C_G(A)=A$ and the result follows.

(A5) Recall that $\operatorname{SL}(2,3)\cong \langle
i,j,-(1+i+j+ij)/2\rangle$ where $i,j$ are the usual generators of
$Q_8$. Thus, we can easily see that $F[\operatorname{SL}(2,3)]=F[Q]$
(where $Q$ is the $Q_8$- subgroup of $\operatorname{SL}(2,3)$). Now,
a similar argument as used in (A3) shows that $D$ is a crossed
product division algebra.

(A6) From \cite[Th. 2.1.11]{ShWe86} we know that in this case $D$ is
the ordinary quaternion division algebra. Thus, it is a crossed
product division algebra.
\end{proof}
\subsection{Tits' Alternative for subgroups of a division algebra}
As we have seen in the introduction, there are several group
theoretic properties whose occurrence in a subnormal subgroup $G$ of
$D^*$ lead to the commutativity of $G$. For example, if $G$ is
periodic or soluble then by known results of Herstein \cite{Her78}
and Stuth \cite{Stu64} $G$ is contained in $F^*$. Also, whenever
$\dim_FD<\infty$ and $G$ satisfies a group identity, then combining
the above results and a theorem of Platonov \cite[p. 149]{Zal93}
entails that $G\subset F^*$. But, many of such conditions are
special cases of a more general condition:
\begin{center}
``$G$ has no non-cyclic free subgroups.''
\end{center}
The existence of a non-cyclic free subgroup in $D^*$ has been posed
as a conjecture by Lichtman in \cite{Lic77} and then in subnormal
subgroups of $D^*$ by Goncalves and Mandel in \cite{GonMan86}.
However, it is known that by the work of Goncalves if $D$ is of
finite dimension over its center, then these problems have positive
answers (cf. \cite[p. 207]{Kar88}). On the other hand, in the finite
dimensional setting, a stronger result obtained by Wehrfritz asserts
that if $\deg(D)$ is a prime power then every subgroup of $D^*$
either contains a non-cyclic free subgroup or possesses an abelian
normal subgroup of a finite index dividing $60\deg(D)^2$ (see
\cite{We90c}). Here, using the results of the previous section, we
give a similar characterization without any condition on $\deg(D)$.
The main theorem of this section is:
\begin{thm}\label{t}
Let $D$ be a finite dimensional division algebra over its center $F$
and let $G$ be a subgroup of $D^*$. Then the following statements
are equivalent:
\begin{enumerate}
  \item $G$ contains no non-cyclic free subgroups;
  \item $G$ is soluble-by-finite;
  \item $G$ is abelian-by-finite;
  \item $G$ satisfies a group identity.
\end{enumerate}
Moreover, if $\deg(D)$ is odd then we can replace $(2)$ by
\begin{enumerate}
  \item[$(2)'$] $G$ is soluble.
\end{enumerate}
Furthermore, if one of the above conditions holds then $G$ contains
an abelian normal subgroup of index dividing $30\deg(D)^2$.
\end{thm}
To prove Theorem \ref{t}, we need to recall
\begin{thm}[Tits' Alternative \cite{Tit72}]\label{thms4}
If $F$ is a field then
\begin{enumerate}
  \item If $\operatorname{char}(F)=0$ then every subgroup of $\operatorname{GL}_n(F)$
  either is soluble-by-finite or contains a noncyclic free subgroup;
  \item If $H$ is a finitely generated subgroup of $\operatorname{GL}_n(F)$,
  then either $H$ is soluble-by-finite or $H$ contains a non-abelian free subgroup.
\end{enumerate}
\end{thm}
Note that Theorem \ref{t} essentially says that a same conclusion as
in Tits' Alternative is valid for every subgroup $G$ of $D^*$
without any extra condition on the set of generators of $G$ or
$\operatorname{char} D$. Thus, it can be considered as a version of
Tits' Alternative for subgroups of division algebras.

We need also the following lemma which is a special case of
\cite[Lem. 1]{Weh71} combining with Schur's Theorem (cf. \cite[Th.
$9.9'$]{Lam01}).
\begin{lem}\label{l}
Let $G$ be a subgroup of $\operatorname{GL}_n(F)$ such that every
finitely generated subgroup of $G$ is abelian-by-finite. Then $G$
has an abelian normal subgroup $A$ such that $G/A$ is a periodic
linear group over $F$, so that $G/A$ is a locally finite group.
\end{lem}
Now, we are in a position to prove Theorem \ref{t}.

\begin{proof}[Proof of Theorem \ref{t}]
Lemma \ref{lem} yields that (2) and (3) are equivalent. If $A$ is a
normal abelian subgroup of $G$ with $|G:A|=e$ then clearly
$[x^e,y^e]=1$ is a group identity for $G$. This proves
(3)$\Rightarrow$(4). Also, (4)$\Rightarrow$(1) is trivial. So, it
remains to prove (1)$\Rightarrow$(3). View $G$ as a subgroup of
$\operatorname{GL}_n(F)$ where $n=\dim_FD$. If
$\operatorname{char}(F)=0$ then by (1) of Theorem \ref{thms4} $G$ is
soluble-by-finite and the result follows from Lemma \ref{l} in this
case. Thus, we may assume that $\operatorname{char}(F)>0$. If this
is the case, then from Tits' Alternative, it follows that every
finitely generated subgroup of $G$ is soluble-by-finite. Combining
this with Lemma \ref{lem} implies that every finitely generated
subgroup of $G$ is abelian-by-finite. Now, by Lemma \ref{l} we
conclude that $G$ contains an abelian normal subgroup $A$ such that
$G/A$ is locally finite. Put
$$\mathcal{S}=\{N\lhd G|N\  \textrm{is abelian and}\  G/N\  \textrm{is locally finite}\}.$$
Since $A\in\mathcal{S}$, it follows that $\mathcal{S}\neq
\emptyset$. So we can choose a maximal non-trivial element $B\in
\mathcal{S}$. Let $H=C_G(B)$. Since $B$ is maximal in $\mathcal{S}$
it can be seen that $Z(H)=B$ and so $H$ is center-by-(locally
finite). This forces that $H'$, the derived subgroup of $H$, is
locally finite. Now, given $h,k\in H'$ the subgroup $\langle
h,k\rangle$ is finite and so is cyclic as
$\operatorname{char}(F)>0$. Thus $hk=kh$ which means that $H'$ is
abelian. But, $H'\lhd G$ and centralizes $B$. Therefore $BH'$ is a
normal abelian subgroup of $G$ and thus $BH'=B$, because $B$ is
maximal. Hence $H'\subseteq B$. So, if we put $K=F[B]$ then from
Lemma \ref{lem5} it follows that $[K[H]:K]=|H:B|$. On the other hand
from Lemma \ref{lem4} we know that $[F[G]:F[H]]=|G:H|$. This gives
$|G:A|=[F[G]:F[H]][F[H]:K]$ (note that $F[H]=K[H]$) which
immediately implies that $|G:A|$ divides $\deg(D)^2$. Thus, $G$ is
abelian-by-finite.

Now, if $\deg(D)$ is odd, then $G$ contains an abelian normal subgroup $A$
of index dividing $\deg(D)^2$ by Theorem \ref{thm14}. Now, since
$G/A$ has an odd order, by Feit-Thompson theorem about the solubility
of the groups of odd order (\cite{FeTh63}), it is soluble and hence
in this case we have (2)$\Rightarrow(2)'$.

The rest follows from Theorem \ref{tt}.
\end{proof}

\section{Finitely generated subnormal subgroups of division rings}
As we have seen in the first section, by a theorem of Herstien every
periodic subnormal subgroup of $D^*$ is central. Trivially, this
implies that every finite subnormal subgroup in a division ring is
contained in the center. However, a natural question is to ask what
happens when we replace the phrase ``finite'' by ``finitely
generated''. This problem was initially investigated in
\cite{AkMa00}, where Akbari and Mahdavi-Hezavehi proved that if $D$
is a division algebra and its center is an algebraic extension of
$\mathbb{Q}$ then every finitely generated normal subgroup of $D^*$
is central. They also showed that for every division algebra $D$ and
every $n\geq 2$ the group $\operatorname{GL}_n(D)$ has no noncentral
finitely generated normal subgroups. Another notable result in this
direction obtained in \cite{AkMaMah99} imparts that every finitely
generated normal subgroup of a division algebra is central. Finally,
the structure of finitely generated normal subgroups of general
linear groups over division algebras was completely determined in
\cite{MaMahYa00}. The main result of this paper guarantees that for
every division algebra $D$ and each $n\in \mathbb{N}$, the group
$\operatorname{GL}_n(D)$ has no noncentral finitely generated
subnormal subgroups. This section is mainly devoted to give this
observation and some peripheral consequences.  Before presenting the
main result, it is beneficial to list some needful old theorems.
\begin{thm}[Stuth \cite{Stu64}]\label{thms1}
Let $D$ be a division ring. Then every subnormal soluble subgroup of $D^*$ is central.
\end{thm}
\begin{thm}[Herstein \cite{Her78}]\label{thms2}
All periodic subnormal subgroups of a division ring are central.
\end{thm}
\begin{cor}\label{cors3}
Let $D$ be a division ring. Then every subnormal soluble-by-finite subgroup of
$D^*$ is central.
\end{cor}
\begin{proof}
Let $H$ be a subnormal soluble-by-finite subgroup of $D^*$ and $S$ be a soluble normal subgroup
of $H$ of finite index. Now, subnormality of $S$ in $D^*$ implies that $S$ is central (Theorem \ref{thms1}).
This forces that $H$ is center-by-finite and so $H'$ is a finite subnormal subgroup of $D^*$. Now, from
Theorem \ref{thms2} it follows that $H'$ is central and so it is abelian. Therefore, $H$ is soluble and
hence it is central.
\end{proof}
\begin{thm}\label{thms5}
Let $F$ be a field and $H$ be a finitely generated subgroup of
$G=\operatorname{GL}_n(F)$. If $\{a+xI_n|a\in H, x\in F\}\subseteq
N_G(H)$, then $H$ is soluble-by-finite.
\end{thm}
\begin{proof}
If $F$ is finite, there is nothing to prove. Thus, we may assume
that $F$ is infinite. Let $H=\langle a_1,a_2,\ldots,a_k\rangle$ and
$R$ be the subring of $\operatorname{M}_n(F)$ generated by the
elements of the set $\{a_j,a_j^{-1}\}_{j=1}^k$. Now, one can easily
check that $R\subseteq \operatorname{M}_n(P(\Lambda))$ where $P$ is
the prime subfield of $F$ and $\Lambda$ is the set of elements in
$F$ occurring as the entries of $a_j$ and $a_j^{-1}$,
$j=1,\ldots,k$. Thus, we have $H\subseteq
\operatorname{GL}_n(P(\Lambda))$. Now, the Noether Normalization
Lemma implies that $P(\Lambda)$ contains a subfield $L$ of finite
codimension such that $L=\mathbb{Q}$ or $L=K(y)$ for some subfield
$K$ and $y$ is transcendental over $K$ (note that since $H$ is
infinite, if $\operatorname{char}(F)>0$, then $P(\Lambda)$ is not
algebraic over $P$). Put $m=[P(\Lambda):L]$ and consider the
following sequence of mappings
$$
\xymatrix{ \operatorname{M}_n(P(\Lambda))\ar[r]^{nat\ \ \
}&\operatorname{M}_n(L)\otimes_LP(\Lambda) \ar[r]^{1\otimes\iota\ \
}& \operatorname{M}_n(L)\otimes_L\operatorname{M}_m(L)\ar[r]^{\ \ \
\ \ \ nat}&\operatorname{M}_{nm}(L),}
$$
where $\iota$ is the regular representation. Since all the above
maps are injective, $H$ is embedded in $\operatorname{GL}_{nm}(L)$.
Also, since $\iota:u\mapsto uI_m$ we have $\{a+uI_{nm}|a\in H,u\in
L\}\subseteq N_{\operatorname{GL}_{nm}(L)}(H)$. Now, since $L$ is
the field of fractions of $\mathbb{Z}$ or $K[y]$ and $H$ is finitely
generated, a similar argument as above implies that $H\subseteq
\operatorname{GL}_{nm}(S[\Delta])$ where $S$ is either $\mathbb{Z}$
or $K[y]$ and $\Delta=\{u_1/v_1,\ldots,u_t/v_t\}$ is a finite subset
of $L$. Now, let $H$ have a free subgroup $\mathcal{F}$ and
$a,b\in\mathcal{F}$ be arbitrary. Suppose at least one of the
entries of $(b+xI_{nm})^{-1}aba(b+xI_{nm})\in H$ depends on $x$.
Since $\det(b+xI_{nm})$ is a polynomial in $x$ of degree $nm$, we
conclude that for each $1\leq i,j\leq nm$, the $(i,j)$th entry of
$(b+xI_{nm})^{-1}$ is of the form $p_{ij}(x)/q(x)\in L(x)$ where
$\deg q(x)=nm$ and $\deg p_{ij}(x)\leq nm-1$. So, the $(i,j)$th
entry of $(b+xI_{nm})^{-1}aba(b+xI_{nm})$ is of the form
$p_{ij}(x)/q(x)\in L(x)$ where $\deg q(x)=nm$ and $\deg
p_{ij}(x)\leq nm$. Let the $(r,s)$th entry of
$(b+xI_{nm})^{-1}aba(b+xI_{nm})$ depend on $x$. Put
$p_{rs}(x)=\sum_{i=0}^{nm}a_ix^i$ and
$q(x)=x^{nm}+\sum_{i=0}^{nm-1}b_ix^i$. If $a_{nm}=u_{t+1}/v_{t+1}$,
then $p(x)/q(x)=p_{rs}(x)/q(x)-a_{nm}\in
S[\Delta\cup\{u_{t+1}/v_{t+1}\}]$ for each $x\in L$. Also, it is
clear that $\deg p(x)\leq nm-1$. Multiplying $p(x)$ and $q(x)$ by
suitable scalers, we may assume that $p(x),q(x)\in S[x]$. Put
$p(x)=\sum_{i=0}^{nm-1}a_i'x^i$ and $q(x)=\sum_{i=0}^{nm}b_i'x^i$.
But $b_0'\neq0$, because $\det(b)\neq0$. Now, changing the variable
$x$ to $b_0'x$ gives $p_1(x),q_1(x)\in S[x]$ such that $\deg q_1=nm$
and $\deg p_1\leq nm-1$, where the constant term of $q_1(x)$ is $1$
and for each $x\in L$, we have $p_1(x)/q_1(x)\in
S[\Delta\cup\{u_{t+1}/v_{t+1}\}]$. Let
$\mathcal{P}=\{\rho_1,\ldots,\rho_l\}$ be the set of all primes
occurring in the factorizations of $\{v_1,\ldots,v_{t+1}\}$ into
prime factors in $S$. For each natural number $r$, put
$x_r=(\rho_1\ldots \rho_l)^r$. Since $\deg p_1(x)<\deg q_1(x)$, for
a large enough number $r$, we have $p_1(x_r)/q_1(x_r)<1$ if
$S=\mathbb{Z}$ and the degree of the denominator of
$p_1(x_r)/q_1(x_r)$ with respect to $y$ is greater than the
nominator if $S=K[y]$. Thus, there is a $\rho_j\in \mathcal{P}$ such
that $\rho_j$ divides $q_1(x_r)$. But, since the constant term of
$q_1(x)$ is $1$, we have $\gcd (q_1(x_r),\rho_j)=1$ for each $r\geq
1$ and each $1\leq j\leq l$ which is a contradiction. Thus
$(b+xI_{nm})^{-1}aba(b+xI_{nm})$ does not depend on $x$. Hence, we
must have $babab^{-1}=(b+1)aba(b+1)^{-1}$. Consequently,
$baba=abab$. But, this gives a nontrivial relation between $a$ and
$b$ which is a contradiction, because $a,b$ are arbitrary in the
free group $\mathcal{F}$. This shows that $H$ has no free subgroup.
Now, the result follows by Tits' Alternative.
\end{proof}
\begin{cor}\label{cors6}
Let $D$ be an $F$-division algebra. If $H$ is a finitely generated
normal subgroup of $D^*$, then $H$ is central.
\end{cor}
\begin{proof}
View $D^*$ as a subgroup of $\operatorname{GL}_m(F)$ where
$m=[D:F]$. Since $\{a+xI_m|a\in H, x\in F\}$ is contained in
$N_{\operatorname{GL}_m(F)}(D^*)$, Theorem \ref{thms5} shows that
$H$ must be soluble-by-finite. Now, the result follows from
Corollary \ref{cors3}.
\end{proof}
There is an elegant proof of Corollary \ref{cors6} in \cite{Ma02}
based on some known results in the theory of \textbf{PI}-rings.
However, the argument used in Theorem \ref{thms5} has this
flexibility that can be modified to give similar results for
subnormal subgroups instead of normal subgroups. To achieve this,
let $H$ be a finitely generated subgroup of
$\operatorname{GL}_n(F)$. Let there be a finite class
$\{H_j\}_{j=1}^r$ of subgroups of $\operatorname{GL}_n(F)$ such that
$H=H_r\lhd H_{r-1}\lhd\ldots\lhd H_1$ where $\{a+xI_n|a\in H, x\in
F\}\subseteq N_G(H_1)$. Keep the notations of the proof of Theorem
\ref{thms5} and recall that in the course of the proof of this
theorem we had viewed $H$ as a subgroup of
$\operatorname{GL}_{nm}(L)$. Now, for each pair $a,b\in H$ and $x\in
L$, set $c_1(a,b,x)=(b+xI)a(b+xI)^{-1}$, and for $i>1$ define $c_i$
inductively by $c_{i-1}bc_{i-1}^{-1}$. Thus $c_1\in H_1$ and by
induction $c_r\in H$. Here, we claim that for each $i$ we have
$c_i=(b+xI)w_i(a,b)(b+xI)^{-1}$ where $w_i(a,b)$ is a reduced word
in $a,a^{-1},b,b^{-1}$, the first and last letters of which are
$a$ or $a^{-1}$, respectively. For $i=1$ there is nothing to prove.
Also, if $c_i=(b+xI)w_i(a,b)(b+xI)^{-1}$ then by induction we
conclude that
$c_{i+1}=c_ibc_i^{-1}=[(b+xI)w_i(a,b)(b+xI)^{-1}]b[(b+xI)w_i(a,b)^{-1}(b+xI)^{-1}]
=c_{i+1}=c_ibc_i^{-1}=(b+xI)[w_i(a,b)bw_i(a,b)^{-1}](b+xI)^{-1}$ and
since the first and last alphabets of $w_i(a,b)bw_i(a,b)^{-1}$ are
$a$ and $a^{-1}$, the claim is established. Now, a similar argument
as in the proof of Theorem \ref{thms5} by using
$c_r=(b+xI)w_r(a,b)(b+xI)^{-1}$ instead of $(b+xI)aba(b+xI)^{-1}$
implies that $H$ is soluble-by-finite. Now, as in Corollary
\ref{cors6} we can directly obtain
\begin{thm}\label{thms6-1}
Let $D$ be an $F$-central division algebra. Then every finitely
generated subnormal subgroup of $D^*$ is central.
\end{thm}
Here is a good place to exhibit some additional results related to
the structure of subnormal subgroups of general linear groups over
division rings.
\begin{thm}[\cite{MaAk98}]\label{thms6-2}
Let $D$ be a division ring with center $F$. If either $n\geq3$ or $n=2$ but $D$ contains at least four elements, then for every
subnormal subgroup $N$ of $\operatorname{GL}_n(D)$ we have either
$N\subseteq F$ or $\operatorname{SL}_n(D)\subseteq N$.
\end{thm}
Also, in \cite{AkMa00}, it is proved that if $n\geq2$ and $N$ is an
infinite noncentral normal subgroup of $\operatorname{GL}_n(D)$ then
$N$ is not finitely generated provided that $D$ is of finite
dimension over its center. Combining this with Theorems
\ref{thms6-1} and \ref{thms6-2} one can easily prove that
\begin{thm}
Let $D$ be an infinite $F$-central division algebra. If $N$ is a
finitely generated subnormal subgroup of $\operatorname{GL}_n(D)$
($n\geq1$) then $N\subseteq F^*$.
\end{thm}
We close this section by recalling a useful result of
\cite{AkMaMah99} which will be applied in the subsequent sections.
\begin{cor}\label{cors7}
Let $D$ be a division ring with center $F$ and let $M$ be a maximal
subgroup of $D^*$. If $|M:M\cap F^*|<\infty$, then $D=F$.
\end{cor}
\begin{proof}
Put $A=M\cap F^*$. Let $T=\{u_1,\ldots,u_k\}$ be a left transversal
of $A$ in $M$ and $E=\{\sum_{j=1}^kf_ju_j|f_j\in F\}$. Clearly $E$
is a finite dimensional division algebra and $M$ is a maximal
subgroup of $E^*$ (not necessarily proper maximal subgroup). Choose
an element $u\in E^*$ such that $u=1$ if $M=E^*$ and $u\in
E^*\setminus M$ if $M\neq E^*$. From the maximality of $M$, it
follows that $E^*=\langle T,u\rangle F^*$. So $\langle T,u\rangle$
is a finitely generated normal subgroup of $E^*$. Thus, $\langle
T,u\rangle$ is central in $E^*$ by Corollary \ref{cors6}. This
forces that $E$ is a field. Now, if $M\neq E^*$ then $E=D$ and the
result follows. Thus we may assume that $M=E^*$. This implies that
$|E^*:F^*|<\infty$ and hence we have either $M^*=E^*=F^*$ or $E$ is
a finite field properly containing $F$. If $M=F^*$ then
$|D^*:F^*|=p$ for some prime $p$ as $F^*\lhd D^*$. This implies that
$D^*/F^*$ is cyclic and hence $D$ is commutative. Now, consider the
case in which $E$ is a finite field. In this case, there is an
element $a\in D^*$ such that $E^*=\langle a\rangle$. Since $a$ has a
finite order and is noncentral, Herstein's Lemma (\cite[p.
206]{Lam01}) implies that there is a $b\in D^*$ such that
$a^b=a^i\neq a$. Thus, $b\in N_{D^*}(E^*)$ and so $\langle
M,b\rangle\subseteq N_{D^*}(E^*)$. Now, since $M$ is maximal we
conclude that $N_{D^*}(E^*)=D^*$ and thus $E^*\lhd D^*$. Finally,
the Cartan-Brauer-Hua \cite[p. 211]{Lam01} Theorem implies that $D$
is commutative.
\end{proof}
\section{Maximal subgroups}\label{SecT}
Some old results presented in Section~\ref{sec1birk} show that in a division ring
$D$ subnormal subgroups behave similar to $D^*$ in several manners. For
example they are very far from being commutative as well as they are
not periodic. These phenomena imply that the subnormal subgroups
are ``big''. But, one would like to know that how big maximal
subgroups are in $D^*$? This section is devoted to presenting some
recent results concerning this question.

\subsection{Existence of free subgroups in maximal subgroups}
As we have seen, by a result of Goncalves every subnormal subgroup
of a division algebra contains a noncyclic free subgroup.
Following this outcome, in \cite{Ma01} the author studied the existence of
noncyclic free subgroups in a maximal subgroup of a division
algebra.  The main result of this study asserts that in a noncrossed
product division algebra $D$ every maximal subgroup of $ D^*$
contains a noncyclic free subgroup, that is a similar result which
holds for subnormal subgroups of $D^*$. Here, we are going to
achieve this interesting result. We begin by
\begin{prop}\label{propt8}
Let $D$ be an $F$-central division algebra. If $D^*$ has a
noncommutative soluble-by-finite maximal subgroup $M$, then $M$ is
absolutely irreducible and contains an abelian normal subgroup $A$
such that $C_M(A)=A$, hence $D$ is a crossed product division
algebra.
\end{prop}
\begin{proof}
Since $M$ is maximal we have either $F[M]=D$ or $F[M]^*=M$. But the
latter case implies that the division algebra $F[M]$ is commutative
(Theorem \ref{thms1}). This yields the commutativity of $M$ which is
a contradiction. So $F[M]=D$, i.e., $M$ is an absolutely irreducible
subgroup of $D^*$. Now, since $M$ is soluble-by-finite, it contains
an abelian normal subgroup of finite index $A$. Take $A$ maximal in
$M$. However, since $M$ is maximal and $\langle
M,C_{D^*}(A)\rangle\subseteq N_{D^*}(A)$ we have either
$N_{D^*}(A)=D$ or $C_{D^*}(A)\subseteq M$. But, if $N_{D^*}(A)=D$
then by Theorem \ref{thms1} we obtain $A\subseteq F^*$ and so
$M/M\cap F^*$ is finite. Here, Corollary \ref{cors7} yields $D=F$
which is absurd. Thus $C_{D^*}(A)\subseteq M$ and consequently
$C_{D^*}(A)$ is soluble-by-finite. On the other hand, by the Double
Centralizer Theorem $C_{D}(A)=C_{D}(F[A])$ is a division algebra.
This forces that $C_{D}(A)$ is commutative, because its unit group
is soluble-by-finite. Thus $C_M(A)$ is abelian. However $A\subseteq
C_{M}(A)\lhd M$ and $A$ is a maximal abelian normal subgroup of $M$.
This gives $C_{M}(A)=A$ and the result follows.
\end{proof}
Combining Theorem \ref{t} and Proposition \ref{propt8} implies that:
\begin{thm}\label{cort14}
Let $D$ be a finite dimensional division algebra over its
center $F$. If $M$ is a non-abelian maximal subgroup of $D^*$ then
the following statements are equivalent:
\begin{enumerate}
  \item $M$ does not contain a non-cyclic free subgroup;
  \item $M$ contains an abelian normal subgroup $A$ such that $C_M(A)=A$
  and $M/A\cong\Gal(F[A]/F)$;
  \item $M$ is soluble-by-finite;
  \item $M$ is abelian-by-finite;
  \item $M$ satisfies a group identity.
\end{enumerate}

Moreover, if either of the above conditions holds, then $D$ is a
crossed product division algebra. In particular, if $D$ is a
noncrossed product division algebra then every maximal subgroup of
$D^*$ contains a noncyclic free subgroup.
\end{thm}

To proceed, we have to recall the following theorem which
characterizes the field extensions $K/F$, where $K$ is radical over
$F$, i.e., $K^*/F^*$ is periodic. For a proof see \cite[p.
245]{Lam01}.
\begin{thm}\label{thmt9}
Let $K/F$ be a proper field extension
and let $P$ be the prime subfield of $F$. If $K$ is radical over $F$
then $\operatorname{char}(F)>0$, and either $K$ is purely
inseparable over $F$ or $K$ is algebraic over $P$.
\end{thm}
\begin{cor}\label{cort10}
Let $D$ be a noncommutative $F$-central division algebra. If $M$ is
a soluble-by-finite maximal subgroup of $D^*$. Then
\begin{enumerate}
  \item $F^*\subseteq M$;
  \item If $M/F^*$ is periodic, then $M$ is commutative.
\end{enumerate}
\end{cor}
\begin{proof}
(1) Since $M$ is maximal we have either $F^*M=M$ or $F^*M=D^*$. But
the latter case implies that $D^*$ is soluble-by-finite which is
impossible. Thus $F^*M=M$ and the result follows.

(2) Suppose that $M$ is noncommutative. By Proposition \ref{propt8},
$M$ has a proper normal abelian subgroup such that $C_M(A)=A$. Thus
$F[A]/F$ is Galois by Lemma \ref{lem4}. On the other hand since $M$
is maximal, as in (1)  we conclude that $F[A]\subseteq M$ and hence
$F[A]^*/F^*$ is torsion. Therefore, from Theorem \ref{thmt9} it
follows that $F[A]/F$ is separable, $\operatorname{char}F=p>0$ and
$F[A]$ is algebraic over the prime subfield of $F$. Consequently,
$F$ is algebraic over the prime subfield $\mathbb{F}_p$. This forces
that $D$ is commutative which is a contradiction.
\end{proof}
Corollary \ref{cort10} makes a facility to improve the result of
Corollary \ref{cors7} as follows (of course in the finite dimensional case).
\begin{thm}\label{thmt11}
Let $D$ be a noncommutative $F$-central division algebra. If
$M$ is a maximal subgroup of $D^*$ such that $M/M\cap F^*$ is periodic
then $M$ is abelian.
\end{thm}
\begin{proof}
Since $M/M\cap F^*$ is torsion, it follows that $M$ contains no
non-cyclic free subgroup. Now, the result follows from
Theorem \ref{t} and Corollary \ref{cort10}.
\end{proof}
\begin{cor}\label{cort12}
Let $D$ be an $F$-central division algebra and $M$ be a maximal subgroup of
$D^*$. If $M$ is nilpotent then $M$ is abelian.
\end{cor}
\begin{proof}
As in the proof of Proposition \ref{propt8} one can conclude that
$F[M]=D$. Moreover, it is not hard in this case to see that
$Z(M)=F^*$. But $M$ is center-by-finite (It is well-known, however
for a reference one can see \cite[Th. 1]{Wehr07}) and so the result
follows from Theorem \ref{thmt11}.
\end{proof}
Corollary \ref{cort14} has been nicely extended in \cite{Mah04} by
describing some properties of maximal subgroups of
$\operatorname{GL}_n(D)$ where $D$ is a division algebra. The
results appeared in \cite{Mah04} specify that every maximal subgroup
$M$ of $\operatorname{GL}_n(D)$ either contains a non-cyclic free
subgroup or contains a (subnormal) subgroup $A$ that is a direct
product of the multiplicative groups of finitely many proper field
extensions of the center $F$ of $D$, such that $A$ has a finite
index in $M$ when $F$ has characteristic zero, or $M/A$ is locally
finite when $F$ has characteristic $p\neq0$. Moreover, a more
general result was presented in \cite{KiaMa10} concerning the
existence of free subgroups in the maximal subgroups of a subnormal
subgroup $N$ in $\operatorname{GL}_n(D)$ which we are going to
display here without proof.
\begin{thm}\label{thmt15}
Let $D$ be a noncommutative division algebra, and $N$ a subnormal
subgroup of $\operatorname{GL}_n(D)$ ($n\geq 1$). Given a maximal
subgroup $M$ of $N$, then either $M$ contains a noncyclic free
subgroup or there exists an abelian subgroup $A$ and a finite family
$\{K_j\}_{j=1}^m$ of fields properly containing $F$ with
$K_j^*\subset M$ for all $1\leq j\leq r$ such that $M/A$ is finite
if $\operatorname{char}(F)=0$ and $M/A$ is locally finite if
$\operatorname{char}(F)\neq 0$, where $A\subseteq K_1^*\times\ldots
K_m^*$.
\end{thm}
As we have seen, proving the above results were strongly dependent
on the conclusion of Corollary \ref{cors7}. This corollary was also
applied to determine the structure of maximal subgroups of
$\operatorname{GL}_n(D)$ with a finite conjugacy class in
\cite{KiaRam09} by proving
\begin{thm}\label{thmt16}
Let $D$ be an infinite division ring with center $F$ and $M$ a
maximal subgroup of $\operatorname{GL}_n(D)$ ($n\geq1$). If $M$ is a
finite conjugacy class group then it is abelian.
\end{thm}
The proof of Theorem \ref{thmt16} requires using some known results
from the theory of \textbf{PI}-rings. However, we are going to
present the proof for the case $n=1$. The following theorem is the
main needed material at the beginning of this direction.
\begin{thm}\label{thmt17}
A linear group with finite conjugacy class is center-by-finite.
\end{thm}
\begin{proof}[Proof of Theorem \ref{thmt16} for $n=1$]
At first, let $D$ be of finite dimension over $F$. Since $M$ is
maximal in $D^*$ we have either $F[M]^*=M$ or $F[M]=D$. But, in the
former case Theorem \ref{thmt17} implies that the finite dimensional
division algebra $F[M]$ has a center-by-finite unit group. Thus
$F[M]$ is abelian and the result follows. Also, in the latter case
one can easily deduce that $Z(M)=F^*$ and again by Theorem
\ref{thmt17} we conclude that $|M:F^*|$ is finite and the result
follows from Corollary \ref{cors7}.

Now, consider the general case. As above we can conclude that
$F^*\subseteq M$. Let $x$ be a noncentral element of $M$. Since $x$
has finitely many conjugates, it follows that $|M:C_M(x)|<\infty$.
Thus $M$ has a normal subgroup $N$ of finite index which is
contained in $C_M(x)$. But $M\subseteq N_{D^*}(F(N))$ as $N\lhd M$.
However, since $M$ is maximal we have either $N_{D^*}(F(N))=M$ or
$N_{D^*}(F(N))=D^*$. Let $N_{D^*}(F(N))=M$. If this is the case,
then $F(N)^*\subseteq M$ is a finite conjugacy class group. So its
derived group is finite (see \cite[p.442]{Sco87}) and hence is
central. This immediately implies that $F(N)^*$ is soluble. Thus by
Hua's Theorem $F(N)$ is a field with $|M:F(N)^*|<\infty$. Now, we
claim that $D$ is of finite dimension over $F$. For, let $F[M]^*\neq
M$. Since $M$ is maximal we have $F[M]=D$. This forces that
$[D:F(N)]_l<\infty$ as $|M:F(N)^*|$ is finite. This establishes our
claim in this case (cf. \cite[Th. 15.8]{Lam01}). But, if $F[M]^*=M$,
$M\cup\{0\}$ is a division ring and then combining this with the
same argument as above establish our claim. However, to prove that
$M\cup\{0\}$ is a division ring, it is enough to show that for every
pair $a,b\in M$ with $a\neq b$ we have $a-b\in M$. To achieve this,
put $u=ab^{-1}$. Since $u\neq1$ and $|M:F(N)^*|<\infty$ we have
$u^m\in F(N)$ for some $m>1$. If $u^m=1$ then $u-1$ is algebraic
over $F$ and so $(u-1)^{-1}\in F[M]$. This yields $u-1\in F[M]^*=M$.
Also if $u^m\neq1$ then $(u-1)(u^{m-1}+\ldots+1)=u^m-1\in F(N)^*$.
This again implies that $u-1\in M$. Therefore, $a-b=(u-1)b\in M$ and
so our claim. Here, by the first paragraph we deduce that $M$ is
abelian. Now, we are left with the case in which
$N_{D^*}(F(N))=D^*$. Here, the Cartan-Brauer-Hua Theorem guarantees
that $F(N)=D$ or $F(N)\subseteq F$. In the former case, every
element of $D$ commutes with $x$ as $N\subseteq C_M(x)$. This forces
that $x\in F$ which contradicts the choice of $x$. In the latter
case we also have $N\subseteq F^*$ and since $|M:N|<\infty$ we
conclude that $|M:F^*|<\infty$. This implies that $D$ is commutative
which is absurd.
\end{proof}
\subsection{Nilpotent and soluble maximal subgroups}
As we have seen in Theorem \ref{cort14}, if $D$ is a noncrossed
product division algebra then every maximal subgroup of $D^*$
contains a noncyclic free subgroup. But, there is no substantial
information for maximal subgroups of an infinite dimensional
division ring. However, some special types of groups that have no
noncyclic free subgroup are nilpotent and soluble groups. Thus it
may be of interest to explore the nature of such groups when they
appear as a maximal subgroup of a division ring. This was the main
theme in a series of recent works including \cite{AkEbMoSa03},
\cite{AkMaMah99} and \cite{Ebr04}. The most successful effort to
specify what types of nilpotent groups can become visible as a
maximal subgroup of a division ring has been made in \cite{Ebr04}
which asserts that such a nilpotent group is abelian. To present the
proof of this fact, at first we must give the following
\begin{lem}\label{lemn1}
Let $D$ be a division ring with center $F$. If $M$ is a nilpotent
maximal subgroup of $D^*$ then:
\begin{enumerate}
  \item $M'$ is abelian;
  \item $M$ has a normal maximal abelian subgroup $A$ containing $M'$
  such that $K=A\cup\{0\}$ is a subfield of $D$;
  \item If $u\in M\setminus A$ then $F(N)=D$ where $N=\langle A,u\rangle$.
\end{enumerate}
\end{lem}
\begin{proof}
(1) If $M$ is abelian then there is nothing to prove. So we assume
that $M$ is nonabelian. Since $M$ is maximal and nilpotent we have
$F^*=Z(M)\varsubsetneqq M$ and hence $M/F^*$ has a nontrivial
center. Let $F^*u\in Z(M/F^*)$ and $u\not\in F^*$. Now, consider the
homomorphism $\phi:M\rightarrow F^*,x\mapsto xux^{-1}u^{-1}$. Since
$\ker\phi=C_M(u)$ we conclude that $M/C_M(u)$ is abelian and hence
$M'\subseteq C_M(u)$. But, $F(M')\neq D$ because $u\not\in F^*$. At
the other extreme, $M\subseteq N_{D^*}(F(M')^*)$ yields either
$N_{D^*}(F(M')^*)=D$ or $F(M')^*\lhd M$. However,  in the former
case the Cartan-Brauer-Hua Theorem implies that $F(M')\subseteq F$
because $F(M')\neq D$. Also, in the latter case $F(M')$ is abelian
by Theorem \ref{thms1}. Therefore, in either case we conclude that
$M'$ is abelian, as desired.

(2) By Zorn's Lemma $M$ has a maximal abelian normal subgroup $A$
containing $M'$. Now, since $M$ is maximal in $D^*$ we have either
$N_{D^*}(F(A))=D$ or $F(A)^*\subseteq M$. But, if $N_{D^*}(F(A))=D$
then by Theorem \ref{thms1} we should have $F(A)^*\subseteq
F^*\subseteq M$. Thus, in either case $F(A)^*$ is an abelian normal
subgroup of $M$ containing $A$. However, since $A$ is a maximal
abelian normal subgroup of $M$ we conclude that $F(A)^*=A$.

(3) First note that since $A\varsubsetneqq N$, $N$ is nonabelian.
Now, since $M'\subseteq N$ we have $N\lhd M$. Therefore $M\subseteq
N_{D^*}(F(N)^*)$. Thus we conclude that $F(N)^*\lhd D^*$ or
$F(N)^*\subseteq M$. But by Theorem \ref{thms1} the second case can
not occur. So $F(N)^*\lhd D^*$. Now, by the Cartan-Brauer-Hua
Theorem $F(N)=D$ because $N$ is nonabelian.
\end{proof}
Recall that a domain $R$ is called a \textit{right Ore domain} if
for every pair of nonzero elements $a,b\in R$ one has $aR\cap bR\neq
0$. Also, a \textit{left Ore domain} is defined in a similar manner.
A domain that is both right and left Ore domain is called an
\textit{Ore domain}. By a theorem of Ore, a ring $R$ has a classical
ring of quotients if and only if it is an Ore domain (see \cite[p.
146]{Pas77}). The following proposition is also helpful.
\begin{prop}[{\cite[p. 26]{ShWe86}}]\label{Propn2}
Let $R=E[G]$ be a ring where $E$ is division subring of $R$ and $G$
is a subgroup of $R^*$. If $G/E\cap G$ is an infinite cyclic group
and $R$ is a crossed product of $E$ by $G/E\cap G$ then $R$ is an
Ore domain.
\end{prop}
\begin{thm}\label{thmn3}
Let $D$ be a noncommutative division ring with center $F$. If $M$ is
a nilpotent maximal subgroup of $D^*$, then $D$ is of finite
dimension over $F$ and $M$ is abelian.
\end{thm}
\begin{proof}
By Lemma \ref{lemn1}, $M$ contains a normal abelian subgroup $A$
such that $K=A\cup\{0\}$ is a field. First suppose that no elements
of $M\setminus K^*$ is algebraic over $K$ and let $v\in M\setminus
K^*$. If we set $G=K^*\langle v^2\rangle$ then by the fact that $v$
is not algebraic over $K$ one can easily show that the ring
$F[G]=\oplus_{j\in\mathbb{Z}}Kv^{2j}$ is a crossed product of $K$ by
$G/K^*\cong \langle v^2\rangle$. But, since $\langle v^2\rangle$ is
an infinite cyclic group from Proposition \ref{Propn2} it follows
that $F[G]$ is an Ore domain. Hence the ring of quotients of $F[G]$
is $F(G)=D$ (note that $F(G)=D$ by Lemma \ref{lemn1}). Therefore,
every element of $D$ is of the form $d_1d_2^{-1}$ where $d_1,d_2\in
F[G]$ and $d_2\neq 0$. So $v=s_1s_2^{-1}$ for $s_1,s_2\in F[G]^*$.
On the other hand $s_1=\sum_{j=1}^mk_jv^{2j}$ and
$s_2=\sum_{j=1}^mk_j'v^{2j}$ where $k_j,k_j'\in K$. Hence
$\sum_{j=1}^mvk_j'v^{2j}=\sum_{j=1}^mk_jv^{2j}$. Now, if we let
$l_j=vk_j'v^{-1}$, for any $0\leq j\leq m$, then $l_j$'s lie in $K$
and we have $\sum_{j=1}^ml_jv^{2j+1}=\sum_{j=1}^mk_jv^{2j}$ which
implies that $v$ is algebraic over $K$, which is a contradiction.
Thus $M\setminus K^*$ has an element $u$ algebraic over $K$. Assume
that $u$ satisfies an equation of the form
$x^n+\sum_{j=0}^{n-1}k_jx^j=0$, where $k_j\in K$ for every $0\leq
i\leq n-1$. It is not hard to see that $R=\sum_{j=0}^{n}Ku^j$ is a
division subring of $D$ with $[R:K]_l<\infty$. Therefore $F(N)=R$
where $N=\langle K^*,u\rangle$. On the other hand Lemma \ref{lemn1}
implies that $F(N)=D$. So $[D:K]_l<\infty$ and hence $D$ is finite
dimensional over $F$ (cf. \cite[Th. 15.8]{Lam01}). Now, by Corollary
\ref{cort12} we conclude that $M$ is abelian.
\end{proof}
In fact, a more general result about the structure of maximal
nilpotent subgroups of $\operatorname{GL}_n(D)$ was provided in
\cite{Ebr04} which we are going to present here without proof.
\begin{thm}\label{thmn4}
Let $D$ be an infinite division ring and $M$ be a maximal nilpotent
subgroup of $\operatorname{GL}_n(D)$ ($n\geq1$). Then $M$ is
abelian.
\end{thm}
Here we must point out that the properties of a locally nilpotent
maximal subgroup in a division ring are also investigated in
\cite{HaiThi09} where Hai and Thin show that if $D$ is algebraic
over its center then every locally nilpotent maximal subgroup of
$D^*$ is abelian. Moreover, in another work, Hai has proved that if
$M$ is a locally nilpotent maximal subgroup of a division ring that
is algebraic over the center then it is either the multiplicative
group of some maximal subfield or it is center-by-(locally finite)
\cite{Hai11}.

Now, for the case in which the division ring is of finite dimension
over its center the following result is also proved.
\begin{thm}\label{thmn5}
Let $D$ be a noncommutative $F$-central division algebra. Let $M$ be
a nonabelian soluble maximal subgroup of $D^*$. Then $D$ contains a
maximal subfield $K$ such that $K^*\lhd M$, $K/F$ is a cyclic
extension of degree $p$ where $p$ is a prime and $M/K^*\cong
\Gal(K/F)$, i.e., $D$ is a cyclic division algebra of prime degree.
\end{thm}
\begin{proof}
By Proposition \ref{propt8} we know that $M$ is absolutely irreducible
in $D^*$ and has an abelian normal subgroup $A$ such that $F[A]$ is
a maximal subfield of $D$ that is Galois over $F$ and $M/A\cong
\Gal(F[A]/F)$. Put $K=F[A]$. Since $K^*M$ is soluble and $M$ is
maximal, we can easily conclude that $K^*M=M$ and so $K^*\lhd M$.
Also, by the maximality of $A$ in $M$ we deduce that $K^*=A$. Hence
$M/K^*\cong \Gal(K/F)$. Now, it remains to prove that $|\Gal(K/F)|$
is a prime number. Otherwise, let $|\Gal(K/F)|$ is not a prime
number. Thus $\Gal(K/F)$ has a nontrivial normal subgroup $N$. Let
$E$ be the fixed field of $N$. So we have a chain $F\subset E\subset
K$ of subfields. From field theory we know that $E$ is invariant
under each $\sigma\in\Gal(K/F)$. This forces that $M$ lies in
$N_{D^*}(E^*)$. But, if $M\neq N_{D^*}(E^*)$ by the maximality of
$M$ we must have $E^*\lhd D^*$ and by the Cartan-Brauer-Hua Theorem
$E\subseteq F$ which is a contradiction. Thus $M=N_{D^*}(E^*)$ and
so $N_{D^*}(E^*)$ is soluble. On the other hand
$C_{D^*}(E^*)\subseteq N_{D^*}(E^*)$. Therefore $C_D(E)$ is a
division subalgebra with a soluble unit group. It follows from Hua's
Theorem that $C_D(E)$ is a field. Finally, by the Centralizer
Theorem we conclude that $C_D(E)=E$, i.e., $E$ is a maximal subfield
and so $E=K$ which is a contradiction.
\end{proof}

Let $M$ be a maximal subgroup of $D^*$ where $D$ is a finite dimensional
division algebra. Also, suppose that $M$ contains no non-cyclic free subgroup. If
$\deg(D)$ is odd, then from Theorem \ref{t} it follows that $M$ is
soluble. Now, suppose that $\deg(D)=2^m$ for some $m$.
By Theorem \ref{cort14}, $M$ contains an abelian
normal subgroup $A$ such that $M/A$ has order $\deg(D)$
and hence $M$ is soluble. Here, using Theorem \ref{thmn5}
yields that $\deg(D)$ is a prime number. Thus, the following corollary
is immediate:
\begin{cor}
Let $D$ be an $F$-central division algebra and
$M$ be a maximal subgroup of $D^*$. If either $\deg(D)$ is odd
or $\deg(D)=2^m$ for some $m$, then the following statements
are equivalent:
\begin{enumerate}
  \item $M$ has no non-cyclic free subgroups;
  \item $M$ is soluble and $\deg(D)$ is a prime number.
\end{enumerate}
\end{cor}
Therefore, whenever $\deg(D)$ is not a prime and
either it is an odd number or is equal to $2^m$ for some $m$,
then every maximal subgroup of $D^*$ contains a non-cyclic
free subgroup.

A perfect description of the structure of a division ring with an
algebraic nonabelian locally soluble maximal subgroup is also
accessible that we shall provide at this place, for the proof see
\cite[Th. 6]{AkEbMoSa03}.
\begin{thm}\label{thmn6}
Let $D$ be a division ring with center $F$ and $M$ a nonabelian
locally soluble maximal subgroup of $D^*$ that is algebraic over
$F$. Then $[D:F]=p^2$, where $p$ is a prime number, and there exists
a maximal subfield $K$ of $D$ such that $K/F$ is a cyclic extension
of degree $p$, $K^*\lhd M$ and $M/K^*\cong \Gal(K/F)$.
\end{thm}
We close this section by remarking some additional results.  By a
work of Hai and Phuong Ha, the same conclusion as in Theorem
\ref{thmn6} is also valid if we replace the requirement ``$M$ is
algebraic over $F$'' by ``$M'$ is algebraic over $F$''
\cite{HaiHa10}. Also, in another work, Dorbidi, Fallah-Moghaddam and
Mahdavi-Hezavehi proved that if for a division ring $D$ the group
$\operatorname{GL}_n(D)$ has a nonabelian soluble maximal subgroup
then $n=1$, $[D:F]<\infty$ and hence a similar result as in Theorem
\ref{thmn6} holds. They also proved that if either $F$ is an
algebraically closed field or a real closed field then
$\operatorname{GL}_2(F)$ contains no soluble maximal subgroups
\cite{DorFalMah11}.
\section{Descending central series of division rings and extending of valuations}
As we have mentioned before in the introduction, one aspect in the
study of the structure of the unit group of a division ring is to
explore the roles of multiplicative commutators in the structure of
$D$. This point of view was followed in \cite{Ma95},
\cite{MaAkMeHa95}, \cite{MaAk95}, \cite{MaMo97} and \cite{ma98}.
Also, a survey on this topic was provided in \cite{Ma00}. Motivated
by various results appeared in the above papers, in \cite{Haz02} the
structure of subgroup that become visible in the descending central
series of a division ring was investigated and some analogous
results for this kind of subgroups were successfully made. At this
stage our aim is to provide some interesting results of the above
works that are natural generalizations of known old results. Here,
before giving our theorems we must recall
\begin{thm}[Wedderburn's Factorization]\label{thmd1}
Let $D$ be a division ring with center $F$ and $u\in D^*$. If $u$ is
algebraic over $F$ with minimal polynomianl
$f(t)=t^m+a_{m-1}t^{m-1}+\ldots+a_1t+a_0$, then there are
$d_1,\ldots,d_m\in D$ such that
$f(t)=(t-d_1ud_1^{-1})\ldots(t-d_mud_m^{-1})$.
\end{thm}
Surely, the following lemma is the most important tool in our
subsequent work.
\begin{lem}\label{lemd2}
Let $D$ be a division ring with center $F$ and $N$ be a normal
subgroup of $D^*$. If $u\in N$ is algebraic over $F$ and the degree
of the minimal polynomial of $u$ over $F$ is $m$, then
$u^m=N_{F(u)/F}(u)c$ for some $c\in [D^*,N]$. Thus $u^m\in
Z(N)[D^*,N]$.
\end{lem}
\begin{proof}
Let $f(t)$ be the minimal polynomial of $u$ over $F$. From field
theory, we have
\begin{equation}\label{eqd}
f(t)=t^m-Tr_{F(u)/F}(u)t^{m-1}+\ldots+(-1)^mN_{F(u)/F}(u).
\end{equation}
Now, using  Wedderburn's Factorization Theorem and (\ref{eqd}), one
obtains
$$N_{F(u)/F}(u)=d_1ud_1^{-1}\ldots d_mud_m^{-1}$$
where $d_i\in D$. But
$$d_1ud_1^{-1}\ldots d_mud_m^{-1}=
[d_1^{-1},u^{-1}]u^{-1}\ldots [d_m^{-1},u^{-1}]u^{-1}=u^mc$$ for
some $c\in [D^*,N]$. Thus we have $u^m=N_{F(u)/F}(u)c$. Now, since
$N\lhd D^*$, we have $c\in N$ and so $N_{F(u)/F}(u)\in F^*\cap
N\subseteq Z(N)$, as desired.
\end{proof}
\begin{cor}\label{cord3}
If $D$ is a division ring with center $F$ and $N$ is a subgroup of
$D^*$, then:
\begin{enumerate}
  \item If $N$ contains some $\zeta_i D^*$ and $a\in D$ is algebraic over $F$, then $u$ is
  radical over $F^*N$;
  \item If $D$ is finite dimensional over $F$ with $\deg(D)=n$ and if
  $N$ is normal in $D^*$, then $N^n\subseteq \operatorname{Nrd}_D(N)[D^*, N]\subseteq Z(N)[D^*,N]$.
\end{enumerate}
\end{cor}
\begin{proof}
(1) Clearly we can replace $N$ by $\zeta_iD^*$. The proof will be by
induction on $i$. For $i=0$, there is nothing to prove. Let
$D^*/F^*\zeta_{i-1}D^*$ be torsion. Thus, if $a\in D^*$ then there
is a nonnegative integer $r$ such that $a^r=fb$ for some $f\in F$
and $b\in\zeta_{i-1}D^*$. Now, by Lemma \ref{lemd2} we have $b^m\in
F^*\zeta_iD^*$ for some $m$. Thus $a^{rm}=f^mb^m\in F^*\zeta_iD^*$,
as desired.

(2) Let $a\in N$ and $K$ be a maximal subfield of $D$ containing
$a$. Put $[K:F(a)]=k$. By Lemma \ref{lemd2} we conclude that
$a^{[F(a):F]}=N_{F(a)/F}(a)c$ for some $c\in [D^*,N]$. Now, we have
\begin{align}\nonumber
a^n=a^{k[F(a):F]}&=N_{F(a)/F}(a)^kc^k\\ \nonumber
&=N_{F(a)/F}(a^k)c^k \\ \nonumber &=N_{F(a)/F}(N_{K/F(a)}(a))c^k\\
\nonumber &=N_{K/F}(a)c^k \\ \nonumber
&=\operatorname{Nrd}_D(a)c^k\in \operatorname{Nrd}_D(N)[D^*, N].
\end{align}
\end{proof}
\begin{cor}\label{cor765}
Let $D$ be a division algebra of index $n$. For any $i>0$, the
quotient $\zeta_iD^*/\zeta_{i+1}D^*$ is a torsion abelian group of
bounded exponent $n$.
\end{cor}
\begin{proof}
Since $D'$ is contained in the kernel of $\operatorname{Nrd}_D$, if
in Corollary \ref{cord3} we put $N=\zeta_iD^*$ then we have
$(\zeta_iD^*)^n\subseteq \zeta_{i+1}D^*$.
\end{proof}
\begin{thm}\label{thmd4}
Let $D$ be a division ring with center $F$. If $\zeta_i D^*$ is
algebraic over $F$ for some $i$, then $D$ is algebraic over $F$.
\end{thm}
\begin{proof}
If $D=F$ then there is nothing to prove. Thus we may assume that
$D\neq F$. Put $H=F^*\zeta_i D^*$. Since $D^*$ is not nilpotent we
have $\zeta_i D^*\nsubseteq F^*$ and so $F^*\subsetneq H$. Now,
consider the ring $R$ generated by the elements of $H$. First we
claim that $R$ is algebraic over $F$. For, if $a\in \zeta_i D^*$ and
$b\in D$ is algebraic over $F$, then $\overline{ab}=\overline{ba}$
in the group $D^*/F^*\zeta_{i+1}D^*$. On the other hand by Corollary
\ref{cord3}, $\overline{a}$ and $\overline{b}$ are torsion and so
$\overline{ab}$ is torsion. Therefore, $(ab)^k\in H$ for some $k$
and hence $ab$ is algebraic over $F$. Now, if $u,v\in H$ we have
$u+v=u(1+u^{-1}v)$. But $u$ and $1+u^{-1}v$ are algebraic over $F$.
It follows that $u+v$ is algebraic over $F$. Hence $R$ is algebraic
over $F$ and the claim is established. Now, since $R$ is algebraic
over $F$ we conclude that $R$ is a division subring of $D$. Clearly,
$R^*\lhd D^*$ and so, by the Cartan-Brauer-Hua Theorem, $R=D$, as
desired.
\end{proof}
\begin{cor}\label{cord5}
Let $D$ be a division ring with center $F$. If $\zeta_i D^*$ is
radical over $F$ for some $i$, then $D$ is commutative.
\end{cor}
\begin{proof}
From Theorem \ref{thmd4} it follows that $D$ is algebraic over $F$.
Therefore, by Corollary \ref{cord3}, $D$ is radical over
$F^*\zeta_iD^*$ and so $D^*$ is radical over $F^*$. Now, the result
follows by Kaplansky's Theorem \cite{Kap51}.
\end{proof}
\begin{cor}\label{cord6}
If $D$ is a noncommutative division ring with center $F$, then every
$\zeta_iD^*$ contains an element separable over $F$.
\end{cor}
\begin{proof}
If every element of $\zeta_iD^*$ is purely inseparable over $F$,
then $\zeta_iD^*$ is radical over $F$ and hence $D$ is necessarily
commutative, which is a contradiction.
\end{proof}
Let $D$ be a division ring and $K$ be a proper division subring of
$D$. By a result of Hua (\cite{Hua49}) if $\zeta_iD^*\subseteq K$
then $D$ is commutative. The following theorem generalizes this
phenomenon.
\begin{thm}\label{thmd7}
Let $D$ be a division ring with center $F$ and let $K$ be a proper
division subring of $D$. If $\zeta_iD^*$ is radical over $K$ for
some $i$, then $D$ is a field.
\end{thm}
\begin{proof}
Let $\zeta_iD^*$ be radical over $K$ for some $i$. If $\zeta_iD^*=1$
then $D^*$ is nilpotent and thus there is nothing to prove. So we
may consider the case in which $\zeta_iD^*\neq 1$. We prove the
theorem in three steps.

\textbf{Step 1}. We show that if $a, b \in \zeta_iD^*$, $1\neq a\in
K$ and $b\not\in K$ then there is a positive integer $n$ such that
$ba^n=a^nb$. To see this, consider the elements
$u_1=(b+a)^{-1}a(b+a)$ and $u_2=(b+1)^{-1}a(b+1)$. Clearly $u_1,
u_2$ are in $\zeta_iD^*$. Now, since $\zeta_iD^*$ is radical over
$K$, there is a positive integer $n$ such that
$u_1^n=(b+a)^{-1}a^n(b+a)\in K$ and $u_2^n=(b+1)^{-1}a^n(b+1)\in K$.
But a simple calculation shows that
$b(u_1^n-u_2^n)=a^n(a-1)+u_2^n-au_1^n\in K$. Therefore, if
$u_1^n-u_2^n\neq 0$ then $b\in K$ which is absurd. So we have
$u_1^n-u_2^n=0$ and hence $(a-1)a^n=(a-1)u$ where $u=u_1^n=u_2^n$.
However, $a\neq 1$ implies that $u=a^n$, $(b+1)a^n=a^n(b+1)$ and so
$ba^n=a^nb$.

\textbf{Step 2}. Here we claim that for every pair $u,v\in
\zeta_iD^*\cap K$, there is some positive integer $m$ such that
$v^mu=uv^m$. If $\zeta_iD^*\subseteq K$, then Cartan-Brauer-Hua
Theorem implies that $D=F(\zeta_iD^*)=K$ which is impossible. So
$\zeta_iD^*\nsubseteq K$ and thus we can choose an element $a\in
\zeta_iD^*$ such that $a\not\in K$. Hence $au\in\zeta_iD^*$ and
$au\not\in K$. Now, by Step 1 there are positive integers $r,s$ such
that $v^ra=av^r$ and $v^sau=auv^s$. This forces that
$v^{rs}u=uv^{rs}$ and the claim is established.

\textbf{Step 3}. Here, we prove the theorem. Let $u,v\in \zeta_iD^*$
and $L=F(u,v)$ denote the division subring generated by $u,v$ and
$F$. Suppose that $a\in \zeta_iL^*$ is arbitrary. Since
$\zeta_iL^*\subseteq \zeta_iD^*$, we have $a^t\in K$ for some
positive integer $t$. Now, by Step 1 and Step 2 we conclude that
there is a positive integer $n$ such that $a^nu=ua^n$ and
$a^nv=va^n$. This shows that $a^n\in Z(L)$. So $\zeta_iL^*$ is
center-by-periodic and hence $L$ is commutative by Corollary
\ref{cord5}. Therefore, $uv=vu$ and thus $\zeta_iD^*$ is
commutative. Now, from Corollary \ref{cord5} the result follows.
\end{proof}
Let $\Gamma$ be an additive totally ordered abelian group. Let
$\Delta$ be a set properly containing $\Gamma$ and
$\infty\in\Delta\setminus\Gamma$. Extend the operation of $\Gamma$
to $\Gamma\cup\{\infty\}$ by
\begin{center}
$x+\infty=\infty+x=\infty$ for all $x\in \Gamma\cup\{\infty\}$.
\end{center}
Also, extend the ordering of $\Gamma$ to $\Gamma\cup\{\infty\}$ by
$x<\infty$ for all $x\in \Gamma$. Let $D$ be a division algebra over
its center $F$, of index $n$. By a valuation on $D$ with values in
$\Gamma$, we mean a map $v:D\rightarrow \Gamma\cup\{\infty\}$
satisfying, for all $a,b\in D$,
$$
  \begin{array}{ll}
    (\textrm{V}1) & \hbox{$v(a)=\infty$\ \ \  iff\ \ \  $a=0$;} \\
    (\textrm{V}2) & \hbox{$v(ab)=v(a)+v(b)$;} \\
    (\textrm{V}3) & \hbox{$v(a+b)\geq \min\{v(a),v(b)\}$.}
  \end{array}
$$
The standard reference for noncommutative valuation theory is
Schilling's book \cite{Schi50}. Let $K\subseteq D$ be a division
ring extension. Let $w$ and $v$ are valuations on $K$ and $D$ with
values in $\Gamma\cup\{\infty\}$, respectively. We say that $v$ is
an extension of $w$ whenever $v|_K=w$. It is well-known that if
$F\subseteq K$ is a field extension, then every valuation on $F$ has
an extension to $K$ (possibly many different extensions). On the
other hand, in the setting of noncommutative division rings this
property does not always hold. However, a criterion for when a
valuation can be extended from $F=Z(D)$ to $D$ has been given
independently by Wadsworth (\cite{Wad86}) and Ershov (\cite{Er83})
for the case in which $[D:F]<\infty$. More precisely, they showed
that if $D$ is an $F$-central division algebra, then the following
are equivalent:
\begin{itemize}
  \item[(i)] a valuation $v$ on $F$ extends to $D$;
  \item[(ii)] $v$ has a unique extension to each subfield $K$ with
$F\subseteq K\subseteq D$.
\end{itemize}
But, a careful reading of Wadsworth's proof shows that
(1)$\Rightarrow$(2) is valid if $D$ is algebraic over $F$. Motivated
by this observation, in \cite{Ma94} the following generalization of
the above mentioned result was provided.
\begin{thm}\label{thmd8}
Let $D$ be a division ring algebraic over its center $F$, and let
$v$ be a valuation on $F$. Then the following are equivalent:
\begin{enumerate}
  \item $v$ extends to a valuation on $D$;
  \item $v$ has a unique extension to each subfield $K$ of $D$ with
  $[K:F]<\infty$ and $Z(D')\subseteq\{u\in F|v(u)=0\}$.
\end{enumerate}
\end{thm}
To prove Theorem \ref{thmd8} we need the following lemma which is a
direct outcome of Lemma \ref{lemd2}.
\begin{lem}\label{lemd9}
Let $D$ be a division ring with center $F$. Let $u, v$ and $uv$ in
$D^*$ be algebraic over $F$ and denote by $K_u, K_v$ and $K_{uv}$
subfields of $D$ containing $u, v$ and $uv$ with $[K_u:F]=r$,
$[K_v:F]=s$ and $[K_{uv}:F]=t$, respectively. Then
$$N_{K_u/F}(uv)^{d/t}=N_{K_u/F}(u)^{d/r}N_{K_v/F}(v)^{d/s}c$$
for some $c\in Z(D')$, where $d$ is the least common multiple of $r,
s$ and $t$.
\end{lem}
\begin{proof}
By Lemma \ref{lemd2} we have
$\overline{N_{K_u/F}(u)}=\overline{u}^r$,
$\overline{N_{K_u/F}(v)}=\overline{v}^s$ and
$\overline{N_{K_u/F}(uv)}=\overline{uv}^t=
\overline{u}^t\overline{v}^t$ in the Whitehead group
$K_1(D)=D^*/D'$. Thus
$\overline{N_{K_u/F}(uv)}^{d/t}=\overline{N_{K_u/F}(u)}^{d/r}\overline{N_{K_u/F}(v)}^{d/s}$
and the result follows.
\end{proof}
\begin{proof}[Proof of Theorem \ref{thmd8}]
(1)$\Rightarrow$(2). By the remark before the theorem, $v$ has a
unique extension to each subfield $K$. Let $w$ be the valuation on
$D$ such that $w|_F=v$. Since $D$ is algebraic over $F$, one can
easily conclude that the value group $\Gamma_D:=w(D^*)$ is abelian.
Thus $w(a)=0$ for every $a\in D'$.

(2)$\Rightarrow$(1). Let $\Gamma_F:=v(F^*)$ be the value group of
$v$. Let $\Delta\cong \Gamma_F\otimes_{\mathbb{Z}}\mathbb{Q}$ be the
divisible hull of $\Gamma_F$. The total ordering on $\Gamma_F$
extends uniquely to a total ordering on $\Delta$, and for each
algebraic extension $L/F$ and each extension $w$ of $v$, we may view
$w(L^*)$ as a subgroup of $\Delta$. Now, for each $a\in D^*$ put
$K_a=F(a)$ and define the function $w:D^*\rightarrow \Delta$ by
$$w(a)=\frac{1}{n}v(N_{K_a/F}(a)),$$
where $n=[K_a:F]$. First we verify that $w|_{K_a}$ is the valuation
on $K_a$ extending $v$. Let $N$ be the normal closure of $K_a$ over
$F$ and $u:N\rightarrow \Delta$ be any valuation on $N$. If
$a_1,\ldots,a_n$ are the roots of the minimal polynomial of $a$ over
$F$, we know that all of them lie in $N$ and $N_{K_a/F}(a)=a_1\ldots
a_n$. For each $j$ there is an $F$-automorphism $\sigma_i$ of $N$
sending $a$ to $a_j$. Since $u|_{K_a}$ and $(u\circ
\sigma_j)|_{K_a}$ are each valuations on $K_a$ extending $v$, by
hypothesis they are the same. So, $u(b_j)=u(\sigma_j(b))=u(b)$. Thus
we have
$$w(a)=\frac{1}{n}v(N_{K_a/F}(a))=\frac{1}{n}v(a_1\ldots a_n)$$
$$=\frac{1}{n}(u(b_1)+\ldots+u(b_n))=u(b).$$
Thus, $w|_{K_a}=u|_{K_a}$ which is the valuation on $K_a$ extending
$v$. To show that $w$ is a valuation on all of $D$, take any $a,b\in
D^*$. Now, by Lemma \ref{lemd9} we have
\begin{align}\nonumber
w(ab)&=\frac{1}{t}v(N_{F(ab)/F}(ab))=\frac{1}{d}v(N_{F(ab)/F}(ab)^{d/t})\\
\nonumber &=\frac{1}{d}v(N_{F(a)/F}(a)^{d/r}N_{F(b)/F}(b)^{d/s}c)\\
\nonumber
&=\frac{1}{r}v(N_{F(a)/F}(a))+\frac{1}{s}v(N_{F(b)/F}(b))+\frac{1}{d}v(c)=w(a)+w(b).
\end{align}
Finally, assume that $b\neq -a$. Since $w|_{F(ab)}$ is a valuation
we have $w(1+a^{-1}b)\geq\min\{w(1),w(a^{-1}b)\}$. Thus, using the
multiplicative property for $w$,
\begin{align}\nonumber
w(a+b)&=w(a)+w(1+a^{-1}b)\geq w(a)+\min\{w(1),w(a^{-1}b)\}\\
\nonumber &=\min\{w(a),w(b)\},
\end{align}
as desired.
\end{proof}
\begin{rem}\label{rem}
If in Theorem \ref{thmd8}, $D$ is of finite dimensional over $F$
with index $n$ then one can easily check that for every $a\in D^*$,
$$\frac{1}{[F(a):F]}v(N_{F(a)/F}(a))=\frac{1}{n}v(N_{K/F}(a))$$
where $K$ is any maximal subfield containing $a$. Thus, in this case
we have
$$w(a)=\frac{1}{n}v(\operatorname{Nrd}_D(a)).$$
This is exactly the Wadsworth-Ershov formula for the extended
valuation.
\end{rem}

Here, we need to recall some definitions from the theory of valued
division algebras. Let $D$ be a valued $F$-central division
algebra with value group $\Gamma_D$. We denote the \textit{valuation
ring} of $D$ by $V_D=\{d\in D^*|\ v(d)\geq 0\}\cup\{0\}$, its unique
maximal ideal by $M_D=\{d\in D^*|\ v(d)>0\}\cup\{0\}$, and its
\textit{residue class division ring} by $\overline{D}$. Set
$U_D=\{d\in D^*|\ v(d)=0\}$ so that $U_D=V^*_D$. The restriction of
$v$ to $F^*$ is denoted by $w$ which defines a valuation with value
group $w(F^*)=\Gamma_F$. Also $V_F,\ M_F,\ \overline{F}$, and $U_F$
are defined similarly. Since $V_F\cap M_D=M_F$, we can consider the
residue class field $\overline{F}$ as a subalgebra of
$\overline{D}$. The valuation $w$ over $ F $ is called
\textit{Henselian} if $w$ has a unique extension to any algebraic
extension field of $F$. So by Theorem \ref{thmd8} if $F$ is
Henselian then the valuation of $F$ has a unique extension to $D$.
In this setting, $D$ is called a \textit{tame} division algebra if
$Z(\overline{D})/\overline{F}$ is separable and
$\operatorname{char}(\overline{F})\nmid n$.

Now, we return to our study of descending central series of a
division algebra by applying Lemma \ref{lemd2} to find a short and
elementary proof for Platonov's Congruence Theorem that is the key
step in reduced $K$-Theory connecting $\SK(D)$ to
$\SK(\overline{D})$ (See Section \ref{SecR}). The method that we are
going to follow was given in \cite{Haz02}.
\begin{thm}[Congruence Theorem]\label{thmd9}
Let $D$ be a tame division algebra over a Henselian field $F=Z(D)$,
of index $n$. Then $1+M_D=(1+M_F)[D^*,1+M_D]$ and $(1+M_D)\cap
D^{(1)}=[D^*,1+M_D]$, where $D^{(1)}$ is the kernel of the reduced
norm map.
\end{thm}
\begin{proof}
First we claim that $(1+M_F)\cap D^{(1)}=1$. To prove our claim let
$1+f\in (1+M_F)\cap D^{(1)}$. Thus we have $(1+f)^n=1$. But
\begin{align}\label{eqd1}
v(0)=v((1+f)^n-1)&=v(nf+(n(n-1)/2)f^2+\ldots+f^{n})
\\ \nonumber
&=v(f)+v(n+(n(n-1)/2)f+\ldots+f^{n-1}).
\end{align}
Now, since $D$ is tame $\overline{n}\neq 0$ in $\overline{F}$ and so
$n\in U_F$, that is $v(n)=0$. On the other hand, since $v(f)>0$, we
have $v((n(n-1)/2)f+\ldots+f^{n-1})>0$. Therefore
\begin{align}\nonumber
v(n+(n(n-1)/2)f+\ldots+&f^{n-1})=
\\ \nonumber \min \{v(n), &v((n(n-1)/2)f+\ldots+f^{n-1})\}=0.
\end{align}
Hence by (\ref{eqd1}) we have $v(0)=v(f)$ and from (V1) it follows
that $f=0$ and so our claim. Now, in Corollary \ref{cord3} if we
take $N=1+M_D$ then we obtain
$$(1+M_D)^n\subseteq ((1+M_D)\cap F^*)[D^*,1+M_D].$$
Since the valuation is tame and Henselian, the Hensel's Lemma
implies that $(1+M_D)$ is $n$-divisible. Therefore
$1+M_D=(1+M_F)[D^*,1+M_D]$. Now, using the fact that $(1+M_F)\cap
D^{(1)}=1$, the theorem follows.
\end{proof}
Using the Congruence Theorem and results of \cite{Er83} and the fact
that if $D$ is totally ramified then
$\overline{D'}\cong\mu_e(\overline{F})$ where
$e=\exp({\Gamma_D/\Gamma_F})$ (see Proposition 3.1 of \cite{TW87}).
In \cite{LT93}, Tignol and Lewis presented an explicit formula for
$\SK(D)$. In fact they have shown that if $D$ is a tame division
algebra over its center $F$ of degree $n$, then $\SK(D)\cong
\mu_n(\overline{F})/\mu_e(\overline{F})$. We will obtain this result
in the next section as a particular case of a more general setting.
\begin{cor}\label{cord11}
Let $D$ be a tame and Henselian division algebra. Then
$[D^*,[D^*,1+M_D]]=[D^*,1+M_D]$, i.e., $[D^*,1+M_D]$ is
$D^*$-perfect. In particular, $[D^*,1+M_D]\subseteq \zeta_iD^*$ for
all $i>0$.
\end{cor}
\begin{thm}\label{thmd13}
Let $F$ be a Henselian field and $D$ be a tame and totally
ramified $F$-central division algebra of index $n$. Then
\begin{enumerate}
  \item $\zeta_iD^*=\zeta_{i+1}D^*$ for all $i\geq 2$;
  \item $\zeta_1D^*/\zeta_2D^*\cong \mathbb{Z}_e$, where
  $e=\exp(\Gamma_D/\Gamma_F)$.
\end{enumerate}
\end{thm}
\begin{proof}
(1) Since $D$ is totally ramified, we have
$\overline{D}=\overline{F}$ and so $U_D=U_F(1+M_D)$. But
$D'\subseteq U_D$ and hence $D'\subseteq U_F(1+M_D)$. This yields
$\zeta_2D^*\subseteq [D^*,1+M_D]$. Now, using Corollary \ref{cord11}
implies that $\zeta_2D^*=[D^*,1+M_D]$ and thus $\zeta_2D^*$ is
$D^*$-perfect. Therefore $\zeta_2D^*=\zeta_iD^*$ for all $i\geq 2$.

(2) Consider the reduction map
$$U_D/1+M_D\rightarrow\overline{D}^* \ ,\ a(1+M_D)\mapsto\overline{a}.$$
The restriction of this map to $D'$ gives an isomorphism
$$\frac{D'}{D'\cap(1+M_D)}\cong \overline{D'}.$$
Now, from the Congruence Theorem, the equality
$\zeta_2D^*=[D^*,1+M_D]$ and the fact that
$\overline{D'}\cong\mathbb{Z}_e$ it follows that
$\zeta_1D^*/\zeta_2D^*\cong\mathbb{Z}_e$.
\end{proof}
\begin{thm}\label{thmd14}
Let $F$ be a Henselian field and $D$ be a tame and
unramified $F$-central division algebra of index $n$. Then
\begin{enumerate}
  \item $\zeta_iD^*/\zeta_{i+1}D^*\cong\zeta_i\overline{D}^*/\zeta_{i+1}\overline{D}^*$,
  for all $i\geq 1$.
  \item $[D^*,1+M_D]\subsetneq \zeta_iD^*$, for any $i\geq 1$;
\end{enumerate}
\end{thm}
\begin{proof}
(1) We first show that
$\overline{\zeta_iD^*}=\zeta_i\overline{D}^*$. To do so, consider
the following exact sequence:
$$
 \xymatrix{
1\ar[r]&F^*U_D/F^*\ar[r]&D^*/F^*\ar[r]^v&\Gamma_D/\Gamma_F\ar[r]&1.}
$$
Since $\Gamma_D=\Gamma_F$, the above sequence implies that
$D^*=F^*U_D$. Therefore, if $a\in D^*$ and $b\in\zeta_iD^*$ then the
element $a^{-1}b^{-1}ab$ may be written in the form
$\alpha^{-1}\beta^{-1}\alpha\beta$ where $\alpha,\beta\in U_D$. This
shows that $\overline{\zeta_iD^*}=\zeta_i\overline{D}^*$. On the
other hand from the Congruence Theorem and Corollary \ref{cord11} we
have
$$[D^*,1+M_D]\subseteq (1+M_D)\cap\zeta_iD^*\subseteq [D^*,1+M_D],$$
and so $(1+M_D)\cap\zeta_iD^*=[D^*,1+M_D]$. Thus the restriction of
the reduction map on the subgroup $\zeta_iD^*$ gives rise to an
isomorphism
\begin{equation}\label{eqd3}
\frac{\zeta_iD^*}{[D^*,1+M_D]}\cong \zeta_i\overline{D}^*.
\end{equation}
Therefore
$\zeta_iD^*/\zeta_{i+1}D^*\cong\zeta_i\overline{D}^*/\zeta_{i+1}\overline{D}^*$.

(2) By (\ref{eqd3}) we have an isomorphism
$$\frac{D'}{[D^*,1+M_D]}\cong \overline{D}'.$$
Also, by Corollary \ref{cord5}, $\overline{D}'$ is not a torsion
group and so $D'/[D^*,1+M_D]$ is not torsion. But, by Corollary
\ref{cor765} the group $D'/\zeta_iD^*$ is a torsion group and by
Corollary \ref{cord11}, $[D^*,1+M_D]\subseteq \zeta_iD^*$. This
implies that $[D^*,1+M_D]\subsetneq \zeta_iD^*$.
\end{proof}
\section{The functor $\CK$ and Reduced $K$-Theory}\label{SecR}
As is known, some useful tools that are used to explore the
structure of $K_1$ groups of simple Artinian rings are
``Determinant-like'' functions. Namely, if $A=\operatorname{M}_n(D)$
is a simple Artinian ring ($D$ is a division ring) then thanks to
the Dieudonn$\acute{\textrm{e}}$ determinant we have $K_1(A)\cong
K_1(D)$. Thus, the computation of $K_1(A)$ reduces to that of
$K_1(D)$. But in the case that $D$ is of finite dimension over its
center $F$, another determinant-like function is the reduced norm
map $\operatorname{Nrd}_D:D^*\rightarrow F^*$. Now, since $F^*$ is
commutative and $\operatorname{Nrd}_D$ is a multiplicative map, all
multiplicative commutators in $D^*$ are killed by
$\operatorname{Nrd}_D$ and thus we have the following exact sequence
$$
\xymatrix{ 1\ar[r]& \SK(D)\ar[r]&K_1(D)\ar[r]^{\
\operatorname{Nrd}_D}& K_1(F) }
$$
where $\SK(D)=\{a\in D^*|\operatorname{Nrd}_D(a)=1\}/D'$. In
comparison with the properties of usual determinant on matrix groups
over commutative fields, it is natural to ask if in general
$\SK(D)=1$. The question of the triviality of $\SK(D)$ was open for
many years and was called the ``Tanaka-Artin'' problem. Before 1975,
the time that Platonov gave the first examples of division algebras
with non-trivial $\SK$, the general feeling was that $\SK$ is
trivial for every division algebra and mostly the interested
researchers were taking steps in directions to prove this problem
positively. For example, in 1943 Nakayama and Matsushima showed that
if $F$ is a local field then $\SK(D)=1$ (cf. \cite{NaMa43}). Also,
in 1950 \cite{Wa50} Wang proved that if the index of $D$ is square
free or if $F$ is an algebraic number field, then $\SK(D)=1$.
However, the methods used by Platonov to compute $\SK(D)$ are called
``Reduced $K$-Theory'' these days. For a survey on $\SK$ and the
results concerning the computation of this group after Platonov see
\cite[\S 6]{Wad99} and references therein. Now, let us review some
interesting characteristics of the group $\SK$ below. (See
\cite{Mer95} for the complete list of the properties of $\SK$ and
\cite{Dra83} for the proofs.) Let $F$ be a field and $D$ be an
$F$-central division algebra. Then
\begin{itemize}
  \item For any field extension $L/F$ one has a homomorphism $\SK(D)\rightarrow \SK(D\otimes_FL)$.
  On the other hand $\SK$ enjoys a transfer map, that is, for each
  finite extension $L/F$ there exists a norm homomorphism $\SK(D\otimes_FL)\rightarrow \SK(D)$.
  Now, since $\SK(\operatorname{M}_n(L))=1$, one can deduce that $\SK(D)$ is a torsion abelian group
  of bounded exponent $\operatorname{ind}(D)$ and if the degree $[L:F]$ is relatively prime to
  $\operatorname{ind}(D)$, then $\SK(D)\hookrightarrow \SK(D\otimes_FL)$.
  \item If $D=D_1\otimes_FD_2$ and
  $\operatorname{ind}(D_1)$ and $\operatorname{ind}(D_2)$ are relatively prime, then
  $\SK(D)\cong \SK(D_1)\times \SK(D_2)$.
  \item In the case of a valued division algebra $\SK$ is stable, namely $\SK(D)\cong
  \SK(\overline{D})$, where $D$ is an unramified division algebra. Indeed $\SK(D((x)))\cong \SK(D)$.
  \item $\SK$ is stable under purely transcendental extensions, i.e., if
  $L/F$ is a purely transcendental extension then $\SK(D\otimes_FL)\cong \SK(D)$.
  Indeed $\SK(D(x))\cong \SK(D)$.
\end{itemize}
Now in contrast, for an $F$-central simple algebra $A$ consider the
map $K_1(F)\rightarrow K_1(A)$ which is induced by the inclusion and
consider the group
$$\CK(A)=\operatorname{Coker}(K_1(F)\rightarrow K_1(A))\cong A^*/F^*A'.$$
The first systematic investigation of the algebraic properties of
$\CK$ for a division algebra $D$ was initiated in \cite{HaMaHMi99}
and then was continued in \cite{Haz01}. The main point of view in
the above mentioned papers was showing that the algebraic properties
of $\CK(D)$ is closely related to those of $\SK(D)$ and some
functorial properties of $\SK(D)$ can be carried over $\CK(D)$.
Also, one can easily see that the following sequence is exact:
$$
\xymatrix{ 1\ar[r]&\displaystyle{\frac{\mu_n(F)}{Z(D')}}\ar[r]^f&
\SK(D)\ar[r]^g&\CK(D)\ar[r]^{\nu}&
\displaystyle{\frac{\operatorname{Nrd}_D(D^*)}{F^{*n}}}\ar[r]&1, }
$$
where $n=\operatorname{ind}(D)$, $f,g$ are canonical homomorphisms
and $\nu$ is induced by the reduced norm map. Therefore, we have
$\SK(D)\cong \mu_n(F)/Z(D')$ if and only if $\CK(D)\cong
\operatorname{Nrd}_D(D^*)/F^{*n}$. This observation can be used to
open a new way for the computation of $\SK(D)$ as we shall see
later. Here, we are going to review some $\SK$-like properties of
$\CK$. At first, we observe that if $A=\operatorname{M}_n(D)$ is a
central simple algebra then  thanks to the
Dieudonn$\acute{\textrm{e}}$ determinant, one has $\CK(A)\cong
D^*/F^{*n}D'$. Thus, the functor $\CK$ from the category of central
simple algebras to the category of abelian groups is not Morita
invariant. (Recall that for the functor $\SK$ we have $\SK(A)\cong
\SK(D)$.) Now, Corollary \ref{cord3} implies that $\CK(D)$ is
abelian torsion of bounded exponent $\operatorname{ind}(D)$.
Therefore, $\CK(A)$ is abelian torsion of bounded exponent
$\deg(A)$.

Let $D$ be an $F$-central division algebra and $A$ be a simple
Artinian ring such that $F\subseteq Z(A)$ and $[A:F]<\infty$. The
natural embedding of $D$ in $D\otimes_FA$ ($a\mapsto a\otimes 1$)
induces a group homomorphism $\mathcal{I}:\CK(D)\rightarrow
\CK(D\otimes_FA)$. The following proposition provides a reverse map.
\begin{prop}[Transfer Map]\label{propf1}
Let $D$ be an $F$-central division algebra and $A$ be a simple Artinian
ring such that $F\subseteq Z(A)$ and $[A:F]<\infty$. Then there is a homomorphism
$\mathcal{P}:\CK(D\otimes_FA)\rightarrow \CK(D)$ such that
$\mathcal{P}\mathcal{I}=\eta_m$, where $m=[A:F]$ and $\eta_m(x)=x^m$.
\end{prop}
\begin{proof}
Consider the following sequence when we tensor with $D$ over $F$,
\begin{equation}\label{eqf1}
\xymatrix{ D\ar[r]& D\otimes_FA\ar[r]^{1\otimes l\ \ \ \ }&
D\otimes_F\operatorname{M}_m(F)\ar[r]^{\ \ \ \cong}&
\operatorname{M}_m(D),}
\end{equation}
where $l:A\rightarrow \operatorname{M}_m(F)$ is the left regular
representation. If $a\in D$ then pushing $a$ along the sequence
\ref{eqf1} we arrive at $aI_m$ where $I_m$ is the identity matrix.
Now, the above maps and the Dieudonn$\acute{\textrm{e}}$ determinant
yield the following sequence of abelian group homomorphisms
\begin{equation}\label{eqf2}
\xymatrix{
\CK(D)\ar[r]^{\mathcal{I}\ \ \ \ }& \CK(D\otimes_FA)\ar[r]^{\mathcal{J}}&
\CK(\operatorname{M}_m(D))\ar[r]^{\ \ \ \delta}&
\CK(D).
}
\end{equation}
where $\mathcal{J}(\overline{a})=\overline{aI_m}$ and
$\delta$ is induced by the Dieudonn$\acute{\textrm{e}}$ determinant.
Now, one can observe that $\delta\mathcal{J}\mathcal{I}=\eta_m$. Thus if
we let $\mathcal{P}=\delta\mathcal{J}$ then we are done.
\end{proof}
The following corollary shows that the analogous result for the
behavior of $\SK$ under the extension of the ground field holds for
$\CK$ too.
\begin{cor}\label{corf2}
Let $D$ be an $F$-central division algebra and $A$ be a simple
Artinian ring such that $F\subseteq Z(A)$ and $[A:F]<\infty$. If
$\operatorname{ind}(D)$ and $[A:F]$ are relatively prime, then the
canonical homomorphism $\mathcal{I}:\CK(D)\rightarrow
\CK(D\otimes_FA)$ is injective.
\end{cor}
\begin{proof}
Let $\operatorname{ind}(D)=n$ and $[A:F]=m$. Suppose
$\mathcal{I}(x)=1$ for some $x\in \CK(D)$. By \ref{propf1},
$x^m=\mathcal{P}\mathcal{I}(x)=1$. But $\CK(D)$ is of bounded
exponent $n$. Hence $x^n=1$. Now, since $m$ and $n$ are relatively
prime, we conclude that $x=1$.
\end{proof}
\begin{thm}\label{thmf3}
Let $A$ and $B$ be $F$-central division algebras such that
$\operatorname{ind}(A)$ and $\operatorname{ind}(B)$ are relatively
prime. Then $\CK(A\otimes_FB)\cong \CK(A)\times \CK(B)$.
\end{thm}
\begin{proof}
Since $\CK(A\otimes_FB)$ is abelian of bounded exponent $mn$ where
$m=\operatorname{ind}(A)$ and $n=\operatorname{ind}(B)$ we have that
$\CK(A\otimes_FB)\cong \mathcal{H}\times\mathcal{K}$ where
$\exp(\mathcal{H})|m$ and $\exp(\mathcal{K})|n$. Consider the
following diagram of mappings
$$
\xymatrix{
\CK(A)\ar[d]^{\mathcal{I}_1}&&&\\
\CK(A\otimes_FB)\ar[d]^{\mathcal{I}_2}&&&\\
\CK(A\otimes_FB\otimes_FB^{op})\ar[r]^{\ \ \ \ \
\phi}\ar[d]^{\mathcal{I}_3}&
\CK(\operatorname{M}_{n^2}(A))\ar[r]^{\ \ \ \delta}&\CK(A)\\
\CK(A\otimes_FB\otimes_FB^{op}\otimes_FB)\ar[r]^{\ \ \ \ \ \
\widehat{\phi}}& \CK(\operatorname{M}_{n^2}(A\otimes_FB)) \ar[r]^{\
\ \ \widehat{\delta}}&\CK(A\otimes_FB),& }
$$
where $\mathcal{I}_1,\mathcal{I}_2$ and $\mathcal{I}_3$ are induced
by the natural embedding, $\phi,\widehat{\phi}$ are induced by the
usual isomorphism and $\delta,\widehat{\delta}$ are induced by the
Dieudonn$\acute{\textrm{e}}$ determinant. Now, by Proposition
\ref{propf1} we have
$\delta\phi\mathcal{I}_2\mathcal{I}_1=\eta_{n^2}$. But
$\operatorname{ind}(A)$ and $n$ are coprime. So
$\delta\phi\mathcal{I}_2\mathcal{I}_1$ is an isomorphism. Also
$\mathcal{I}_1(\CK(A))\subseteq \mathcal{H}$ as $\CK(A)$ is
$m$-torsion. Thus we conclude that
$\delta\phi\mathcal{I}_2|_{\mathcal{H}}$ is surjective. Now, we show
that $\delta\phi\mathcal{I}_2|_{\mathcal{H}}$ is injective. To
prove this, first note that if $1\neq w\in\mathcal{H}$ then
$\widehat{\delta}\widehat{\phi}\mathcal{I}_3\mathcal{I}_2(w)=\eta_{n^2}(w)
=w^{n^2}\neq 1$ and so $\mathcal{I}_2|_{\mathcal{H}}$ is injective.
Suppose $u\in\mathcal{H}$ is such that
$\delta\phi\mathcal{I}_2(u)=1$. Then
$\phi\mathcal{I}_2(u)\in\ker\delta$ and it is not hard to see that
$\phi\mathcal{I}_2(u)^{n^2}=1$. Thus $\phi\mathcal{I}_2(u^{n^2})=1$
and hence $u^{n^2}=1$ as $\phi$ and $\mathcal{I}_2|_{\mathcal{H}}$
are injective. On the other hand $u^{m}=1$. This forces that $u=1$,
in other words $\CK(A)\cong \mathcal{H}$. Similarly, it can be shown
that $\CK(B)\cong \mathcal{K}$. This completes the proof.
\end{proof}
Here, we have to point out that the same functorial property as in
Theorem \ref{thmf3} also holds for central simple algebras. In
other words, if $A$ and $B$ are $F$-central simple algebras with
coprime degrees then $\CK(A\otimes_FB)\cong \CK(A)\times \CK(B)$
(cf. \cite{HazWil06}).
In the next proposition we are going to find a scheme to compute
$\CK(D)$ as well as $\SK(D)$.
\begin{prop}\label{propf4}
Let $F$ be a Henselian field with valuation $w$ and $D$ be a
tame $F$-central division algebra of index $n$. If
$\overline{D}^*/\overline{F}^*\overline{D'}=1$ then $\CK(D)\cong
\Gamma_D/\Gamma_F$ and $\SK(D)\cong \mu_n(F)/Z(D')$.
\end{prop}
\begin{proof}
Let $v$ be the extension of $w$ to $D$. Set
$NU_F=\operatorname{Nrd}_D(D^*)\cap U_F$ and let $x\in NU_F$. Then
$w(x)=0$ and $x=\operatorname{Nrd}_D(b)$ for some $b\in D^*$. Now,
by Remark \ref{rem} we have that $v(b)=0$ and thus $b\in U_D$.
Therefore, the reduced norm map induces an epimorphism
$$\operatorname{Nrd}_D:\frac{F^*U_D}{F^*D'}\rightarrow\frac{F^{*n}NU_F}{F^{*n}}.$$
Also since $\Gamma_D$ is torsion free, multiplying by $n$
induces an isomorphism
$$\eta_n:\frac{\Gamma_D}{n\Gamma_D}\rightarrow\frac{n\Gamma_D}{n^2\Gamma_D}.$$
On the other hand by Remark \ref{rem} we have
$v(\operatorname{Nrd}_D(D^*))=n\Gamma_D$. Thus, $v$ induces the two
following epimorphisms:
$$v:\CK(D)\rightarrow\frac{\Gamma_D}{n\Gamma_D}\ \ ,\ \
v:\frac{\operatorname{Nrd}_D(D^*)}{F^{*n}}\rightarrow\frac{n\Gamma_D}{n^2\Gamma_D}.$$
Now, applying the fact that $\CK(D)$ is of bounded exponent $n$, one
deduces that the following diagram is commutative with exact rows
and columns where $f,g$ are canonical maps.
\begin{equation}\label{eq3}
 \xymatrix{
&&&1\ar[d]&1\ar[d]& \\
&&&\displaystyle{\frac{F^*U_D}{F^*D'}}\ar[r]^{\operatorname{Nrd}_D}\ar[d]&
\displaystyle{\frac{F^{*n}NU_F}{F^{*n}}}\ar[r]\ar[d]& 1\\
1\ar[r]&\displaystyle{\frac{\mu_n(F)}{Z(D')}}\ar[r]^f&
\SK(D)\ar[r]^g& \CK(D)\ar[r]^{\operatorname{Nrd}_D\ \ }\ar[d]^v&
\displaystyle{\frac{\operatorname{Nrd}_D(D^*)}{F^{*n}}}\ar[r]\ar[d]^v&1 \\
&&1\ar[r]&\displaystyle{\frac{\Gamma_D}{\Gamma_F}}\ar[r]^{\eta_n}\ar[d]&
\displaystyle{\frac{n\Gamma_D}{n\Gamma_F}}\ar[r]\ar[d]&1 \\
&&&1&1& }
\end{equation}
Now, if we show that $F^*U_D=F^*D'$ then we are done. To prove this,
first note that by the Congruence Theorem we have $1+M_D\subseteq
(1+M_F)D'$. Consider the reduction map $U_D\rightarrow
\overline{D}^*$. We have the sequence
$$
\xymatrix{ \overline{D}^*\ar[r]^{\cong\ \ \ \ \ \ }&
U_D/1+M_D\ar[r]^{nat.\ \ \ \ }& U_D/(1+M_F)D'\ar[r]^{\ \
nat.}&F^*U_D/F^*D'. }
$$
Thus $\theta:\overline{D}^*/\overline{F}^*\overline{D'}\rightarrow
F^*U_D/F^*D'$ is an isomorphism. Considering the fact that
$\overline{D}^*/\overline{F}^*\overline{D'}=1$, the theorem follows.
\end{proof}
\begin{cor}\label{corf5}
Let $F$ be a Henselian field and $D$ be a tame $F$-central
division algebra of index $n$. If
$\overline{D}/\overline{F}^*\overline{D'}=1$ then $\SK(D)\cong
\mu_n(\overline{F})/\overline{D'}\cap\overline{F}$.
\end{cor}
\begin{proof}
Since the valuation is tame and Henselian, using Hensel's Lemma, it
is easy to see that $\mu_n(F)\rightarrow\mu_n(\overline{F}),
a\mapsto\overline{a}$ is an isomorphism. Also, it is not difficult
to prove that $Z(D')\cong \overline{Z(D')}\cong
\overline{D'}\cap\overline{F}$. Therefore, $\SK(D)\cong
\mu_n(F)/Z(D')\cong
\mu_n(\overline{F})/\overline{D'}\cap\overline{F}$.
\end{proof}
To continue our study, we need to recall the following result of
\cite{Er83}.
\begin{thm}[Ershov]\label{thmf6}
Let $F$ be a Henselian field and $D$ be a tame $F$-central
division algebra of index $n$. If $a\in U_D$ then
$$\overline{\operatorname{Nrd}_D(a)}=N_{Z(\overline{D})/\overline{F}}
\operatorname{Nrd}_{\overline{D}}(\overline{a})^{n/mm'}$$ where
$m=\operatorname{ind}(\overline{D})$ and
$m'=[Z(\overline{D}):\overline{F}]$.
\end{thm}
Let $D$ be Henselian tame $F$-central division algebra of index $n$.
From the Congruence Theorem we have $(1+M_D)=(1+M_F)[D^*,1+M_D]$.
Taking the reduced norm from both sides we obtain
$\operatorname{Nrd}_D(1+M_D)=(1+M_F)^n$. Now, since $D$ is tame and
Henselian $(1+M_F)$ is $n$-divisible. Thus we have the following
equality:
\begin{equation}\label{eqd4}
\operatorname{Nrd}_D(1+M_D)=1+M_F.
\end{equation}
We will use this observation in the proof of the next theorem.
\begin{thm}\label{thmf7}
Let $F$ be a Henselian field and $D$ be a tame $F$-central
division algebra of index $n$. Then
\begin{enumerate}
  \item If $D$ is unramified then $\CK(D)\cong \CK(\overline{D})$ and
  $\SK(D)\cong \SK(\overline{D})$;
  \item If $D$ is totally ramified then $\CK(D)\cong \Gamma_D/\Gamma_F$ and
  $\SK(D)\cong \mu_n(\overline{F})/\mu_e(\overline{F})$ where $e=\exp(\Gamma_D/\Gamma_F)$;
  \item If $\overline{D}$ is a field, $\overline{D}/\overline{F}$ is
  a cyclic extension with $[\overline{D}:\overline{F}]=k$ and
  $N_{\overline{D}/\overline{F}}(\overline{D}^*)=\overline{F}^{*k}$ then
  $\CK(D)\cong \Gamma_D/\Gamma_F$ and
  $\SK(D)\cong \mu_n(\overline{F})/\overline{D'}\cap\overline{F}$.
\end{enumerate}
\end{thm}
\begin{proof}
(1) Since $D$ is unramified $\Gamma_D/\Gamma_F=1$ and thus by
(\ref{eq3}) and a similar argument as used in the proof of
Proposition \ref{propf4}, we have
\begin{equation}\label{eqd5}
\CK(D)\cong F^*U_D/F^*D'\cong
\overline{D}^*/\overline{F}^*\overline{D'}.
\end{equation}
On the other hand, by \cite[Prop. 1.7]{jw},
$Z(\overline{D})/\overline{F}$ is normal and
$\Gal(Z(\overline{D})/\overline{F})$ is a homomorphic image of
$\Gamma_D/\Gamma_F$. However, $D$ is a tame division algebra. So
$Z(\overline{D})/\overline{F}$ is Galois with
$\Gal(Z(\overline{D})/\overline{F})=1$. Therefore,
$Z(\overline{D})=\overline{F}$. Also, by the proof of Theorem
\ref{thmd14} we have $\overline{D'}=\overline{D}'$. So, by
(\ref{eqd5}) we conclude that $\CK(D)\cong \CK(\overline{D})$.
Moreover, if $a\in U_D$ then
$\overline{\operatorname{Nrd}_D(a)}=\operatorname{Nrd}_{\overline{D}}(\overline{a})$
by Theorem \ref{thmf6}. (Note that by the fundamental inequality we
have $\operatorname{ind}(\overline{D})=\operatorname{ind}(D)$.)
Therefore, the reduction map induces a monomorphism
$$D^{(1)}/D^{(1)}\cap(1+M_D)\rightarrow \overline{D}^{(1)}.$$
This map is in fact surjective and hence an isomorphism, for if
$\overline{a}\in \overline{D}^{(1)}$ then
$\overline{\operatorname{Nrd}_D(a)}=1$. Thus
$\operatorname{Nrd}_D(a)\in 1+M_F$. So, by (\ref{eqd4}) there exists
$b\in 1+M_D$ such that
$\operatorname{Nrd}_D(a)=\operatorname{Nrd}_D(b)$. Therefore
$ab^{-1}\in D^{(1)}$ and $ab^{-1}\mapsto \overline{a}$, as desired.
Now, from the Congruence Theorem we obtain $\SK(\overline{D})\cong
\SK(D)$.

(2) Since $D$ is totally ramified, clearly we have
$\overline{D}^*/\overline{F}^*\overline{D'}=1$. So by Proposition
\ref{propf4} we obtain $\CK(D)\cong \Gamma_D/\Gamma_F$. Also since
in this case $\overline{D'}=\mu_e(\overline{F})$ from Corollary
\ref{corf5} we conclude that $\SK(D)\cong
\mu_n(\overline{F})/\mu_e(\overline{F})$.

(3) From Proposition \ref{propf4} it is enough to prove
$N_{\overline{D}/\overline{F}}(\overline{D}^*)/\overline{F}^{*k}\cong
\overline{D}/\overline{F}^*\overline{D'}$. To show this, consider
the norm function
$N_{\overline{D}/\overline{F}}:\overline{D}^*\rightarrow
\overline{F}^*$. Now, if $x\in U_D$ then by Ershov's formula it
follows that
$\overline{\operatorname{Nrd}_D(x)}=N_{\overline{D}/\overline{F}}(\overline{x})$.
This shows that $\overline{D'}\subseteq \ker
N_{\overline{D}/\overline{F}}$. Conversely, if $x\in \ker
N_{\overline{D}/\overline{F}}$ then by Hilbert Theorem 90, there is
a $\overline{b}$ such that $x=\overline{b}\sigma(\overline{b})^{-1}$
where $\sigma$ is the generator of
$\Gal(\overline{D}/\overline{F})$. Now, since the fundamental
homomorphism $D^*\rightarrow\Gal(Z(\overline{D}/\overline{F}))$ is
surjective, it follows that $\sigma$ is of the form
$\sigma(\overline{a})=\overline{cac^{-1}}$, for some $c\in D^*$.
This implies that $x\in \overline{D'}$ and so $\ker
N_{\overline{D}/\overline{F}}= \overline{D'}$. Therefore
$$\overline{D}/\overline{F}^*\overline{D'}\cong
N_{\overline{D}/\overline{F}}(\overline{D}^*)/\overline{F}^{*k},$$
because the image of $\overline{F}^*\overline{D'}$ under
$N_{\overline{D}/\overline{F}}$ is $\overline{F}^{*k}$.
\end{proof}
\begin{cor}\label{corf8}
Let $D$ be an $F$-central division algebra of index $n$. If
$\operatorname{char}(F)\nmid n$ then $\CK(D)\cong \CK(D((x)))$ and
$\SK(D)\cong \SK(D((x)))$.
\end{cor}
Here, we are going to use the above results in computing of $\CK$
and $\SK$ for some division algebras.
\begin{exa}\label{exaf9}
Recall that a field $F$ is called \textit{real Pythagorean} if
$-1\not\in F^{*2}$ and the sum of any two square elements of $F$ is
a square in $F$. Let $F$ be a real Pythagorean field and
$\mathcal{Q}$ be the ordinary quaternion algebra over $F$, i.e.,
$$\mathcal{Q}=\{a+bi+cj+dij| a,b,c,d\in F, i^2=j^2=-1, ij=-ji\}.$$
It is easy to see that $\mathcal{Q}$ is a division algebra. If
$u=a+bi+cj+dij\in \mathcal{Q}^*$ then
$\operatorname{Nrd}_{\mathcal{Q}}(u)= a^2+b^2+c^2+d^2\in F^{*2}$ as
$F$ is real Pythagorean. So there is a $v\in F^*$ such that
$v^2=\operatorname{Nrd}_{\mathcal{Q}}(u)$. Now we have
$\operatorname{Nrd}_{\mathcal{Q}}(u/v)=
\operatorname{Nrd}_{\mathcal{Q}}(u)/v^2=1$. Thus $u/v\in D^{(1)}$.
But by Wang's Theorem $\SK(D)=1$ and hence $u/v\in D'$. Therefore
$u\in F^*D'$. This implies that $\CK(\mathcal{Q})=1$. Now, we will
show that the converse is also true. Let $F$ be a field and
$\mathcal{Q}$ be the ordinary quaternion algebra over $F$. We
establish that if $\mathcal{Q}$ is a division algebra and
$\CK(\mathcal{Q})=1$ then $F$ is real Pythagorean. To prove this, first
note that since $\mathcal{Q}$ is a division ring, $-1\not\in
F^{*2}$. Also, since $\CK(\mathcal{Q})=1$ we have $D^*=F^*D'$. Now,
taking the reduced norm from both sides we conclude that
$\operatorname{Nrd}_{\mathcal{Q}}(\mathcal{Q})=F^{*2}$. Thus if
$u,v\in \mathcal{F}$ then
$u^2+v^2=\operatorname{Nrd}_{\mathcal{Q}}(u+vi) \in F^{*2}$. This
immediately implies that $F$ is real Pythagorean.

Now, let $F$ be real Pythagorean and $\mathcal{Q}$ be the quaternion
algebra over $F$. Using Corollary \ref{corf8} we conclude that
$\CK(\mathcal{Q}((x)))=1$ and so $F((x))$ (the center of
$\mathcal{Q}((x))$) is real Pythagorean. Therefore, we have proved
that if $F$ is real Pythagorean then for each $n$ the field
$F((x_1\ldots x_n))$ is real Pythagorean.
\end{exa}

\begin{exa}\label{exaf10}
Let $\mathbb{R}$ be the field of real numbers and
$F=\mathbb{R}((x_1,\ldots, x_n))$. The field $F$ has a natural
valuation $v:F\rightarrow \oplus_{i=1}^n\mathbb{Z}$ given by
\begin{equation}\label{eqd6}
v(\sum_{j_1}\ldots\sum_{j_n}c_{j_1\ldots j_n}x_1^{j_1}\ldots
x_n^{j_n})=\min \{(j_1,\ldots,j_n)|c_{j_1\ldots j_n}\neq 0\},
\end{equation}
where $\oplus_{i=1}^n\mathbb{Z}$ is given right-to-left
lexicographical ordering. The above valuation is Henselian and if
$D$ is an $F$-central division algebra then $D$ is obviously tame.
It is also clear that $\overline{F}=\mathbb{R}$. Because the
algebraically closed field $\mathbb{C}=R(\sqrt{-1})$ has a trivial
Brauer group, we conclude that $\overline{D}=\mathbb{R}$ or
$\overline{D}=\mathbb{C}$ or $\overline{D}=\mathcal{Q}$ where
$\mathcal{Q}$ is the ordinary quaternion division algebra over
$\mathbb{R}$. Now, the above example shows that in all cases
$\overline{D}^*/\overline{F}^*\overline{D'}=1$. Therefore, Corollary
\ref{corf5} yields $\SK(D)\cong \mu_n(\mathbb{R})/\overline{D'}\cap
\mathbb{R}$. But from number theory we know that $[D:F]$ is a power
of $2$. So $\mu_n(\mathbb{R})=1$. On the other hand by the work of
Wadsworth $-1\in D'$ (cf. \cite{Wad90}).  Thus we have $\SK(D)=1$.

Now let $F=\mathbb{C}((x_1,x_2,x_3,x_4))$ and $n$ a natural number.
For $j=1,3$ let $D_j=A_{\omega_n}(x_j,x_{j+1};F)$ where $\omega_n$
is the primitive $n$-th root of unity and
$A_{\omega_n}(x_j,x_{j+1};F)$ is the symbol algebra generated by
symbols $\alpha,\beta$ subject to the relations $\alpha^n=x_j$,
$\beta^n=x_{j+1}$ and $\alpha\beta=\omega_n\beta\alpha$. By example
3.6 of \cite{TW87} we know that $D=D_1\otimes_FD_3$ is an
$F$-central division algebra with $\operatorname{ind}(D)=n^2$ and
$\Gamma_D/\Gamma_F\cong
\mathbb{Z}/n\mathbb{Z}\oplus\mathbb{Z}/n\mathbb{Z}$. Thus
$e=\exp(\Gamma_D/\Gamma_F)=n$. On the other hand, $D$ is totally
ramified as $\mathbb{C}$ is algebraically closed. Therefore, from
Theorem \ref{thmf7} we have $\SK(D)\cong \mathbb{Z}/n\mathbb{Z}$.
This implies that there exists a field such that when $D$ runs over
all $F$-central division algebras, then $\SK(D)$ can be any finite
cyclic group. Note that if $D$ is an arbitrary division algebra over
$F$, then $D$ is totally ramified and $\SK(D)$ is a finite cyclic
group.
\end{exa}
In Example \ref{exaf10} we have deduced both Theorems of
\cite{Lip76}, which are obtained by using heavy machinery of reduced
$K$-Theory, as natural examples of the above theorems. In the next
example, we show that the group $\CK$ can be any finite cyclic group
of odd order.
\begin{exa}\label{exaf11}
Let $F=\mathbb{F}_2((x))$. Equipped with its natural valuation, $F$
is local and hence $Br(F)\cong \mathbb{Q}/\mathbb{Z}$. Now, since
$\exp([D])=\operatorname{ind}(D)$ for every $F$-division algebra $D$
($[D]$ is the equivalent class of $D$ in the Brauer group of $F$),
given a natural number $n$, there exists an $F$-central division
algebra $D$ such that $\operatorname{ind}(D)=n$ . Now, let $n$ be
odd. Since $F$ is Henselian, the natural valuation has a unique
extension to $D$. Clearly $\overline{F}=\mathbb{F}_2$, $D$ is tame,
and $\overline{D}/\overline{F}$ is cyclic (recall that the Brauer
group of every finite field is trivial and every finite extension of
a finite field is cyclic). At the other extreme, we have
$N_{\overline{D}/\overline{F}}
(\overline{D}^*)=\mathbb{F}_2^*=\{1\}$ which is obviously
$n$-divisible. Now, Theorem \ref{thmf7} implies that $\CK(D)\cong
\mathbb{Z}/n\mathbb{Z}$. Thus, when $[D]$ ranges over the Brauer
group of $F$ then $\CK(D)$ can be any finite cyclic group of odd
order.
\end{exa}
\begin{exa}\label{exaf12}
Let $\mathbb{C}$ be the field of complex numbers, $n_1,n_2,...,n_r$
an arbitrary sequence of positive integers and $n=n_1n_2...n_r$. If
$i$ is even, set $D_{i+1}=D_i((x_{i+1}))$ and if $i$ is odd we
define $\sigma_i:D_i\to D_i$ by $\sigma_i(x_i)=\omega_ix_i$, where
$\omega_i$ is a primitive $n_{(i+1)/2}$-th root of unity. Then
$D_{i+1}=D_i((x_{i+1},\sigma_i))$. So, we have
$D=D_{2r}=\mathbb{C}((x_1,...,x_{2r},\sigma_1, \sigma_3, ...,
\sigma_{2r-1}))$. By Hilbert's construction (cf. \cite{Dra83}),
$F=Z(D)=\mathbb{C}((x_1^{n_1}, x_2^{n_1},..., x_{2r-1}^{n_r},
x_{2r}^{n_r}))$ with $[D:F]=n^2$. Now, if we consider the natural
valuation on $D$ then
$\Gamma_D=\oplus_{j=1}^{r}(\mathbb{Z}\oplus\mathbb{Z})$ and
$\Gamma_F=\oplus_{j=1}^{r}(n_j\mathbb{Z}\oplus n_j\mathbb{Z})$. In
this setting $D$ is a tame and totally ramified as $\mathbb{C}$ is
algebraically closed. Therefore, from Theorem \ref{thmf7} we have
$\CK(D)\cong \Gamma_D/\Gamma_F\cong
\oplus_{j=1}^{r}(\mathbb{Z}/n_j\mathbb{Z}\oplus\mathbb{Z}/n_j\mathbb{Z})$.
This shows that for every abelian group $A$ there exists a division
algebra $D$ with $\CK(D)\cong A\oplus A$.
\end{exa}
\begin{thm}\label{thmf13}
Let $D$ be a tame division algebra over a Henselian field $F=Z(D)$
of index $n$. If $D$ has a maximal cyclic extension $L/F$ such that
$\overline{D}^*/\overline{L}^*\overline{D'}=1$ then $\SK(D)=1$.
\end{thm}
\begin{proof}
As in the proof of Proposition \ref{propf4} one may conclude that
the reduction map induces an isomorphism
$\overline{D}^*/\overline{L}^*\overline{D'}\rightarrow U_D/U_LD'$.
This yields $U_D=U_LD'$. Now, if $x\in D^{(1)}$ then $x=ld$ where
$l\in L$ and $d\in D'$. So $\operatorname{Nrd}_D(x)=N_{L/F}(l)=1$.
From the Hilbert Theorem 90, $l$ is of the form $a\sigma(a)^{-1}$ where
$\sigma$ is the generator of $\Gal(L/F)$. Now the Skolem-Noether Theorem
implies that $\sigma(a)=cac^{-1}$ for some $c\in D^*$. Therefore
$l=aca^{-1}c^{-1}\in D'$, as desired.
\end{proof}
\begin{exa}\label{exaf14}
Let $L/F$ be a cyclic extension with
$\Gal(L/F)=\langle\sigma\rangle$. Let $\operatorname{char}(F)\nmid n$
where $n=[L:F]$ and consider the division algebra $D=L((x,\sigma))$
which is obtained by the Hilbert construction. We know that
$Z(D)=F((x^n))$ and $\operatorname{ind}(D)=n$. Considering the
natural valuation of $D$ we have $\overline{D}=\overline{L}$. So
Theorem \ref{thmf13} shows that $\SK(D)=1$.
\end{exa}
In the following example we show that in contrast with $\SK$, the
group $\CK$ is not stable under purely transcendental extensions.
\begin{exa}\label{exaf15}
Let $D$ be a division algebra over its center $F$ with index $n$.
Let $\mathcal{P}$ runs over all irreducible monic polynomials of
$F[x]$ and $n_{\mathcal{P}}$ be the index of
$D\otimes_FF[x]/\mathcal{P}$. By \cite[Prop. 7]{LT93} the sequence
\begin{equation}\label{eqd7}
1\rightarrow K_1(D)\rightarrow K_1(D(x))\rightarrow
\bigoplus_{\mathcal{P}} (n_{\mathcal{P}}/n)\mathbb{Z}\rightarrow 1
\end{equation}
which is obtained from the localization of exact sequences in
algebraic $K$-Theory, is split exact. Thus, since every
$(n_{\mathcal{P}}/n)\mathbb{Z}$ is torsion-free, the sequence
(\ref{eqd7}) induces an isomorphism
$$\psi:\tau(K_1(D))\rightarrow \tau(K_1(D(x)))\ ,\ \overline{a}\mapsto \overline{a\otimes 1},$$
where $\tau(K_1(D))$ and $\tau(K_1(D(x)))$ are the torsion subgroups
of $K_1(D)$ and $K_1(D(x))$, respectively. Now, if $\overline{a}\in
\SK(D)$ then $\operatorname{Nrd}_{D(x)}(a\otimes 1)=
\operatorname{Nrd}_D(a)=1$ and so $a\otimes 1\in \SK(D(x))$.
Conversely, if $\overline{b}\in \SK(D(x))$ then there is an
$\overline{a}\in \tau(K_1(D))$ such that
$\overline{b}=\overline{a\otimes 1}$. But
$\operatorname{Nrd}_{D(x)}(a\otimes 1)=\operatorname{Nrd}_D(a)=1$.
This implies that the restriction of $\psi$ on $\SK(D)$ induces an
isomorphism $\SK(D)\rightarrow \SK(D(x))$. On the other hand since
$\CK(D)$ is the cokernel of the natural map $K_1(F)\rightarrow
K_1(D)$, applying the Snake Lemma to the commutative diagram
$$
\xymatrix{
1\ar[r]&K_1(F)\ar[r]\ar[d]&K_1(F(x))\ar[r]\ar[d]&\bigoplus_{\mathcal{P}}\mathbb{Z}\ar[r]\ar[d]&1\\
1\ar[r]&K_1(D)\ar[r]&K_1(D(x))\ar[r]&\bigoplus_{\mathcal{P}}(n_{\mathcal{P}}/n)\mathbb{Z}\ar[r]&1
}
$$
yields the following split exact sequence
\begin{equation}\label{eqd8}
1\rightarrow \CK(D)\rightarrow
\CK(D(x))\rightarrow\bigoplus_{\mathcal{P}}
\frac{\mathbb{Z}}{(n/n_{\mathcal{P}})\mathbb{Z}}\rightarrow 1.
\end{equation}
Now, consider the special case in which $D$ is the real quaternion
$\mathcal{Q}$. Since the degree of every irreducible polynomial over
$\mathbb{R}$ is either $1$ or $2$ we can easily see that
$\oplus_{\mathcal{P}}(n/n_{\mathcal{P}})\mathbb{Z} \cong
\oplus_{i=1}^{\infty}2\mathbb{Z}$. So, from Example \ref{exaf9} and
(\ref{eqd8}) we conclude that $\CK(\mathcal{Q}(x))\cong
\oplus_{i=1}^{\infty} \mathbb{Z}/2\mathbb{Z}$.
\end{exa}
Unlike $\SK$ it seems that $\CK$ is rarely trivial. In fact in
\cite{HaMaHMi99} the following conjecture was posed.
\begin{conj}\label{conjf16}
$\CK(D)$ is trivial if and only if $D$ is quaternion over a real Pythagorean field.
\end{conj}
The first attempt to prove Conjecture \ref{conjf16} was fulfilled in
\cite{HaVi05} where it is proved that if $D$ is a tensor
product of cyclic algebras then the answer to this conjecture is
positive. Afterwards in \cite{HaWa07}, Hazrat and Wadsworth proved that if
$F$ is finitely generated over a subfield $F_0$ but not algebraic
over $F_0$, then $F$ is NKNT, i.e., for any non-commutative division
algebra $D$ finite dimensional over $F$ (and not necessary central
over $F$), then $\NK(D)=D^*/F^*D^{(1)}$ is non-trivial and hence
$\CK(D)\neq 1$. In what follows we are going to review these
results.
\begin{thm}\label{thmf17}
Let $D$ be a noncommutative division algebra similar to a cyclic
algebra $A$. If one of the following conditions holds, then
$\NK(D)\neq 1$:
\begin{enumerate}
  \item $F$ contains a square root of $-1$;
  \item $\operatorname{char}(F)=2$;
  \item The degree of $A$ is odd.
\end{enumerate}
\end{thm}
\begin{proof}
It is not hard to see that every primary component of $D$ is similar
to a cyclic algebra. Thus by Theorem \ref{thmf3} we may assume that
$\operatorname{ind}(D)=p^n$ where $p$ is a prime number and
$n\geq1$. Put $A=(E/F,\sigma,a)$. Choose $A$ of minimal degree. Then
$\operatorname{deg}(A)=p^m$ and $a\not\in F^{*p}$, because if
$a=b^p$ then $[A]=[(E_0/F,\sigma,b)]$ where $[E:E_0]=p$,
contradicting the minimality of $\deg(A)$. Let $\alpha\in A^*$ be
such that $\sigma(l)=\alpha^{-1}l\alpha$ and $\alpha^{p^m}=a$. Now,
one can easily see that the minimal polynomial of $\alpha$ over $F$
is $x^{p^m}-a$. Thus
$\operatorname{Nrd}_A(\alpha)=N_{F(\alpha)/F}(\alpha)=
(-1)^{p^m-1}a$. But if $\operatorname{char}F=2$ or $p$ is odd then
$\operatorname{Nrd}_A(\alpha)=a\not\in F^{*p}$. Also if $F$ contains
a square root of $-1$ and $p=2$ then
$\operatorname{Nrd}_A(\alpha)=-a\not\in F^{*2}$ as $a\not\in
F^{*2}$. Therefore all of the above conditions give
$\operatorname{Nrd}_{A}(\alpha)\not\in F^{*p}$. However, thanks to
the Dieudonn$\acute{\textrm{e}}$ determinant,
$\operatorname{Nrd}_A(A^*)= \operatorname{Nrd}_D(D^*)$. Thus
$\operatorname{Nrd}_D(\alpha)\not\in F^{*p^n}$. This forces that
$\operatorname{Nrd}_D(D^*)\neq F^{*p^n}$ and so $\NK(D)\neq 1$.(Note
that $\NK(D)\cong D^*/F^*D^{(1)}$ with the isomorphism given by the
reduced norm map.)
\end{proof}
Recall that by Albert's Main Theorem \cite[p. 110]{Dra83} every
$p$-algebra is Brauer equivalent to a cyclic $p$-algebra. Using this
fact and Theorem \ref{thmf17} the following corollary is immediate.
\begin{cor}\label{corf18}
If $D$ is a $p$-division algebra then $\NK(D)\neq 1$.
\end{cor}
\begin{prop}\label{propf19}
Let $D$ be a cyclic division algebra. Then $\NK(D)=1$ if and only if
$D$ is the ordinary quaternion division algebra and $F$ is real
Pythagorean.
\end{prop}
\begin{proof}
Let $\NK(D)=1$. From Theorem \ref{thmf3} and Theorem \ref{thmf17} it
follows that $\deg(D)=2^m$ for some $m$ and $\sqrt{-1}\not\in F^*$.
Let $D=(E/F,\sigma,a)$ and put $n=\deg(D)$. Let $\alpha\in D^*$ be
such that $D=\oplus_{j=1}^{n-1}E\alpha^j$ and $\alpha^n=a$. As in
the proof of Theorem \ref{thmf17} we have
$\operatorname{Nrd}_D(\alpha)=-a=-\alpha^n$. But
$\operatorname{Nrd}_D(\alpha)\in F^{*n}$ because $\NK(D)=1$. So
$-\alpha^n=f^n$ for some $f\in F^*$. This yields $(\alpha
f^{-1})^n=-1$. Thus we may replace $a$ by $-1$, i.e.,
$D=(E/F,\sigma,-1)$. Now, let $m>1$ and let $L\supset F$ be a
subfield of $E$ with $[L:F]=2$. Since $m>1$ then the division
algebra $Z_D(L)\cong (E/L,\sigma^2,-1)$ has index $2^{m-1}>1$. But
if $L$ contains a square root of $-1$ then $Z_D(L)$ is not a
division algebra which is a contradiction. So $\sqrt{-1}\not\in L$.
On the other hand $L=F(\beta)$ for some $\beta\in L$ with
$\beta^2\in F$. Since the minimal polynomial of $\beta$ over $F$ is
$x^2-\beta^2$ we conclude that $N_{L/F}(\beta)=-\beta^2$. Therefore
\begin{align}\label{eqd9}
\operatorname{Nrd}_D(\beta)=N_{L/F}&(N_{E/L}(\beta))=\\ \nonumber
N_{L/F}&(\beta^{2^{m-1}})=N_{L/F}(\beta)^{2^{m-1}}=(-\beta^2)^{2^{m-1}}=\beta^{2^m}.
\end{align}
At the other extreme, $\beta=fd$ for some $f\in F$ and $d\in
D^{(1)}$ as $\NK(D)=1$. Thus $\operatorname{Nrd}_D(\beta)=f^{2^m}$.
From (\ref{eqd9}) we obtain $\beta^{2^m}=f^{2^m}$ and hence $(\beta
f)^{2^{m}}=1$. Let $j$ be the smallest nonnegative integer such that
$(\beta f)^{2^{j}}=1$. If $j=0$ or $1$ then $\beta=\pm f$ which is
absurd. So $j\geq 2$. Hence $\sqrt{-1}\in L$ and $L=F(\sqrt{-1})$.
This contradiction implies that $m=1$. Thus $D$ is a quaternion
division algebra. Now, a similar argument as above shows that every
maximal subfield of $D$ is $F$-isomorphic to $F(\sqrt{-1})$ and
therefore $D$ is the ordinary quaternion division algebra. Now, from
Example \ref{exaf9} it follows that $F$ is real Pythagorean.
\end{proof}
Here we must point out that by the work of Hazrat and Vishne, a more
general conclusion than Proposition \ref{propf19} holds for cyclic
central simple algebras.  In \cite[Cor. 2.11]{HaVi05} they showed
that if $A$ is a cyclic $F$-central simple algebra with $\NK(A)=1$
then $A$ is a matrix algebra over $F$ or over an ordinary quaternion
division algebra. Now we are in a position to prove the following
theorem.
\begin{thm}\label{thmf20}
Let $D=C_1\otimes_F\ldots\otimes_F C_k$ be an $F$-central division
algebra where $C_1,\ldots,C_k$ are cyclic algebras over $F$. If
$\NK(D)=1$ then $D$ is a cyclic division algebra. Hence $D$ is the
ordinary quaternion division algebra and $F$ is real Pythagorean.
\end{thm}
\begin{proof}
By Theorem \ref{thmf3} we may assume that $\deg(D)$ is a prime
power. Let $n=\deg(D)$ and $n_i=\deg(C_i)$. For each $i$, let $K_i$
be a cyclic maximal subfield of $C_i$, and $e_i\in C_i$ an element
inducing an automorphism $\sigma_i$ of order $n_i$ of $K_i/F$. Then
$b_i=e_i^{n_i}\in F^*$. Now,
$\operatorname{Nrd}_D(z_i)=\operatorname{Nrd}_{C_i}(z_i)^{n/n_i}=
((-1)^{n_i-1}b_i)^{n/n_i}=b_i^{n/n_i}$ (note that since $n$ is a
prime power $(n_i-1)n/n_i$ is even). But since $\NK(D)=1$ we have
$b_i^{n/n_i}\in F^{*n}$. Thus multiplying $e_i$ by a central element
we may assume $b_i$ is an $n/n_i$-th root of unity. Now, let $b$ be
the generator of the group $\langle b_1,\ldots,b_k\rangle$. So every
$C_i$ is a cyclic algebra of the form $(K_i/F,\sigma_i,b^{g_i})$ for
some $g_i$ and $D$ is similar in the Brauer group to cyclic algebra
of degree $\operatorname{lcm}(n_1,\ldots,n_k)$. But since
$\operatorname{lcm}(n_1,\ldots,n_k)\leq n_1\ldots n_k(=n)$ and $D$
is a division algebra we conclude that
$\operatorname{lcm}(n_1,\ldots,n_k)=n_1\ldots n_k$. This forces that
$k=1$ (because $n$ is a prime power) and now the result follows from
Proposition \ref{propf19}.
\end{proof}
\begin{exa}\label{exaf21}
Since every division algebra over a global field is a cyclic
algebra, Theorem \ref{thmf20} implies that every global field is
NKNT. Likewise, every nonreal local field is NKNT. But the field of
real numbers $\mathbb{R}$ is not NKNT as the ordinary
$\mathbb{R}$-division algebra has trivial $\NK$. At the other
extreme $\mathbb{R}(x)$ is NKNT. For if $L$ is a finite field
extension of $\mathbb{R}(x)$ and $D$ is an $L$-central
noncommutative division algebra, then by Tsen's Theorem,
$L(\sqrt{-1})$ splits $D$ and thus $D$ is a quaternion division
algebra. Now, since $L$ is not Pythagorean $\NK(D)\neq1$ (Theorem
\ref{thmf20}).
\end{exa}
As mentioned above, a significant result concerning Conjecture
\ref{conjf16} was obtained in \cite{HaWa07}. Here we present the
main result of of that paper without proof.
\begin{thm}\label{thmf22}
Let $D$ be a noncommutative division algebra with center $F$ which
is finitely generated over some subfield $F_0$. If $\NK(D)$ is
trivial, then $[F:F_0]<\infty$.
\end{thm}
Theorem \ref{thmf22} can be restated as :
\begin{cor}
Let $F$ be field which is finitely generated but not algebraic over
some subfield $F_0$. Then $F$ is $\operatorname{NKNT}$.
\end{cor}
Also from Theorem \ref{thmf22} and Example \ref{exaf21} the
following corollary is immediate.
\begin{cor}
If $D$ is a noncommutative division algebra whose center is finitely
generated over its prime subfield or over an algebraically closed
field, then $\NK(D)$ is nontrivial. Hence $\CK(D)\neq 1$.
\end{cor}

\begin{rem}
An inverse problem to the reduced $K$-theory\footnote{Not to be mistaken with the established ``inverse problem'', that  is to find all the abelian groups which appear as the group $\SK$.} is the problem of triviality of the group $\CK(D)$. Note that
\begin{equation}\label{uuy11}
\SK(D)=\ker\big(K_1(D) \stackrel{\Nrd}{\longrightarrow} K_1(F)\big),
\end{equation}
whereas
\begin{equation}\label{uuy22}
\CK(D)=\text{coker}\big(K_1(F) \stackrel{id.}{\longrightarrow} K_1(D)\big).
\end{equation}

In the reduced $K$-theory, the nontriviality of $\SK$ is the major question, whereas with the group $\CK$, it is the triviality which has remained an open question. The conjecture is that $\CK(D)$ is trivial if and only if $D$ is an ordinary quaternion division algebra over Pythagorean field \cite{HW09}. Recall that this conjecture  has a direct application in solving the open problem
whether a multiplicative group of a division algebra has a maximal
subgroup \cite{HW09}. Indeed, since $\CK(D)$ is torsion of
bounded exponent, if it is not
trivial, it has maximal subgroups and therefore $D^*$ has (normal)
maximal subgroups. Thus finding the maximal subgroups in $D^*$
reduces to the case that $\CK(D)$ is trivial. It was shown
in \cite{HW09} that quaternion division algebras over pythagorean fields do have (non-normal)
maximal subgroups. Thus if the above conjecture is settled
positively, one concludes that the multiplicative group of a
division algebra does have a maximal subgroup.

The groups $\SK$ and $\CK$ are both torsion groups of bounded exponent. This is one of the main properties of these groups that make them tractable. With the view of (\ref{uuy22}), for an Azumaya algebra $A$ over its center $R$, the inclusion map $\text{id}:R\rightarrow
A$ gives the following exact sequence
\begin{equation}\label{zk}
1 \rightarrow \ZK[i](A) \rightarrow K_i(R) \rightarrow K_i(A)
\rightarrow \CK[i](A) \rightarrow 1.
\end{equation}
Here $K_i$ for $i \geq 0$ are the Quillen $K$-groups and $\ZK[i](A)$,  $\CK[i](A)$ are the kernel and co-kernel of
$K_i(R) \rightarrow K_i(A)$ respectively. Clearly when $A$ is a division algebra and $i=1$, we get $\CK(D)$ defined in (\ref{uuy22}). The fact that $\CK[i](A)$ and $\ZK[i](A)$ are torsion of bounded exponent was studied in several papers (see~\cite{bant,hazhoobler}). A consequence of this, is that the $K$-theory of $A$ coincides with the $K$-theory of its base ring up to torsions.
\end{rem}

\section{Graded approach to the theory of division algebras}
As previous sections show, valued division algebras have been one of the main sources of producing (counter)examples in the theory of central simple algebras. In fact, valued division algebras have been the key ingredients in Amitsur's non-crossed product construction and in Platonov's construction of division algebras $D$ with a non-trivial reduced Whitehead group. These two constructions settled major and several-decade open problems. A glance at these important works shows the formidable technical calculations required in the presence of valuations to achieve the results.

Starting with a division algebra $D$ with a valuation,  one can construct a graded ring
\[{\gr(D)= \bigoplus_{\ga \in \Ga_D}\gr(D)_\gamma},\] where
$\Ga_D$ is the value group of $D$ and
the summands $\gr(D)_\gamma$ arise from the filtration on $D$
induced by the valuation
(see \S\ref{prel} for details). Infact, $\gr(D)$ is a graded division ring,
i.e., every nonzero homogeneous element of $\gr(D)$
is a unit.  While $\gr(D)$ has a much simpler
structure than $D$, nonetheless $\gr(D)$ provides a
remarkably good  reflection of $D$ in many ways,
particularly when the valuation on the center
$Z(D)$ is Henselian. The approach of making calculations
in $\gr(D)$, then lifting back to get nontrivial
information about $D$ has been remarkably successful.
This has provided enough motivation for a systematic study of this
correspondence, notably by Boulagouaz, Hazrat, Hwang,
Tignol and Wadsworth \cite {boulag,hwsk1,hwunitary,hwalg,hwcor,tigwad,wyan}, and to
compare certain functors defined on these objects,
notably the Brauer groups, Witt groups, and the reduced Whitehead groups.

In the case of reduced Whitehead group that we are concerned in this note,   it was proved
in \cite[Th.~4.8, Th.~5.7]{hwsk1}  that if a valuation $v$ on $Z(D)$ is
Henselian and $D$ is tame over $Z(D)$, then
\begin{equation}\label{tthhyy1}
\SK(D) \cong
\SK(\gr(D))
\end{equation}
and
\begin{equation}\label{tthhyy2}
\SK(\gr(D))\cong \SK(q(\gr(D))),
\end{equation}
where $q(\gr(D))$ is the division ring of quotients of
$\gr(D)$.  This has allowed the recovery of many of some known
calculations of $\SK(D)$ with much easier proofs, as well
as leading to the determination of $\SK(D)$ in some
new cases. In this section we demonstrate how to obtain~(\ref{tthhyy1}). Most of the material in this section are taken from~\cite{hwsk1,millar}.

\subsection{Graded division algebras}\label{prel}

Here we establish the notation and
recall some fundamental facts about graded division
algebras indexed by a totally ordered abelian group, and
about their connections with valued division algebras.
In addition, we
 establish some important  homomorphisms relating the group
structure of a valued division algebra to the group
structure of its associated graded division algebra.

Let
$R = \bigoplus_{ \ga \in \Gamma} R_{\ga}$ be a
graded ring, i.e.,
  $\Gamma$ is an abelian group,  and $R$ is a
unital ring such that each $R_{\ga}$ is a
subgroup of $(R, +)$ and
$R_{\ga} \cdot R_{\de} \subseteq R_{\ga +\de}$
for all $\ga, \de \in \Ga$. Set
\begin{itemize}
\item[] $\Gamma_R  \ = \  \{\ga \in \Ga \mid R_{\ga} \neq 0 \}$,
 \  the grade set of $R$;
\vskip0.05truein
\item[] $R^h  \ = \ \bigcup_{\ga \in
\Ga_{R}} R_{\ga}$,  \ the set of homogeneous elements of $R$.
\end{itemize}
For a homogeneous element of $R$  of degree $\gamma$, i.e., an
$r \in R_\gamma\mi{0}$, we write $\deg(r) = \gamma$.
Recall  that $R_{0}$ is a subring of $R$ and that for each $\ga \in
\Ga_{R}$, the group $R_{\ga}$ is a left and right $R_{0}$-module.
A subring $S$ of
$R$ is a \emph{graded subring} if $S= \bigoplus_{ \ga \in
\Gamma_{R}} (S \cap R_{\ga})$.  For example, the
center of $R$, denoted $Z(R)$, is a graded subring of
$R$.
If $T = \bigoplus_{ \ga \in
\Gamma} T_\gamma$ is another graded ring,
a {\it graded ring homomorphism} is a ring homomorphism
$f\colon R\to T$ with $f(R_\gamma) \subseteq T_\gamma$
for all $\gamma \in \Gamma$.  If $f$ is also bijective,
it is called a graded ring isomorphism;  we then write
$R\conggr T$.

For a graded ring $R$, a graded left $R$-module $M$ is
a left  $R$-module with a grading ${M=\bigoplus_{\ga \in \Ga'}
M_{\ga}}$,
where the $M_{\ga}$ are all abelian groups and $\Ga'$ is an
abelian group containing $\Ga$, such that $R_{\ga} \cdot
M_{\delta} \subseteq M_{\ga + \delta}$ for all $\ga \in \Ga_R,
\delta \in \Ga'$.
Then, $\Gamma_M$ and $M^h$ are
defined analogously to $\Gamma_R$ and~$R^h$.  We say that $M$ is
a {\it graded free} $R$-module if it has a base as a free
$R$-module consisting of homogeneous elements.

A graded ring $E = \bigoplus_{ \ga \in \Gamma} E_{\ga}$ is
called a \emph{graded division ring} if $\Ga$ is a
torsion-free abelian group and
every non-zero homogeneous
element of $E$ has a multiplicative inverse.
Note that the grade set  $\Ga_{E}$ is actually a group.
Also, $E_{0}$ is a division ring,
and  $E_\gamma$ is a $1$-dimensional
left and right $E_0$ vector space for every $\gamma\in \Gamma_E$.
The requirement that $\Gamma$ be torsion-free is made
because we are interested in graded division rings arising
from valuations on division rings, and all the grade groups
appearing there are torsion-free.  Recall that every
torsion-free abelian group $\Gamma$ admits total orderings compatible
with the group structure.  (For example, $\Gamma$ embeds in
$\Gamma \otimes _{\mathbb Z}\mathbb Q$ which can be given
a lexicographic total ordering using any base of it as a
$\mathbb Q$-vector space.)  By using any total ordering on
$\Gamma_E$, it is easy to see that $E$ has no zero divisors
and that $E^*$, the multiplicative group of units of $E$,
coincides with
 $E^{h} \mi \{0\}$ (cf. \cite{hwcor}, p.~78).
Furthermore, the degree map
\begin{equation}\label{degmap}
\deg\colon E^* \rightarrow \Gamma_E
\end{equation} is a group homomorphism with kernel $E_0^*$.

By an easy adaptation of the ungraded arguments, one can see
that every graded  module~$M$  over a graded division ring
$E$ is graded free, and every two
homogenous bases have the same cardinality.
We thus call $M$ a \emph{graded vector space} over $E$ and
write $\dim_E(M)$ for the rank of~$M$ as a graded free $E$-module.
Let $S \subseteq E$ be a graded subring which is also a graded
division ring.  Then, we can view $E$ as a graded left $S$-vector
space, and  we write $[E:S]$ for $\dim_S(E)$.  It is easy to
check the \lq\lq Fundamental Equality,"
\begin{equation}\label{fundeq}
[E:S] \ = \ [E_0:S_0] \, |\Gamma_E:\Gamma_S|,
\end{equation}
where $[E_0:S_0]$ is the dimension of $E_0$ as a left vector space
over the division ring $S_0$ and $|\Gamma_E:\Gamma_S|$ denotes the
index in the group $\Gamma_E$ of its subgroup $\Gamma_S$.

A \emph{graded field} $T$ is a commutative graded division ring.
Such a $T$ is an integral domain, so it has a quotient field,
which we denote $q(T)$.  It is known, see \cite{hwalg}, Cor.~1.3,
that $T$ is integrally closed in $q(T)$.  An extensive theory of
graded algebraic extensions of graded fields has been developed in
\cite{hwalg}.
For a graded field $T$, we can define a
grading on the polynomial ring $T[x]$ as follows:
Let  $\Delta$ be a totally ordered abelian group with
$\Gamma_T \subseteq \Delta$, and fix
$\theta \in \Delta$.  We have
\begin{equation}\label{homogenizable}
T[x]  \,  = \,
\textstyle\bigoplus\limits_{\ga \in
\Delta} T[x]_{\ga}, \quad
\mathrm{where}\quad
T[x]_{\ga} \  = \  \{ \textstyle\sum a_{i} x^{i} \mid a_{i} \in T^{h}, \
\deg(a_{i}) +i \theta = \ga \}.
\end{equation}
This makes $T[x]$ a graded ring, which we denote $T[x]^{\theta}$.
Note that $\Gamma_{T[x]^{\theta}} = \Gamma_T +
\langle \theta \rangle$.
 A homogeneous polynomial in $T[x]^{\theta}$
 is said to be
{\it $\theta$-homogenizable}.
If $E$ is a graded division
algebra with center $T$, and
$a \in E^{h}$ is homogeneous of
degree $\theta$, then the evaluation homomorphism
$\epsilon_a\colon  T[x]^{\theta} \ra T[a]$
given by  $f \mapsto f(a)$ is a graded ring
homomorphism.  Assuming $[T[a]:T]<\infty$, we
have $\ker(\epsilon_a)$ is a principal ideal of
$T[x]$ whose unique monic generator $h_a$ is called the
minimal polynomial of~$a$ over $T$. It is known, see \cite{hwalg},
Prop.~2.2, that
if $\deg(a) = \theta$, then $h_a$ is $\theta$-homogenizable.

If $E$ is a graded division ring, then its center $Z(E)$ is clearly
a graded field.  {\it The graded division rings considered in
this note will always be assumed finite-dimensional over their
centers.}  The finite-dimensionality assures that $E$
has a quotient division ring $q(E)$ obtained by central localization,
i.e., $q(E) = E \otimes_T q(T)$ where $T = Z(E)$. Clearly,
$Z(q(E)) = q(T)$ and $\ind(E) = \ind(q(E))$, where
the index of $E$ is defined by $\ind(E)^2 = [E:T]$.
If $S$ is a graded field which is a graded subring of $Z(E)$
and $[E:S] <\infty$,
then $E$ is said to be a {\it graded division algebra} over~$S$.

A graded division algebra $E$ with center $T$ is
said to be {\it unramified} if $\Gamma_E=\Gamma_T$.
From~(\ref{fundeq}), it follows then that $[E:S]=[E_0:T_0]$.
At the other extreme, $E$ is said to be {\it totally ramified}
if $E_0=T_0$. In a case in the middle, $E$ is
said to be {\it semiramified} if
 $E_0$ is a field and ${[E_0:T_0]=|\Gamma_E:\Gamma_T|=\ind(E)}$.
These definitions are motivated by  analogous  definitions for
valued division algebras (\cite{Wad99}).
Indeed,  if a valued division algebra
is unramified, semiramified, or totally ramfied, then so is its
associated graded division
algebra.

\subsection{Commutators of graded division rings}

In a division ring, additive and multiplicative commutators play
important roles and there are extensive results in the literature
known as commutativity theorems. The main theme in these results
is that, additive and multiplicative commutators are ``dense'' in
a division ring.  For example, if an element commutes with all
additive commutators, then it is already a central element. It
seems that this
trend continues  for the additive commutators for a graded
division ring as we will see in this section. However the multiplicative commutators are
too ``isolated'' to determine the structure of a graded division
ring.

Let $E$ be  a graded division ring with graded center $T$. A
\emph{homogeneous additive commutator} of $E$ is an element of the
form $ab-ba$ where $a,b \in E^{h}$. We will use the notation
$\lbr a,b\rbr
= ab-ba$ for $a,b \in E^{h}$ and let $\lbr H, K\rbr$
be the additive
group generated by
$\{ dk-kd : d \in
H^h, k \in K^h\}$ where $H$ and $K$ are graded subrings of $E$.
Parallel to  the theory of division rings, one can show that if all
the homogenous additive commutators of graded division ring $E$ are
central, then $E$ is a graded field. To observe this, one can carry
over the non-graded proof, {\it mutatis-mutandis}, to the graded
setting, see, e.g., \cite{Lam01}, Prop.~13.4. Alternatively, let $y \in E^h$ be an element which commutes
with homogeneous additive commutators of $E$. Then  $y$ commutes
with all (non-homogeneous) commutators of $E$. Consider $\lbr x_1 ,
x_2\rbr$ where $x_1 , x_2 \in q(E)$. Since $q(E)=E\otimes_T q(T)$,
it follows
that $y\lbr x_1 , x_2 \rbr = \lbr x_1 , x_2 \rbr y$. So $y$ commutes
 with all
commutators of $q(E)$, a division ring, thus $y \in q(T)$. But $
E^h \cap
q(T) \subseteq T^h$, proving that $y \in T^h$. Thus, $E$~is
commutative. Again parallel to the theory of division rings, one can
prove that if $K \subseteq E$ are graded division rings, with
$\lbr E,K\rbr
\subseteq K$ and $\chr(K) \neq 2$, then $K \subseteq Z(E)$ . However,
for this one it seems there is no shortcut, and one needs to
carry out a proof similar to the one for ungraded
division rings, as in  (\cite{Lam01}, Prop.~3.7, see Proposition~\ref{additivecbhthm}).

The  paragraph above shows some similar behavior between the Lie
algebra structure of division rings and that of graded division rings.
However, this analogy often fails for the multiplicative structure
of graded division algebras. For example, the Cartan-Brauer-Hua
theorem (the multiplicative analogue of the  statement above that  if $K \subseteq E$ are graded division rings, with
$\lbr E,K\rbr
\subseteq K$ and $\chr(K) \neq 2$, then $K \subseteq Z(E)$) is not
valid in the graded setting. Also, the  multiplicative group $E^*$
of a
totally ramified graded division algebra $E$ is nilpotent
(since $E' \subseteq E_0^* =T_0^* \subseteq Z(E^*)$),
while the  multiplicative group of a noncommutative division ring
is not even
solvable, cf.~\cite{Stu64}.
Furthermore,  a totally ramified graded division algebra
$E^*$ is radical over its
center $T$ (since $E^{*\exp(\Gamma_E/\Gamma_T)} \subseteq T^*$), but
this is not the case for any non-commutative division
ring (\cite{Lam01}, Th.~15.15). Nonetheless, one significant
theorem involving conjugates that can be extended to
the graded setting is the Wedderburn factorization theorem, Theorem~\ref{domainW}.

We first establish a relation between the support of $E$ and that of the additive commutator subgroup (see Millar's Thesis~\cite{millar} for more on this theme).

\begin{lem}
Let $E = \bigoplus_{\ga \in \Ga} E_{\ga}$ be a graded division
algebra over its centre $T$.
\begin{enumerate}
\item If $E$ is totally ramified, then $\emptyset
\neq \Supp ([E,E]) \subsetneqq \Ga_E$.

\item If $E$ is not totally ramified, then $\Supp (E) = \Supp
([E,E])$.
\end{enumerate}
\end{lem}

\begin{proof}
(1) Clearly $\emptyset \neq \Supp ([E,E]) \subseteq \Ga_E$. Since
$E_0 = T_0 = Z(E) \cap E_0$ we have $E_0 \subseteq Z(E)$. Suppose $0
\in \Supp([E,E])$. Then there is an element $\sum_i (x_i y_i - y_i
x_i) \in [E,E]$, with $\deg(x_i) + \deg(y_i) = 0$ for all $i$. If
$x_i y_i - y_i x_i = 0$ for all $i$, then clearly the sum is also
zero. Thus there are non-zero homogeneous elements $x \in E_{\ga}$,
$y \in E_{\de}$ with $0 \neq xy -yx \in E_0$ and $\ga + \de =0$.

Then $(xy-yx)y^{-1} \neq 0$, as $y^{-1} \in E_{-\de} \mi 0$ and
$xy-yx \in E_{0} \mi 0$, so their product is a non-zero homogeneous
element of degree $-\de$. Since
$$
(xy-yx)y^{-1} \;\; = \;\:\, xy y^{-1} - y x y^{-1} \;\; = \;\:\,
y^{-1} y x - y x y^{-1},
$$
we have $y^{-1} (yx) \neq (yx) y^{-1}$; that is $yx \notin Z(E)$.
Since $yx \in E_0$, this contradicts the fact that $E_0 = T_0$, so
$0 \notin \Supp([E,E])$.

\vspace{3pt}

(2) It is clear that $\Supp ([E,E]) \subseteq \Ga_E$. For the
reverse containment, for $\ga \in \Ga_E$ we will show that there is
an $x \in E_{\ga}$ which does not commute with some $y \in E_{\de}$
for some $\de \in \Ga_E$. Suppose not, then $E_{\ga} \subseteq
Z(E)$, so $E_{\ga} = T_{\ga}$. Let $x \in E_{\ga}$, $d \in E_0$, $y
\in E_{\de}$ be arbitrary non-zero elements. Then
\begin{align*}
x(d y) \;\, = \;\, (d y) x  \;\, =\,\;d( y x)\;\, =\;\, d(x y) \;\,
= \;\, (d x) y \;\, = \;\, y(d x)  \;\, =\,\;(y d) x\;\, =\;\,x (y
d).
\end{align*}
So for all $d \in E_0$, $y \in E_{\de}$ we have $x(d y) =x(y d)$.
Since $x$ is a non-zero homogeneous element, it is invertible. This
implies $d y = y d$, so $E_0 = T_0$ contradicting the fact that $E$
is not totally ramified. Then there is an $x \in E_{\ga}$ which does
not commute with $y \in E_{\de}$, so $x y y^{-1} - y^{-1} x y \neq
0$ proving $\ga \in \Supp ([E,E])$.
\end{proof}

Similarly to the theory of division ring, one can establish the following.

\begin{prop}\label{additivecbhthm}
Let $K \subseteq E$ be graded division rings, with $[E,K] \subseteq
K$. If $\chr K \neq 2$, then $K \subseteq Z(E)$.

\end{prop}

\begin{proof}
First note that the condition $[E,K] \subseteq K$ is equivalent to
$[E^{h},K^{h}] \subseteq K^{h}$. Let $a \in E^{h} \mi K^{h}$ and $c
\in K^{h}$. We will show $ac=ca$. We have $[a, [a,c]] + [a^{2},c] =
2a[a,c] \in K$. If $[a,c] \neq 0$, then, since char$K \neq 2$, this
implies $a \in K$, contradicting our choice of $a$. Thus $[a,c] =0$.

Now let $b,c \in K^{h}$. Consider $a \in E^{h} \mi K^{h}$. Then $a,
ab \in E^{h} \mi K^{h}$ and we have shown that $[a,c], [ab,c]=0$.
Then $[b,c]= a^{-1} \cdot [ab,c] =0$. It follows that $K \subseteq
Z(E)$.
\end{proof}

\subsection{The graded Wedderburn factorization theorem} \label{weddapp}

 Let $D$ be a noncommutative division ring with center $F$,
 and  let $a \in D$ with
   minimal polynomial $f$ in~$F[x]$. Any conjugate of $a$ is also a
root of this polynomial. Since the number of conjugates of $a$ is
infinite (\cite{Lam01}, 13.26), this suggests that $f$ might split
completely in $D[x]$. In fact, this is the case, and it is called the
Wedderburn factorization theorem. We now carry over this theorem to
the setting of graded division algebras.
(This is used in proving Th.~\ref{normalthm}).

\begin{thm}[Wedderburn Factorization Theorem]\label{domainW}
Let $E$ be a graded division ring with center $T$ (with
$\Gamma_E$ torsion-free abelian).
Let  $a$ be
a homogenous element   of $E$ which is algebraic over $T$ with
minimal
polynomial $h_a \in T[x]$.
Then, $h_a$ splits completely in $E$. Furthermore, there exist
$n$ conjugates $a_{1}, \ldots , a_{n} $  of~$a$ such
that  $h_a = (x-a_{n})(x-a_{n-1}) \ldots (x-a_{1})$ in $E[x]$.
\end{thm}

\begin{proof} The proof is similar to Wedderburn's  original proof
for a division ring (\cite{wedd},
see also~\cite{Lam01} for a nice account of the proof).
We sketch the proof for the convenience of the reader.
For $f = \sum c_i x^i\in E[x]$ and $a\in E$, our convention is
that $f(a)$ means $\sum c_ia^i$. Since $\Gamma_E$ is torsion-free,
we have $E^* = E^h\setminus\{0\}$.
\begin{description}
\item[I] Let $f \in E[x]$ with factorization  $f=gk$ in  $E[x]$.
If $a\in E$ satisfies
$k(a) \in T\cdot E^*$, then ${f(a) =
g(a')k(a)}$, for some conjugate  $a'$ of $a$.
(Here $E$ could be  any ring with $T\subseteq Z(E)$.)
\end{description}

\begin{proof}
Let $g = \sum b_i x^i$. Then, $f = \sum b_i k x^i $, so
$f(a) = \sum b_i k(a) a^i$. But, $k(a)=te$, where $t \in T$ and
$e\in E^*$. Thus,  $f(a)=\sum b_i te a^i = \sum
b_i e a^i e^{-1} te = \sum b_i (e a e^{-1})^{i} te = g(ea
e^{-1})k(a)$.
\end{proof}

\begin{description}
\item[II]
Let $f \in E[x]$ be a non-zero polynomial. Then
$r \in E$ is a root of $f$ if and
only if $x-r$ is a right divisor of $f$ in $E[x]$.
(Here, $E$ could be  any ring.)
\end{description}

\begin{proof}
We have $x^i-r^i = (x^{i-1} + x^{i-2}r + \ldots + r^{i-1})(x-r)$
for any $i\ge1$.  Hence,
\begin{equation}\label{x-r}
f - f(r) = g\cdot (x-r)
\end{equation} for some
$g\in E[x]$.  So, if $f(r)= 0$, then $f = g\cdot(x-r)$.
Conversely, if $x-r$ is a right divisor of $f$, then
equation~\eqref{x-r} shows that $x-r$ is a right divisor of the
constant $f(r)$.  Since $x-r$ is monic, this implies that $f(r) = 0$.
\end{proof}

\begin{description}
\item[III]
If a non-zero monic polynomial $f \in E[x]$
vanishes identically on the conjugacy class $A$ of~$a$ (i.e.,
$f(b)=0$ for all  $b \in A$), then $\deg(f) \geq \deg(h_a)$.
\end{description}

\begin{proof}
Consider $f = x^m + d_1
x^{m-1} + \ldots + d_m \in E[x]$  such that $f(A)=0$ and $m <
\deg(h_a)$ with $m$ as small as possible.
Suppose $a\in E_\gamma$, so $A\subseteq E_\gamma$, as
the units of $E$ are all homogeneous. Since the
$E_{m\gamma}$-component
of $f(b)$ is $0$ for each $b\in A$, we may assume that each
$d_i \in E_{i\gamma}$.
Because $f \notin T[x]$,
some $d_i \notin T$.
Choose~$j$ minimal with $d_j\notin T$, and some  $e \in E^*$
such that $e d_j \neq d_j e$. For any $c \in E$, write $c' :=
ece^{-1}$. Thus $d'_j \neq d_j$ but $d_\ell' = d_\ell$
for $\ell<j$. Let
$f' = x^m +d_1'x^{m-1} +\ldots +d_m' \in E[x]$.
 Now, for all $b \in A$, we have  $f'(b') =[f(b)]'= 0' = 0$. Since
$eAe^{-1} = A$, this shows that $f'(A) = 0$.  Let $g = f-f'$,
which has degree $j < m$ with leading coefficient $d_j-d_j'$.
Then, $g(A) = 0$.  But, $d_j-d_j' \in E_{j\gamma}\setminus \{0\}
\subseteq E^*$.  Thus, $(d_j-d_j')^{-1}g$ is monic of degree $j <m$
in $E[x]$,
and it vanishes on $A$.  This contradicts the choice of $f$; hence,
$m\ge \deg(h_a)$.
\end{proof}

We now prove the theorem. Since $h_a(a)=0$, by (II),
$h_a \in E[x]\cdot(x-a)$. Take a factorization
$$
h_a = g \cdot (x-a_r) \ldots (x-a_1) \, ,
$$
where $g \in E[x]$, $a_1, \ldots, a_r \in A$ and $r$ is as large
as possible. Let $k = (x-a_r) \ldots (x-a_1)\in E[x]$.
We claim that
$k(A) = 0$, where $A$ is the conjugacy class of $a$. For,
suppose there exists $b \in A$ such that $k(b) \neq
0$.  Since $k(b)$ is homogenous, we have $k(b) \in E^*$.
But, $h_a = gk$, and $h_a(b) = 0$, as $b\in A$; hence,   (I)
implies that $g(b')=0$ for some  conjugate $b'$ of $b$. We can then
write $g = g_1 \cdot (x-b')$, by (II). So $h_a$ has a right factor
$(x-b')k = (x-b')(x-a_r) \ldots (x-a_1)$, contradicting our choice
of $r$.
Thus $k(A)=0$, and using~(III), we have $r\ge\deg(h_a)$, which says
that
$h_a = (x-a_r) \ldots (x-a_1)$.
\end{proof}

\begin{rem}[Dickson Theorem]
One can also see that, with the same assumptions as in
Th.~\ref{domainW}, if $a, b \in E$ have the same minimal
polynomial $h\in T[x]$, then $a$ and $b$ are conjugates. For,
$h=(x-b)k$ where $k \in T[b][x]$. But then by (III),
there exists a conjugate of $a$, say $a'$,  such that $k(a') \not
=0$. Since $h(a')=0$,  by (I) some  conjugate of $a'$ is a root
of $x-b$. (This is also deducible using the graded version of the
Skolem-Noether theorem, see \cite{hwcor}, Prop.~1.6.)
\end{rem}

\subsection{Graded division ring associated to a valued division algebra}

We recall how to associate a graded division algebra to a valued
division algebra.

Let $D$ be a division algebra finite dimensional over its
center $F$, with a valuation
$v: D^{\ast} \ra
\Ga$. So $\Ga$ is a totally ordered abelian group,
and $v$ satisifies the conditions that for all
$a, b \in D^{\ast}$,
\begin{enumerate}
\item[(i)] $ v(ab) = v(a) + v(b)$;

\item[(ii)] $v(a+b) \geq \min \{v(a),v(b) \}\;\;\;\;\; (b \neq -a).$
\end{enumerate}
Let
\begin{align*}
V_D  \ &= \  \{ a \in D^{\ast} : v(a) \geq 0 \}\cup\{0\},
\text{ the
valuation ring of $v$};\\
M_D  \ &= \  \{ a \in D^{\ast} : v(a) > 0
\}\cup\{0\}, \text{ the unique maximal left (and right) ideal
 of $V_D$}; \\
\overline{D}  \ &= \  V_D / M_D, \text{ the residue
division ring of $v$ on $D$; and} \\
\Ga_D  \ &= \  \mathrm{im}(v), \text{ the value
group of the valuation}.
\end{align*}
For background on  valued division algebras,
see \cite{jw} or the survey paper  \cite{Wad99}.
One associates to $D$ a graded division algebra
as follows:
For each $\gamma\in \Gamma_D$, let
\begin{align*}
 D^{\ge\ga}  \ &=  \
\{ d \in D^{\ast} : v(d) \geq \ga \}\cup\{0\}, \text{ an additive
subgroup of $D$}; \qquad \qquad\qquad\qquad\qquad \ \\
D^{>\ga}  \ &=  \ \{ d \in D^{\ast} : v(d) > \ga \}\cup\{0\},
\text{ a subgroup
of $D^{\ge\ga}$};   \text{ and}\\
 \gr(D)_\gamma \ &= \
D^{\ge\ga}\big/D^{>\ga}.
\end{align*}
Then define
$$
 \gr(D)  \ = \  \textstyle\bigoplus\limits_{\ga \in \Ga_D}
\gr(D)_\gamma. \ \
$$
Because $D^{>\ga}D^{\ge\de} \,+\, D^{\ge\ga}D^{>\de}
\subseteq D^{>(\ga +
\de)}$ for all $\ga , \de \in \Ga_D$, the  multiplication on
$\gr(D)$ induced by multiplication on $D$ is
well-defined, giving that $\gr(D)$ is a graded  ring, called the
{\it associated graded ring} of $D$. The
multiplicative property
(i) of  the valuation $v$ implies that $\gr(D)$ is a graded
division ring.
Clearly,
we have ${\gr(D)}_0 = \overline{D}$ and $\Ga_{\gr(D)} = \Ga_D$.
For $d\in D^*$, we write $\widetilde d$ for the image
$d + D^{>v(d)}$ of $d$ in $\gr(D)_{v(d)}$.  Thus,
the map given by $d\mapsto \widetilde d$ is
a group epimorphism $ D^* \rightarrow {\gr(D)^*}$ with
kernel~$1+M_D$.

The restriction $v|_F$ of the valuation on $D$ to its center $F$ is
a valuation on $F$, which induces a corresponding graded field  $\gr(F)$.
Then it is clear that $\gr(D)$ is a graded $\gr(F)$-algebra, and
by (\ref{fundeq})  and the Fundamental Inequality for
valued division algebras,
$$
[\gr(D):\gr(F)]  \ = \
[\overline{D}:\overline{F}] \, |\Ga_D :\Ga_F|
 \ \le  \ [D:F] \ < \infty.
$$

\subsection{Reduced Whitehead
group of a graded division algebra} \label{reducedsec}

A main theme of this part of the note is to study the correspondence
between  $\SK$ of a valued division
algebra and that of  its associated graded division algebra.
Let $A$ be an Azumaya algebra of constant rank $n^2$
over a commutative ring $R$.
Then there is a  commutative
ring $S$ faithfully flat over $R$ which splits $A$, i.e.,
$A\otimes_R S \cong M_n(S)$. For $a \in A$,
considering $a \otimes 1$ as an element of $M_n(S)$,
one then defines the {\it reduced characteristic polynomial},
 the {\it reduced trace}, and the {\it reduced
norm} of $a$ by
  $$
\chr_A(x,a) \ = \ \det(x-(a\otimes1)) \ = \
x^n-\Trd_A(a)x^{n-1}+\ldots+(-1)^n\Nrd_A(a).
$$
Using descent theory, one shows that $\chr_A(x,a)$ is independent of $S$
and the choice of isomorphism $A\otimes_RS \cong M_n(S)$, and
that $\chr_A(x,a)$ lies in $R[x]$;
furthermore,  the element $a$ is invertible in $A$ if and only if
$\Nrd_A(a)$ is invertible in $R$ (see Knus \cite{knus}, III.1.2,
and Saltman \cite{saltman}, Th.~4.3). Let $A^{(1)}$ denote the set
of elements of $A$ with the reduced
norm $1$. One then  defines the {\it reduced Whitehead group}
of $A$ to be $\SK(A)=A^{(1)}/A'$, where $A'$ denotes the commutator
subgroup of the group $A^*$ of  invertible elements of $A$.  The {\it reduced norm residue group} of $A$ is defined to be
$\SH(A)=R^*/\Nrd_A(A^*)$.  These groups are related by the exact
sequence:
\begin{equation}\label{defsk1}
1\longrightarrow \SK(A) \longrightarrow A^*/A'
\stackrel{\Nrd}{\longrightarrow} R^*  \longrightarrow \SH(A)
\longrightarrow 1
\end{equation}

Now let $E$ be a graded division algebra with
center $T$. Since $E$ is an Azumaya  algebra over~
$T$ (\cite{boulag}, Prop.~5.1 or\cite{hwcor}, Cor.~1.2), its
reduced Whitehead group $\SK(E)$ is defined.

 We have
other tools as well for computing $\Nrd_E$
and~$\Trd_E$:

\begin{prop}\label{normfacts}
Let $E$ be a graded division ring with center $T$.
Let $q(T)$ be the quotient field of
$T$, and let ${q(E) = E\otimes _T q(T)}$, which is the
quotient division ring of  $E$.
We view $E\subseteq q(E)$.
Let $n = \ind(E) = \ind(q(E))$.
Then for any $a\in E$,
\begin{enumerate}
\item $\chr_E(x,a) = \chr_{q(E)}(x,a)$,
so
\begin{equation}\label{nrdquot}
\Nrd_E(a)  \, =  \, \Nrd_{q(E)}(a)
 \quad\text{ and }
\quad \Trd_E(a)  \, = \,  \Trd_{q(E)}(a).
\end{equation}
\item If $K$ is any graded subfield of $E$
containing $T$ and $a\in K$, then
$$
\Nrd_E(a)  \,= \, N_{K/T}(a)^{n/[K:T]}
\quad \text{ and }\quad \Trd_E(a) \, =  \, \textstyle
\frac n{[K:T]}\, \Tr_{K/T}(a).
$$
\item For $\gamma\in \Gamma_E$, if
$a\in E_\gamma$ then $\Nrd_E(a)\in E_{n\gamma}$ and
$\Trd(a)\in E_\gamma$.  In particular, $E^{(1)}
\subseteq E_0$.
\item Set $\dlambda = \ind(E)\big/\big(\ind(E_0)
[Z(E_0):T_0]$\big).  If $a\in E_0$, then,
\begin{equation}\label{NrdD0}
\Nrd_E(a) \, = \, N_{Z(E_0)/T_0}\Nrd_{E_0}(a)^
 {\,\dlambda}
\in T_0
\quad \text{ and }\quad
\Trd_E(a) \, = \, \dlambda \,
\Tr_{Z(E_0)/T_0}\Trd_{E_0}(a) \in T_0.
\end{equation}
\end{enumerate}
\end{prop}

\begin{proof}
(1) The construction of reduced characteristic polynomials
described above is clearly compatible with scalar extensions
of the ground ring.
Hence, $\chr_E(x,a) = \chr_{q(E)}(x,a)$ (as we are identifying
$a\in E$ with
$a\otimes1$ in $E\otimes _T q(T)$\,).
The formulas in \eqref{nrdquot} follow immediately.

(2) Let $h_a = x^m  +t_{m-1}x^{m-1} + \ldots+ t_0 \in q(T)[x]$
be the minimal
polynomial of $a$ over~$q(T)$.  As noted in \cite{hwalg}, Prop.~2.2,
since the integral domain $T$ is integrally closed and $E$~is
integral over~$T$, we have $h_a\in T[x]$.  Let
$\ell_a = x^k + s_{k-1}x^{k-1} + \ldots +s_0\in T[x]$ be the
characteristic polynomial of the $T$-linear function on the
free $T$-module $K$ given by  $c\mapsto ac$.  By definition,
$N_{K/T}(a) = (-1)^ks_0$ and $\Tr_{K/T}(a) = -s_{k-1}$.
Since
$q(K) = K\otimes_T q(T)$, we have $[q(K):q(T)] = [K:T] = k$
and $\ell_a$ is also the characteristic polynomial for the
$q(T)$-linear transformation of $q(K)$ given by $q \mapsto aq$.
So, $\ell_a = h_a^{k/m}$.  Since
$\chr_{q(E)}(x,a) = h_a^{n/m}$ (see \cite{rei}, Ex.~1, p.~124),
we have
$\chr_{q(E)}(x,a) = \ell_a^{n/k}$. Therefore, using (1),
$$
\Nrd_E(a)  \ =  \ \Nrd_{q(E)}(a)  \ = \ \big[(-1)^ks_0\big]^{n/k}
  \ = \ N_{K/T}(a)^{n/k}.
$$
The formula for $\Trd_E(a)$ in (2) follows analogously.

(3) From the equalities
$\chr_E(x,a) = \chr_{q(E)}(x,a) = h_a^{n/m}$ noted in
proving (1) and (2), we have $\Nrd_E(a) = [(-1)^mt_0]^{n/m}$
and $\Trd_E(a) = -\frac nm t_{m-1}$.
As noted in \cite{hwalg}, Prop.~2.2, if $a\in E_\gamma$,
then its minimal polynomial $h_a$ is $\gamma$-homogenizable
in $T[x]$ as in \eqref{homogenizable} above.  Hence,
$t_0\in E_{m\gamma}$ and $t_{m-1}\in E_\gamma$.  Therefore,
$\Nrd_E(a) \in E_{n\gamma}$ and $\Trd(a) \in E_{\gamma}$.
If $a \in E^{(1)}$ then $a$ is homogeneous, since it is a
unit of $E$, and since $1 = \Nrd_E(a) \in E_{n\deg(a)}$,
necessarily $\deg(a) = 0$.

(4) Suppose $a\in E_0$.  Then, $h_a$ is $0$-homogenizable
in $T[x]$, i.e., $h_a\in T_0[x]$.  Hence,
$h_a$  is the minimal polynomial of
$a$ over the field $T_0$.  Therefore, if $L$ is any maximal subfield
of $E_0$ containing $a$, we have
$N_{L/T_0}(a) = [(-1)^mt_0]^{[L:T_0]/m}$. Now,
 $$
n/m  \ = \  \delta \ind(E_0)[Z(E_0):T_0]\big/m \ =
\  \delta\,[L:T_0]/m.
$$
Hence,
\begin{align*}
\Nrd_E(a) \ &= \ \big[(-1)^m t_0\big]^{n/m} \ = \
\big[(-1)^m t_0\big]^{\delta[L:T_0]/m} \ = \ N_{L/T_0}(a)^\de
 \\
&= \ N_{Z(E_0)/T_0}N_{L/Z(E_0)}(a)^\de \ = \
N_{Z(E_0)/T_0}\Nrd_{E_0}(a)^\de.
\end{align*}
The formula for $\Trd_E(a)$ is proved analogously.
\end{proof}

In the rest of this section we study the reduced Whitehead group
$\SK$ of a graded division algebra. As we mentioned
in the introduction, the motif is to show that working
in the graded setting is much easier than in the
non-graded setting.

The most successful approach to computing $\SK$ for
 division algebras over Henselian fields is due to Ershov
in \cite{Er83}, where three linked exact
sequences were constructed involving a division algebra $D$, its
residue division algebra $\ov D$,  and its group of
 units $U_D$ (see also \cite{Wad99},
 p.~425).  From these exact sequences, Ershov recovered  Platonov's
 examples \cite{platonov} of division algebras with nontrivial
$\SK$ and many more examples as well.
In this section we will easily  prove the graded version of
 Ershov's exact sequences (see diagram~\eqref{ershovsk1}),
 yielding formulas for $\SK$ of
unramified,  semiramified, and totally ramified
graded  division algebras. This will be used to show that $\SK$ of
a tame  division algebra over a Henselian field
coincides with $\SK$ of
its associated graded division algebra. We can then
readily deduce from the graded results
many   established
formulas in the literature for the reduced Whitehead
groups of  valued division algebras
(see Cor. ~\ref{sk1appl}).  This demonstrates the merit of the
graded approach.

If $N$ is a group, we denote by $N^n$ the subgroup of $N$
generated by all $n$-th powers of elements of $N$. A~
homogeneous multiplicative commutator of $E$, where $E$ is
a graded division ring, has the form $ab a^{-1}
b^{-1}$ where $a,b \in E^{*} = E^h \mi \{0\}$.
 We will  use the notation $[a,b]
= ab a^{-1} b^{-1}$ for $a,b \in E^{*}$.
Since $a$~and~$b$ are homogeneous, note that $[a,b] \in E_0$.
If $H$ and $K$ are
subsets of $E^{\ast}$, then $[H, K]$ denotes the subgroup of~
$E^{\ast}$
generated by ${\{[h,k] :h \in H, k \in
K\}}$. The group $[E^{\ast}, E^{\ast}]$ will be denoted by $E'$.

\begin{prop} \label{normalthm}
Let $E = \bigoplus_{\al \in \Ga} E_{\al}$ be a graded division
algebra with graded center $T$, with $\ind(E) = n$. Then,
\begin{enumerate}
\item If $N$ is a normal subgroup of
$E^{\ast}$, then $N^{n} \subseteq
\Nrd_E(N)[E^{\ast}, N]$.
\item $\SK(E)$ is $n$-torsion.
\end{enumerate}
\end{prop}

\begin{proof}
Let $a \in N$ and let $h_a \in q(T)[x]$ be the minimal polynomial of
$a$ over $q(T)$, and let $m = \deg(h_a)$.
  As noted in the proof of Prop.~\ref{normfacts},
$h_a \in T[x]$ and $\Nrd_E(a) =[(-1)^mh_a(0)]^{n/m}$.
 By the graded
Wedderburn Factorization Theorem~\ref{domainW}, we have
$h_a = (x - d_1 a d_{1}^{-1}) \ldots (x - d_{m} a d_{m}^{-1})$
where each $d_{i} \in E^{*}\subseteq E^h$. Note that
$[E^*,N]$ is a normal subgroup of $E^*$, since $N$ is normal
in $E^*$.
It follows that
\begin{align*} \Nrd_{E}(a)  \ &= \
\big(d_{1} a d_{1}^{-1} \ldots d_{m} a d_{m} ^{-1}\big)^{n/m}  \ = \
 \big([d_1 , a] a [d_2 , a] a \ldots a [d_m , a] a\big)^{n/m} \\
&= a^n d_a \;\;\;\; \mathrm{where} \;\; d_a \in [E^{\ast} , N].
\end{align*}
Therefore, $a^n = \Nrd_{E}(a) d_a^{-1}\in \Nrd_E(N)[E^*,N]$,
yielding (1). (2)  is immediate  from (1) by taking ${N = E^{(1)}}$.
\end{proof}

We recall the definition of the group
$\widehat H^{-1}(G,A)$, which will appear in
our description of $\SK(E)$. For any finite
group $G$ and any $G$-module $A$,
 define the norm map
$N_G\colon A\rightarrow A$ as follows: For any $a \in A$, let
${N_G(a)=\sum_{g \in G}ga}$. Consider the $G$-module
$I_G(A)$ generated as an abelian group by
${\{a-ga: {a \in A} \text{ and }{g \in G}\}}$.
Clearly, $I_G(A) \subseteq \ker (N_G)$. Then,
\begin{equation}\label{Hminus1}
\widehat H^{-1}(G,A) \ = \ \ker (N_G)\big/I_G(A).
\end{equation}

Recall also that for an abelian group $G$, \[G\wedge G := G\otimes_{\mathbb Z} G / \langle g \otimes g \rangle,\] where $g \in G$.

\begin{thm} \label{bigdiag}
Let $E$ be any graded division ring finite
dimensional over its center $T$.
So, $Z(E_0)$ is Galois over $T_0$;
let $G = \Gal(Z(E_0)/T_0)$. Let $\dlambda =
 \ind(E)\big/\big(\ind(E_0)\,
[Z(E_0):T_0]\big)$, and let $\mu_\dlambda(T_0)$ be
the group of those
$\dlambda$-th roots of unity lying in $T_0$.
Let
$\widetilde N = N_{Z(E_0)/T_0}\circ \Nrd_{E_0}\colon
E_0^* \to T_0^*$.
Then, the rows and column of the following
diagram are exact:
\begin{equation}\label{ershovsk1}
\begin{split}
\xymatrix{
 & & 1 \ar[d] & \\
 & \SK(E_0) \ar[r] & \ker \widetilde N/[E_0^*,E^*]
\ar[r]^-{\Nrd_{E_0}} \ar[d] & \widehat H^{-1}(G,\Nrd_{E_0}(E_0^*))
\ar[r] &  1 \\
&
\Gamma_E\big/\Gamma_T \wedge \Gamma_E\big/\Gamma_T 
\ar[r]^-{\alpha} & E^{(1)}/[E_0^*,E^*]
\ar[r]  \ar[d]^{\widetilde N}& \SK(E) \ar[r] & 1\\
 & & \mu_\dlambda(T_0) \cap  \widetilde N (E_0^*) \ar[d]&\\
 & & 1 &}
 \end{split}
 \end{equation}
\end{thm}

The map $\alpha$ in \eqref{ershovsk1} is given as follows:
For $\gamma,\delta\in \Gamma_E$, take any nonzero $x_\gamma\in
E_\gamma$ and $x_\delta\in E_\delta$.  Then,
$\alpha\big((\gamma+\Gamma_T)\wedge(\delta+\Gamma_T)\big) =
[x_\gamma,x_\delta] \ \text{mod}\ [E_0^*, E^*]$.

\begin{proof}
By Prop.~2.3 in \cite{hwcor},
$Z(E_0)/T_0$ is a Galois extension and the map
$\theta\colon E^* \rightarrow \Aut(E_0)$,
given by $e \mapsto (a\mapsto eae^{-1})$
for $a \in E_0$, induces an epimorphism
$E^* \rightarrow G =\Gal(Z(E_0)/T_0)$.
In the notation for \eqref{Hminus1} with $A
=\Nrd_{E_0}(E_0^*)$, we have $N_G$
coincides with $N_{Z(E_0)/T_0}$ on $A$.
Hence,
\begin{equation}\label{kerNG}
\ker(N_G) \ = \ \Nrd_{E_0}(\ker(\widetilde N)).
\end{equation}
Take any $e\in E^*$ and let $\sigma = \theta(e)
\in \Aut(E_0)$.
For any $a\in E_0^*$, let $h_a \in Z(T_0)[x]$ be the minimal
polynomial of $a$ over $Z(T_0)$. Then $\sigma(h_a)\in Z(T_0)[x]$
is the minimal polynomial of $\sigma(a)$ over $Z(T_0)$.
Hence, $\Nrd_{E_0}(\sigma(a)) =\sigma(\Nrd_{E_0}(a))$.
Since $\sigma|_{Z(T_0)} \in G$, this yields
\begin{equation}\label{Nrdbracket}
\Nrd_{E_0}([a,e]) \ =  \ \Nrd_{E_0}(a\sigma(a^{-1})) \ = \
\Nrd_{E_0}(a) \sigma(\Nrd_{E_0}(a))^{-1} \ \in \ I_G(A),
\end{equation}
hence $\widetilde N([a,e]) = 1$.  Thus, we have
$[E_0^*, E^*]\subseteq \ker(\widetilde N) \subseteq E^{(1)}$
with the latter inclusion from Prop.~\ref{normfacts}(4).
The formula in Prop.~\ref{normfacts}(4) also shows that
$\widetilde N(E^{(1)}) \subseteq \mu_\dlambda(T_0)$.  Thus,
the vertical maps in diagram~\eqref{ershovsk1} are well-defined, and
the column in \eqref{ershovsk1} is exact.  Because
$\Nrd_{E_0}$ maps $\ker(\widetilde N)$ onto $\ker(N_G)$ by
\eqref{kerNG} and it maps $[E_0^*,E^*]$ onto $I_G(A)$ by
\eqref{Nrdbracket} (as $\theta(E^*)$ maps onto $G$), the map
labelled $\Nrd_{E_0}$ in diagram~\eqref{ershovsk1} is surjective with
kernel $E_0^{(1)}\,[E_0^*,E^*]\big/[E_0^*, E^*]$.  Therefore, the
top row of \eqref{ershovsk1} is exact.  For the lower row,
since $[E^*,E^*] \subseteq E_0^*$ and $E^*\big/(E_0^*\, Z(E^*))
\cong \Gamma_E /\Gamma_T$, the following lemma yields an epimorphism
$\Gamma_E /\Gamma_T \wedge \Gamma_E /\Gamma_T \to
[E^*,E^*]/[E_0^*,E^*]$. Given this, the
lower row in \eqref{ershovsk1} is evidently exact.
\end{proof}

We need the following lemma whose proof is left to the reader (see~\cite{hwsk1}).

\begin{lem}\label{wedge}
Let $G$ be a group, and let $H$ be a subgroup of $G$ with $H\supseteq
[G,G]$.  Let ${B = G\big/(H\, Z(G))}$.  Then, there is an epimorphism
$B\wedge B \to [G,G]\big/[H,G]$.
\end{lem}

\begin{cor} \label{skunramthm}
Let $E$ be a graded division ring with
center $T$.

\begin{enumerate}
\item If $E$ is unramified, then $\SK(E) \cong \SK(E_0)$.

\item If $E$ is totally ramified, then
$\SK(E)\cong\mu_n(T_0)/\mu_e(T_0)$ where $n = \ind(E)$ and
$e$ is the exponent of $\Gamma_E/\Gamma_T$.

\item If $E$ is semiramified, then for
$G = \Gal(E_0/T_0) \cong
\Gamma_E/\Gamma_T$ there is an exact sequence
\begin{equation}\label{semiramseq1}
G\textstyle\wedge G  \ \to  \  \widehat H^{-1}(G,E_0^*)
 \ \to \  \SK(E) \  \to \  1.
\end{equation}

\item If $E$ has  maximal graded subfields $L$ and $K$ which are
respectively unramified and totally ramified
over $T$, then $E$ is semiramified and
$\SK(E)\cong \widehat H^{-1}(\Gal(E_0/T_0),E_0^*)$.
\end{enumerate}
\end{cor}

\begin{proof}
(1)   Since $E$ is unramified over $T$, we have
$E_0$ is a
central $T_0$-division algebra, ${\ind(E_0)=\ind(E)}$, and
$E^*=E_0^*T^*$. It follows that $G=\Gal(Z(E_0)/T_0)$ is trivial,
 and thus
$\widehat H^{-1}(G,\Nrd_{E_0}(E_0))$ is trivial;
also, $\dlambda=1$,
and from (\ref{NrdD0}), $\Nrd_{E_0}(a)=\Nrd_E(a)$ for all $a\in E_0$.
Furthermore, $[E_0^*, E^*] = [E_0^*, E_0^*T^*] = [E_0^*, E_0^*]$
as $T^*$ is central.
Plugging this information  into  the exact top row of diagram
(\ref{ershovsk1}) and
noting that the exact sequence extends to the left by
$1\rightarrow [E_0^*,E^*]/[E_0^*,E_0^*]\rightarrow \SK(E_0)$,
part (1) follows.

 (2) When $E$ is totally ramified, $E_0=T_0$, $\dlambda=n$,
$\widetilde N$ is  the
identity map on $T_0$, and \[{[E^*, E_0^*] = [E^*,T_0^*] = 1}.\]
By plugging all this into the exact column  of
diagram (\ref{ershovsk1}), it follows that
$E^{(1)}\cong \mu_n(T_0)$. Also by \cite{hwcor} Prop.~2.1,
$E'\cong\mu_e(T_0)$ where $e$ is the exponent of
the torsion abelian group $\Gamma_E/ \Gamma_T$.
Part (2) now follows.

 (3)  As recalled at the beginning of the proof of
Th.~\ref{bigdiag}, for any graded division algebra~$E$ with
center~$T$, we have $Z(E_0)$ is Galois over $T_0$, and there
is an epimorphism ${\theta\colon E^* \to \Gal(Z(E_0)/T_0)}$.
Clearly, $E_0^*$ and~$T^*$ lie in $\ker(\theta)$, so
$\theta$ induces an epimorphism $\theta'\colon \Gamma_E/\Gamma_T \to
\Gal(Z(E_0)/T_0)$.  When $E$ is semiramified, by definition
${[E_0:T_0] = |\Gamma_E:\Gamma_T| = \ind(E)}$ and $E_0$ is a field.
 Let $G = \Gal(E_0/T_0)$. Because ${|G| = [E_0:T_0] =
|\Gamma_E:\Gamma_T|}$, the map $\theta'$ must be an isomorphism.
In diagram~\eqref{ershovsk1}, since $\SK(E_0) = 1$ and clearly
$\delta = 1$, the exact top row and column yield
$E^{(1)}\big/[E_0^*, E^*] \cong \widehat H^{-1}(G, E_0^*)$.
Therefore, the exact row~\eqref{semiramseq1} follows from the
exact second row of diagram~\eqref{ershovsk1} and the
isomorphism $\Gamma_E/\Gamma_T \cong G$ given by~$\theta'$.

 (4) Since $L$ and $K$ are maximal subfields
of $E$, we have ${\ind(E) = [L:T] = [L_0:T_0] \le [E_0:T_0]}$
and $\ind(E) = [K:T] = |\Gamma_K:\Gamma_T| \le |\Gamma_E:\Gamma_T|$.
It follows from \eqref{fundeq} that these inequalities are
equalities, so $E_0 = L_0$ and $\Gamma_E = \Gamma_K$.
Hence, $E$ is semiramified, and (3) applies.
  Take any $\eta, \nu\in \Gamma_E/\Gamma_T$, and any inverse images
$a,b$ of $\eta, \nu$ in $E^*$.  The left map in \eqref{semiramseq1}
sends $\eta \wedge \nu$ to $aba^{-1} b^{-1}$ mod $I_G(E_0^*)$.
Since $\Gamma_E = \Gamma_K$, these $a$ and $b$ can be chosen in $K^*$,
so they commute.  Thus, the left map of \eqref{semiramseq1} is trivial
here, yielding the isomorphism of (4).
\end{proof}

The proof of the following Lemma is in~\cite{hwsk1}.

\begin{lem}\label{grnorm}
Let $F \subseteq K$ be fields with $[K:F] < \infty$. Let $v$ be a
Henselian valuation on $F$ such that $K$ is defectless over $F$.
Then, for every $a \in K^{\ast}$, with $\widetilde{a}$ its image in
${\mathrm{gr}(K)}^{\ast}$,
$$
\widetilde{N_{K/F} (a)}  \ = \
N_{\mathrm{gr}(K) /\mathrm{gr}(F)} (\widetilde{a}).
$$
\end{lem}

\begin{cor} \label{nrdcor}
Let $F$ be a field with Henselian valuation $v$, and let $D$ be a
tame $F$-central division algebra. Then for every $a \in
D^{\ast}$, $\Nrd_{\mathrm{gr}(D)}(\ti{a}) = \widetilde{\Nrd_D(a)}$.
\end{cor}
\begin{proof}  Recall from \S\ref{prel} that the assumption $D$ is
tame over $F$ means
that ${[D:F]=[\gr(D):\gr(F)]}$ and $\gr(F) = Z(\gr(D))$. Take any maximal
subfield $L$ of $D$ containing $a$. Then $L/F$ is defectless as
$D/F$ is defectless,
so $[\gr(L):\gr(F)] = [L:F] = \ind(D) = \ind(\gr(D))$.
  Hence,  using  Lemma~\ref{grnorm} and Prop.~\ref{normfacts}(2),
we have,
\[
\widetilde{\Nrd_D(a)}  \ = \  \widetilde{N_{L/F}(a)}
 \ = \  N_{\gr(L)/\gr(F)}(\ti{a})  \ = \  \Nrd_{\gr(D)}(\ti{a}). \qedhere
\]
\end{proof}

The next proposition will be used several times below.
It was proved by Ershov in \cite{Er83}, Prop. 2,
who refers to {Yanchevski\u\i} \cite{y} for part of the
argument.  One can find a complete proof in~\cite{hwsk1}.

\begin{prop}\label{normsurj}
Let $F \subseteq K$ be fields with Henselian valuations
 $v$ such that $[K:F]<\infty$ and $K$ is tamely ramified over $F$. Then $N_{K/F}(1+M_K)=1+M_F$.
\end{prop}

\begin{cor} \label{Dnormsurj}
Let $F$ be a field with Henselian valuation $v$, and let $D$ be an
$F$-central
division algebra which is tame with respect to $v$. Then,
$\Nrd_D(1+M_D)=1+M_F$.
\end{cor}
\begin{proof} Take any $a\in 1+M_D$ and any maximal subfield $K$ of
$D$ with $a \in K$. Then, $K$ is defectless over $F$, since $D$ is
defectless over $F$. So, $a \in 1+M_K$, and
${\Nrd_D(a)=N_{K/F}(a) \in 1+M_F}$ by the first part of the proof of
Prop.~\ref{normsurj}, which required only defectlessness, not
tameness. Thus, $\Nrd_D(1+M_D) \subseteq 1+M_F$. For the reverse
inclusion, recall from \cite{hwcor}, Prop.~4.3 that as $D$ is tame over $F$,
it has a maximal
subfield $L$ with $L$ tamely ramified over $F$. Then by
Prop.~\ref{normsurj},
$$
1+M_F \ = \ N_{L/F}(1+M_L) \ = \ \Nrd_D(1+M_L) \ \subseteq  \
\Nrd_D(1+M_D) \ \subseteq  \ 1+M_F,
$$
so equality holds throughout.
\end{proof}

We can now prove the main result of this section.
\begin{thm}\label{sk1prop}
Let $F$ be a field with Henselian valuation $v$ and let $D$ be a
tame $F$-central division algebra. Then
$\SK (D) \cong \SK (\mathrm{gr}(D))$.
\end{thm}
\begin{proof}
Consider the canonical surjective group homomorphism
$\rho\colon D^{\ast} \ra
\gr(D)^{\ast}$
given by $a \mapsto \ti{a}$.
Clearly, $\ker(\rho) = 1+M_D$.  If
$a\in D^{(1)} \subseteq V_D$ then $\ti a \in \gr(D)_0$ and
by Cor.~\ref{nrdcor},
$$
\Nrd_{\gr(D)}(\ti a) \ = \ \ti{\Nrd_D(a)} \ = \,1.
$$
This shows that $\rho (D^{(1)}) \subseteq \gr(D)^{(1)}$. Now
consider the diagram
\begin{equation} \label{diagram1}
\begin{split}
\xymatrix{1 \ar[r] & (1+M_D) \cap D' \ar[r] \ar[d] & D' \ar[d]
\ar[r]^-{\rho} & \gr(D)' \ar[d] \ar[r] & 1 \\
1 \ar[r] & (1+M_D)\cap D^{(1)} \ar[r] & D^{(1)} \ar[r] &
\gr(D)^{(1)} \ar@{.>}[r] & 1}
\end{split}
\end{equation}
The top row of the above diagram  is clearly exact. The Congruence
Theorem
implies that the left
vertical map in the diagram  is an isomorphism. Once we prove
that $\rho(D^{(1)}) = \gr(D)^{(1)}$, we will have the exactness of
the second row of diagram (\ref{diagram1}), and the theorem
follows by the exact sequence for cokernels.

To prove the needed surjectivity, take any $b \in \gr(D)^{\ast}$
with $\Nrd_{\gr(D)}(b) =1$. Thus $b \in \gr(D)_0$ by Th.
~\ref{normalthm}. Choose $a
\in V_D$ such that $\ti a=b$. Then we have,
$$
\overline{\Nrd_D(a)} \ = \ \ti{\Nrd_D(a)} \
= \ \Nrd_{\gr(D)}(b) \ = \, 1.
$$
Thus $\Nrd_D(a)
\in 1+M_F$. By Cor.~\ref{Dnormsurj}, since $\Nrd_D(1+M_D)=1+M_F$,
there is $c \in 1+M_D$ such that $\Nrd_D(c) = \Nrd(a)^{-1}$.  Then,
$ac\in D^{(1)}$ and $\rho(ac) =\rho(a) = b$.
\end{proof}

Having now
established that the reduced Whitehead group of a division algebra
coincides with that of its associated graded division algebra,
we can easily  deduce
stability of $\SK$ for unramified valued division algebra, due
originally to
Platonov (Cor.~3.13 in \cite{platonov}), and also a formula for $\SK$ for a
totally ramified division algebra (\cite{LT93}, p.~363, see
also \cite{Er83}, p.~70), and also a formula for $\SK$ in the nicely
semiramified case (\cite{Er83}, p.~69), as  natural consequences
of Th.~\ref{sk1prop}:

\begin{cor}\label{sk1appl}
Let $F$ be a field with Henselian valuation, and let $D$ be a
tame division algebra with center $F$.

\begin{enumerate}
\item If $D$ is unramified then $\SK(D) \cong \SK(\overline D)$

\item If $D$ is totally ramified then
$\SK(D)\cong\mu_n(\overline F)/\mu_e(\overline F)$
where $n = \ind(D)$ and $e$ is the exponent of~
$\Gamma_D /\Gamma_F$.

\item  If $D$ is semiramified, let
$G = \Gal(\ov D/\ov F) \cong \Gamma_D/\Gamma_F$. Then, there is
an exact sequence
\begin{equation}\label{maps}
G\wedge G  \ \to  \  \widehat H^{-1}(G,\ov D^*)
 \ \to \  \SK(D) \  \to \  1.
\end{equation}

\item If $D$ is nicely semiramified, then $\SK(D)\cong
\widehat H^{-1}(\Gal(\overline D/\overline F),\overline D^*)$.
\end{enumerate}
\end{cor}

\begin{proof}  Because $D$ is tame, $Z(\gr(D)) = \gr(F)$
and $\ind(\gr(D)) = \ind(D)$.  Therefore, for $D$ in each case
(1)--(4) here, $\gr(D)$ is in the corresponding case of
Cor.~\ref{skunramthm}.  (In case (3), that $D$ is semiramified
means  $[\ov D:\ov F] = |\Gamma_D:\Gamma_F| = \ind(D)$
and $\ov D$ is a field.  Hence $\gr(D) $ is semiramified.
In case (4),  since
$D$ is nicely semiramified, by definition (see \cite{jw}, p.~149)
it contains maximal subfields $K$ and
$L$, with $K$ unramified over $F$ and $L$ totally ramified
over~$F$.  (In fact, by \cite{mou}, Th.~2.4, $D$ is nicely
semiramified if and only if it has such maximal subfields.)
Then, $\gr(K)$ and $\gr(L)$ are maximal graded subfields
of $\gr(D)$ by dimension count and the graded double centralizer
theorem,\cite{hwcor}, Prop.~1.5(b),
with $\gr(K)$~unramified over~$\gr(F)$ and $\gr(L)$
totally ramified over $\gr(F)$. Then, so is $\gr(D)$ in the case~(4)
of Cor.~\ref{skunramthm}.)
Thus, in each case Cor.~\ref{sk1appl} for $D$ follows from
Cor.~\ref{skunramthm} for $\gr(D)$ together with the isomorphism
$\SK(D) \cong \SK(\gr(D))$ given by Th.~\ref{sk1prop}.
\end{proof}
\section{The group $G(D)$ and the existence of maximal subgroups}
In previous sections we have seen that the group $\CK(D)$ is
nontrivial in several cases. But the question whether $\CK(D)\neq 1$
in the general case is still open. Now, consider a division algebra
of index $n$ with $\CK(D)\neq1$. Since $\CK(D)$ is abelian of
bounded exponent $n$ one can easily conclude that in this case $D^*$
has a (normal) maximal subgroup. In \cite{AkMa02} it is conjectured
that each finite dimensional division algebra has a maximal (not
necessarily normal) subgroup. In \cite{HW09}, in an attempt to prove
this conjecture it was shown that if $D^*$ has no normal maximal
subgroup then there exists a noncyclic division algebra of a prime
index $p>3$. Also the authors showed that every quaternion division
algebra has a maximal subgroup. Thus the existence of maximal
subgroups reduces to the nontriviality of $\CK(D)$. Now, consider
the group $G(D)=D^*/\operatorname{Nrd}_D(D^*)D'$. By Corollary
\ref{cord3}, this group is also abelian of bounded exponent $n$.
Thus if $G(D)\neq1$ then we can conclude that $D^*$ has a (normal)
maximal subgroup. At the other extreme since $\CK(D)$ is a
homomorphic image of $G(D)$ it seems that showing the nontriviality
of $G(D)$ is easier than that of $\CK(D)$. In \cite{KM05}
Keshavarzipour and Mahdavi-Hezavehi proved that if the $\deg(D)=p^m$
for some prime $p$ and $\mu_p\subset Z(D)$ then triviality of $G(D)$
forces that $D$ is the quaternion division algebra. However, a
general criterion for the triviality of $G(D)$ was provided in
\cite{MaMo12}. This section is devoted to presenting this criterion and
then applying this criterion to achieve the main result of
\cite{HW09}. We begin our study with the following lemma. Before
stating it  we recall that a field $F$ is called \textit{Euclidean}
if $F^*=F^{*2}\times \langle-1\rangle$ and every sum of two squares
is again a square. In what follows $C_m$ stands for the cyclic group
of order $m$.
\begin{lem}\label{lemg1}
Let $F$ be a field such that $F^*=F^{*m}\times C_m$ and $p$ a prime
dividing $m$. Then, for a cyclic extension $L/F$ of degree $p$ the
following statements are equivalent:
\begin{enumerate}
  \item $N_{L/F}(L^*)\neq F^*$;
  \item $L=F(\sqrt{-1})$ and $F$ is Euclidean.
\end{enumerate}
\end{lem}
\begin{proof}
(1)$\Rightarrow$(2) Clearly, $F$ contains a primitive $p$-th root of
unity. Thus, by Kummer theory $L=F(\alpha)$ for some $\alpha\in L$
with $\alpha^p=b\in F^*\setminus F^{*p}$. But, from
$F^*=F^{*m}\times C_m$ we have $F^{*m}=F^{*m^2}$ and so
$F^{*m}=F^{*mp}$ because $p$ divides $m$. Therefore,
$$\frac{F^*}{F^{*p}}\cong \frac{F^{*m}\times C_m}{F^{*mp}\times (C_m)^p}
=\frac{F^{*m}\times C_m}{F^{*m}\times (C_m)^p}\cong C_p.$$ Hence
$F^*=\langle b\rangle F^{*p}$. Now, let $p$ be odd. Since the
minimal polynomial of $\alpha$ over $F$ is $x^p-b$, we obtain
$N_{L/F}(\alpha)=(-1)^{p+1}b =b$ and hence $b\in N_{L/F}(L^*)$. On
the other hand, $F^{*p}\subseteq N_{L/F}(L^*)$ implies $F^*= \langle
b\rangle F^{*p}\subseteq N_{L/F}(L^*)$ and so $N_{L/F}(L^*)=F^*$,
which is a contradiction. Thus, $p=2$. Now, consider the case in
which $-1\in F^{*2}$. Then $\sqrt{-1}\in F$. Therefore,
$N_{L/F}(\sqrt{-1}\alpha)=b$ and consequently $N_{L/F}(L^*)=F^*$
which is also a contradiction. So we have $-1\not\in F^{*2}$,
$F^*=\langle-1\rangle\times F^{*2}$ and $L=F(\sqrt{-1})$. Since
$F^{*2}\subseteq N_{L/F}(L^*)$ and $N_{L/F}(L^*)\neq F^*$ we obtain
$N_{L/F}(L^*)=F^{*2}$. Now, given $u,v\in F^*$, from
$L=F(\sqrt{-1})$ we have $u^2+v^2=N_{L/F}(u+v\sqrt{-1})\in F^{*2}$
which means that $F$ is Euclidean. The converse (2)$\Rightarrow$(1)
is clear.
\end{proof}
To proceed our study we need to recall Merkurjev's Theorem and
Albert's Main Theorem from the theory of central simple algebras.
For the proofs see the Corollary of \cite{Mer83} and \cite[p.
109]{Alb61}, respectively.
\begin{thm}[Merkurjev]\label{thmg2}
Let $F$ be a field and $p$ be a prime number. If $[F(\mu_p):F]\leq
3$ then the ${_p}Br(F)$ (the subgroup of $Br(F)$ that is formed by
all the elements killed by $p$) is generated by cyclic algebras of
degree $p$.
\end{thm}
\begin{thm}[Albert]\label{thmg3}
Let $F$ be a field with nonzero characteristic $p$. Then every
$p$-algebra is similar to a cyclic $p$-algebra.
\end{thm}
Now, let $F$ be a field with $F^*=F^{*p}$ for some prime number $ p
$. If $[F(\mu_p):F]\leq 3$, then by Merkurjev's Theorem, ${_p}Br(F)$
is generated by classes of cyclic $F$-algebras of degree $p$ and
thus every $F$-central simple algebra splits. So $Br(F)_p=0$. Also
if $\operatorname{char}(F)=p$ and $[A]\in {_p}Br(F)$, then by
Albert's Main Theorem $[A]=[(a,L/F,\sigma)]$ for some cyclic field
extension $L/F$, where $[L:F]$ is a $p$-th power. Since
$F^*=F^{*p^e}$ for each $e\in \mathbb{N}$, we conclude that
$N_{L/F}(L^*)=F^*$. So $[A]=0$ in $Br(F)$ and again we obtain
$Br(F)_p=0$.
\begin{lem}\label{lemg4}
Given a field $F$, we have the following:
\begin{enumerate}
  \item Let ${_2}Br(F)\neq 0$. Then $F^{*2}=F^{*4}$ if and only if $F$
is Euclidean;
  \item If $n$ is odd, then $F^*=F^{*n}$ if and only if $F^{*2}=F^{*2n}$;
  \item Let $D$ be a non-commutative $F$-central division algebra of index $n$. If
  $F^{*n}=F^{*n^2}$, then $F^{*2}=F^{*2n}$.
\end{enumerate}
\end{lem}
\begin{proof}
(1) If $ F $ is Euclidean, we clearly have $F^{*2}=F^{*4}$ .
Conversely, let $F^{*2}=F^{*4}$. If $\operatorname{char}(F)=2$, then
$F^{*}=F^{*2}$. Thus the above argument yields ${_2}Br(F)=0$ which
contradicts our assumption. Thus, we may assume that
$\operatorname{char}(F)\neq 2$ and hence $-1\in F$. If $-1$ is a
square, then $F^*$ is $2$-divisible and again ${_2}Br(F)=0$ which is
absurd. So, suppose $-1$ is not a square. So that
$F^*=F^{*2}\times\langle-1\rangle$. Since ${_2}Br(F)\neq 0$ there is
a cyclic algebra of degree 2 over $F$ which is a division algebra.
If $L$ is its maximal subfield then $N_{L/F}$ is not surjective.
Hence, the result follows from Lemma \ref{lemg1}.

(2) If $F^*$ is $n$-divisible then, one may easily check that
$F^{*2}=F^{*2n}$. Conversely, suppose that $F^{*2}=F^{*2n}$. If
$a\in F$ then, there is a $b\in F^*$ such that $a^2=b^{2n}$. Since
$n $ is odd $a=(\pm b)^n\in F^{*n}$ and hence $F^{*}\subseteq
F^{*n}$, as desired.

(3) Firstly, let $\mu_n(F)=1$. Thus $F^*=F^{*n}$. If $n$ is even,
$\operatorname{char}(F)=2$ and the $2$-primary component of $D$
splits which is a contradiction. So $n$ is odd and the result
follows from (2). Now, it remains to consider the case in which
$\mu_n(F)\neq 1$. If this is the case, then $\mu_n(F)$ is a cyclic
group $C_m$ for some $m$ dividing $n$. Clearly, we have
$F^*=F^{*n}C_m$. Now, we claim that $F^{*n}\cap C_m=1$. For let
$\mu_p\subset F^{*n}\cap C_m$ for some prime $p$ dividing $m$. Since
$F^{*n}=F^{*np}$ the $p$-Sylow subgroup $ P $ of $C_m$ is contained
in $F^{*n}$. But $C_m=P\times H$ for some cyclic subgroup $H$ of
$C_m$ of order $m/|P|$. Since $F^{*n}=F^{*n^2}$ and $(p,|H|)=1$, we
conclude that $F^{*}=F^{*n}H=F^{*np}H^p=F^{*p}$. Now, the argument
before the lemma yields $Br(F)_p=0$ which is a contradiction and the
claim is established. Therefore, we obtain $F^*=F^{*n}\times C_m$
and hence $F^*=F^{*m}\times C_m$. If $m$ has an odd prime divisor,
then by Lemma \ref{lemg1}, $N_{L/F}(L^*)=F^*$ for every cyclic
extension $L/F$ of index $p$. Hence, by Merkurjev's Theorem
${_p}Br(F)=0$ which is a contradiction. Therefore, $m=1$ if
$n=\deg(D)$ is odd and $m=2^e$ for some $e\geq 0$ if $n$ is even.
However, the case that $n$ is odd yields $F^{*}=F^{*n}$ and by (2)
the result follows. So, we assume that $n$ is even. If $e=0$, then
$C_m=1$ which is a contradiction, because $\mu_n(F)\neq 1$.
Therefore $e>0$. If $e>1$, then for every quadratic extension $L/F$,
the norm map $N_{L/F}:L^*\rightarrow F^*$ is surjective, by Lemma
\ref{lemg1}. This also gives us the contradiction ${_2}Br(F)\neq0$.
Hence $e=1$ and $F$ is Euclidean. So $F^*=F^{*n}\times C_2$ and
consequently $F^{*2}=F^{*2n}$, as desired.
\end{proof}
Recall that whenever $F$ is a Euclidean field then $F^{*2}$ defines
an ordering. One may easily check that this statement is equivalent
to the definition of a Euclidean field.

We are now in a position to find a criterion for the triviality of
the group $\overline{G}(D)=D^*/\operatorname{Nrd}_D(D^*)D^{(1)}$
which is a homomorphic image of $G(D)$. Note that $\overline{G}(D)$
is isomorphic to $\operatorname{Nrd}_D(D^*)/N_{D/F}(D^*)$ where
$N_{D/F}:D^*\rightarrow F^*$ is the norm map and the required
isomorphism is given by the reduced norm.
\begin{lem}\label{lemg5}
Let $D$ be an $F$-central division algebra of even index $n$. If
$F^{*2}$ is $n$-divisible then
\begin{enumerate}
  \item the $2$-primary component of $D$ is the ordinary quaternion
division algebra over $F$;
  \item The image of $D^*$ under the reduced norm is $F^{*2}$.
\end{enumerate}
\end{lem}
\begin{proof}
(1) Since $n$ is even, $F$ is Euclidean by Lemma \ref{lemg4}. Put
$L=F(\sqrt{-1})$. Given $a+\sqrt{-1}b\in L$, if we set
$$v=\frac{a}{\sqrt{a^2+b^2}}\ ,\
u=\sqrt[4]{a^2+b^2}(\sqrt{\frac{1+v}{2}}+\operatorname{sign(b)}\sqrt{-1}\sqrt{\frac{1-v}{2}}),$$
where
$$
\operatorname{sign}(b)=\left\{
   \begin{array}{ccc}
     1 & , & b>1 \\
     -1 & , & b< 1 \\
   \end{array}
 \right.,
$$
then $u^2=a+\sqrt{-1}b$. This yields $L^*=L^{*2}$ and so
$Br(L)_2=0$. Therefore the $2$-primary component of $D$ splits by
$L$. Hence, by Theorem 7 of \cite[p. 64]{Dra83} there exists an
$F$-central simple algebra $A$ containing a copy of $L$ as a maximal
subfield such that the $2$-primary component of $D$ is Brauer
equivalent to $A$. But, clearly $A$ is the ordinary quaternion
division algebra $\mathcal{Q}$. This establishes (1).

(2) By (1) $F$ is Euclidean and the $2$-primary component of $D$ is
$\mathcal{Q}$. Thus if $a\in D^*$ then
$\operatorname{Nrd}_{D}(a)^m\in
\operatorname{Nrd}_{\mathcal{Q}}(\mathcal{Q}^*)=F^{*2}$ (cf.
\cite[Lem 5, p. 158]{Dra83}). But, $F^{*2}$ is uniquely
$m$-divisible as $F$ is Euclidean. Therefore
$\operatorname{Nrd}_{D}(a) \in F^{*2}$ and so
$$\operatorname{Nrd}_{D}(D^*)\subseteq F^{*2}=F^{*2m}\subseteq \operatorname{Nrd}_{D}(D^*).$$
This forces that $\operatorname{Nrd}_{D}(D^*)=F^{*2}$.
\end{proof}
\begin{thm}\label{thmg6}
Let $D$ be an $F$-central division algebra. Then $\overline{G}(D)=1$
if and only if $F^{*2}$ is $n$-divisible.
\end{thm}
\begin{proof}
If $\overline{G}(D)=1$, then
$$\operatorname{Nrd}_{D}(D^*)=N_{D/F}(D^*)=\operatorname{Nrd}_{D}(D^*)^n.$$
Therefore
$$\operatorname{Nrd}_{D}(D^*)^n\subseteq F^{*n}\subseteq
\operatorname{Nrd}_{D}(D^*)= \operatorname{Nrd}_{D}(D^*)^n$$ and
thus $\operatorname{Nrd}_{D}(D^*)=F^{*n}$. This gives
$F^{*n}=F^{*n^2}$ and so from Lemma \ref{lemg4} we obtain
$F^{*2}=F^{*2n}$.

To prove the converse, first suppose that $n$ is odd. This
assumption combining with Lemma \ref{lemg4} gives $F^*=F^{*n}$. So
$$\operatorname{Nrd}_{D}(D^*)=F^*=F^{*n}=N_{D/F}(D^*)$$
which is equivalent to $\overline{G}(D)=1$. Moreover, if $n$ is even
then Lemma \ref{lemg5} yields
$$\operatorname{Nrd}_{D}(D^*)=F^{*2}=F^{*2n}=N_{D/F}(D^*),$$
as desired.
\end{proof}
From Theorem \ref{thmg2}, Theorem \ref{thmg3} and Theorem
\ref{thmg6} the following corollary is immediate.
\begin{cor}\label{corg7}
Let $D$ be an $F$-central division algebra of index $p^m$ where $p$
is a prime number. If $[F(\mu_p):F]\leq 3$ then $\overline{G}(D)=1$
if and only if $D$ is the ordinary quaternion division algebra and
$F$ is Euclidean. Also, if $\operatorname{char}(F)=p$ then
$\overline{G}(D)\neq1$.
\end{cor}
Note that Theorem \ref{thmg6} and Lemma \ref{lemg5} essentially say
that, if we want to produce a division algebra of degree greater
than $2$ with trivial $\overline{G}$ then $\deg(D)$ should not be a
power of $2$. So we must have ${_p}Br(F)\neq 0$ for some odd prime
$p$ dividing $\deg(D)$; for each such $p$, ${_p}Br(F)$ should be
generated by noncyclic division algebras and so we must have
$[F(\mu_p):F]\geq 4$ (so $p\geq 5$) by Corollary \ref{corg7}.The
existence of such noncyclic division algebras is one of the oldest
and most challenging questions in the theory of division algebras.

Now, we follow our investigation on the existence of maximal
subgroups in division algebras. Before that we recall some notions
from the theory of groups. Let $\pi$ be a set of prime numbers. A
natural number $n$ is said to be a $\pi$-number if every prime
divisor of $n$ belongs to $\pi$. Recall that a group $G$ is called
$\mathfrak{F}_{\pi}$-perfect if it does not contain any subgroup of
a finite $\pi$-number index. If $\pi$ is the set of all primes,
every $\mathfrak{F}_{\pi}$-perfect group is called
$\mathfrak{F}$-perfect. The following theorem provides a criterion
for when $ D^* $ is $\mathfrak{F}_{\pi}$-perfect, where $\pi$ is the
set of all primes dividing $\textrm{ind}(D)$. It also
singles out the relation between the notion of
$\mathfrak{F}_{\pi}$-perfection and the triviality of $ G(D) $.
\begin{thm}\label{thmg8}
Given an $F$-central division algebra $ D $ of index $n$, the
following conditions are equivalent:
\begin{enumerate}
  \item $G(D)=1$;
  \item $\SK(D)=1$ and $F^{*2}$ is $n$-divisible;
  \item $\CK(D)=1$ and $F^{*2}$ is $n$-divisible;
  \item $D^*$ is $\mathfrak{F}_{\pi}$-perfect where
  $\pi$ is the set of all primes dividing $\deg(D)$.
\end{enumerate}
\end{thm}
\begin{proof}
Firstly, we observe that a routine calculation shows that the
following diagram is commutative with exact rows:
\begin{equation}\label{eqg1}
 \xymatrix{
1\ar[r] & SK_n(D)\ar[r]^{f_1}\ar[d]_{i_1}&\SK(D)\ar[r]^{g_1}
\ar[d]_{1_{\SK(D)}}
& G(D)\ar[r]^{\operatorname{Nrd}_D} \ar[d]_{\pi_1}& \overline{G}(D)\ar[r]\ar[d]_{\pi_2}& 1\\
 1\ar[r] & \displaystyle{\frac{\mu_n(F)}{Z(D')}}\ar[r]^{f_2}&\SK(D)\ar[r]^{g_2} &\CK(D)\ar[r]^{\operatorname{Nrd}_D}&
 \NK(D)\ar[r] &1 ,}
\end{equation}
where the maps $f_1,f_2,g_1,g_2$ are canonical homomorphisms, $i_1$
is induced by inclusion, $\pi_1,\ \pi_2$ are natural projections,
and
$$
SK_n(D)=\frac{\mu_n(F)\cap \operatorname{Nrd}_{D/F}(D^*)}{Z(D')\cap
\operatorname{Nrd}_{D/F}(D^*)}.
$$
(Recall that $Z(D')=\mu_n(F)\cap D'$.)

(1)$\Rightarrow$(2) If $G(D)=1$ then $\overline{G}(D)=1$ and hence
$F^{*2}=F^{*2n}$ by Theorem \ref{thmg6}. But, from the exactness of
the first row of (\ref{eqg1}), we have
\begin{equation}\label{eqg2}
\SK(D)\cong SK_n(D).
\end{equation}
Now, if $n$ is odd then $F^*$ is $n$-divisible by Lemma \ref{lemg4}.
By the Merkurjev Theorem $F$ contains no $n$-th roots of unity and hence
by (\ref{eqg2}) we conclude that $\SK(D)=1$. Moreover, if $n$ is
even then by Lemma \ref{lemg4}, $F$ is Euclidean and
$\operatorname{Nrd}_{D/F}(D^*)=F^{*2}$ (Lemma \ref{lemg5}). However,
$F^{*2}$ contains no $n$-th roots of unity as $F$ is Euclidean and
again by (\ref{eqg2}) the result follows.

(2)$\Rightarrow$(3) Since $\SK(D)=1$ we have
\begin{equation}\label{eqg3}
\CK(D)=\NK(D).
\end{equation}
On the other hand $\overline{G}(D)=1$ by Theorem \ref{thmg6}. Hence
$\NK(D)=1$ because $\pi_2$ is an epimorphism. Therefore, by
(\ref{eqg3}) we conclude that $\CK(D)=1$.

(3)$\Rightarrow$(1) Theorem \ref{thmg6} gives us
$\overline{G}(D)=1$. Furthermore, from the assumption $\CK(D)=1$ and
the exactness of the second row of (\ref{eqg1}) we obtain
\begin{equation}\label{eqg4}
\SK(D)\cong \frac{\mu_n(F)}{Z(D')}.
\end{equation}
Thus, if $n$ is odd then by Lemma \ref{lemg4} and the Merkurjev Theorem
we conclude that $\mu_n(F)=1$. Hence (\ref{eqg4}) yields $\SK(D)=1$.
This forces that $G(D)=1$ because the first row of (\ref{eqg1}) is
exact. Finally, let $n$ be even. In this case Lemma \ref{lemg4}
shows that $F$ is Euclidean and so $\mu_n(F)=\{1,-1\}$. On the other
hand, we know that in every division algebra of even degree $-1$
belongs to the commutator subgroup (\cite{Wad90}). So
$\mu_n(F)=Z(D')$ and by (\ref{eqg4}) we conclude that $\SK(D)=1$.
Now, from (\ref{eqg1}) it follows that $G(D)=1$.

(1)$\Rightarrow$(4) If $G(D)=1$, then we have $\overline{G}(D)=1$.
So, $\operatorname{Nrd}_{D}(D^*)$ is $n$-divisible and hence it is
$p$-divisible for every prime number $p$ dividing $n$. Now, suppose
that $D^*$ is not $\mathfrak{F}_{\pi}$-perfect. Thus it contains a
normal subgroup $N$ of index $k$ for some $\pi$-number $k$. By the
Main Theorem of \cite{RSegSei02}, $D^*/N$ is a finite soluble group.
Therefore, it contains a normal maximal subgroup $\overline{M}$ of
index $p$, where $p$ divides $k$ and so does $n$. Hence $D^*$
contains a normal maximal subgroup $M$ of index $p$ for some prime
divisor $p$ of $n$. But $D'\subseteq M$ as $M\lhd D^*$. Now, if
$\operatorname{Nrd}_{D}(D^*)\subseteq M$, then
$\operatorname{Nrd}_{D}(D^*)D'\subseteq M\subsetneq D^*$ and so
$G(D)\neq 1$, which is a contradiction. Moreover, if
$\operatorname{Nrd}_{D}(D^*)\nsubseteq M$ then
$D^*=\operatorname{Nrd}_{D}(D^*)M$ and hence
$$\frac{\operatorname{Nrd}_{D}(D^*)}{\operatorname{Nrd}_{D}(D^*)\cap M}\cong \frac{\operatorname{Nrd}_{D}(D^*)M}{M}
\cong \frac{D^*}{M}\cong C_p.$$ This guarantees that
$\operatorname{Nrd}_{D}(D^*)$ has a normal subgroup of index $p$
which contradicts the fact that $\operatorname{Nrd}_{D}(D^*)$ is
$p$-divisible. Therefore $D^*$ is $\mathfrak{F}_{\pi}$-perfect.

(4)$\Rightarrow$(1) It is clear.
\end{proof}

Let $\mathcal{Q}$ be the real quaternion division algebra. Example
\ref{exaf9} shows that $\mathcal{Q}((x))$ has a trivial $\CK$. But
$Z(\mathcal{Q}((x)))=\mathbb{R}((x))$ has a natural discrete rank 1 valuation
and hence $\mathbb{R}((x))^2\neq \mathbb{R}((x))^4$. So
$G(\mathcal{Q}((x)))\neq 1$ by Theorem \ref{thmg8}. Therefore, even for the
case $\CK(D)=1$, $D^*$ may have a normal maximal subgroup. This
observation shows that to approach the problem of the existence of
normal maximal subgroups in $D^*$ it is better to work with $G(D)$
instead of $\CK(D)$.

Now, let $G$ be a group (possibly nonabelian). We say that $G$ is
divisible if $G/G'$ is divisible as an abelian group. Equivalently,
$G$ is divisible if for every natural number $n$ and $g\in G$ there
are elements $b\in G$ and $h\in G'$ such that $g=b^nh$. One can
easily show that $G$ is divisible if and only if it contains no
normal maximal subgroup. In the following theorem we will see that
the notions of divisibility and $\mathfrak{F}$-perfection are the
same in the setting of division algebras.
\begin{thm}\label{thmg9}
Given an $F$-central division algebra $D$ of degree $n$, then the
following conditions are equivalent:
\begin{enumerate}
  \item $D^*$ is divisible or equivalently $D^*$ has no normal maximal subgroup;
  \item $D^*$ is $\mathfrak{F}$-perfect;
  \item $\SK(D)=1$, $F^*$ is divisible if $n$ is odd and $F^{*2}$ is
  divisible if $n$ is even.
\end{enumerate}
\end{thm}
\begin{proof}
$(1)\Rightarrow(2)$ If $D^*$ has a subgroup of finite index, then it
contains a normal subgroup of finite index. Now, a similar argument
as in the proof of $(1)\Rightarrow(4)$ in Theorem \ref{thmg8} one
can find a normal maximal subgroup which is a contradiction. So
$D^*$ is $\mathfrak{F}$-perfect.

$(2)\Rightarrow(1)$ It is clear from the definition.

$(1)\Rightarrow(3)$ If $D^*$ is divisible then from Theorem
\ref{thmg8} it follows that $\SK(D)=1$. Now, let $n$ be odd. If this
is the case, then $F^{*}$ is $n$-divisible by Theorem \ref{thmg8}.
Thus $\operatorname{Nrd}_D(D^*)=F^{*}$ and hence $F^*$ is divisible
as $D^*$ is divisible. Moreover, if $n$ is even then
$\operatorname{Nrd}_D(D^*)=F^{*2}$ by Lemma \ref{lemg5}. This
implies that $F^{*2}$ is divisible.

$(3)\Rightarrow(1)$ As in the proof of $(1)\Rightarrow(3)$ we have
$\operatorname{Nrd}_D(D^*)=F^{*}$ if $n$ is odd and
$\operatorname{Nrd}_D(D^*)=F^{*2}$ if $n$ is even. In either case
$\operatorname{Nrd}_D(D^*)$ is divisible. On the other hand by
Theorem \ref{thmg8} we know that $G(D)=1$ and so
$D^*=\operatorname{Nrd}_D(D^*)D'$. Therefore
$$\frac{D^*}{D'}\cong \frac{\operatorname{Nrd}_D(D^*)}{\operatorname{Nrd}_D(D^*)\cap D'}$$
and hence $D^*$ is divisible.
\end{proof}
Recall that if $D$ is a division algebra then thanks to the
Dieudonn$\acute{\textrm{e}}$ determinant the group
$\operatorname{GL}_n(D)/\operatorname{SL}_n(D)$ is isomorphic to
$D^*/D'$ for every natural number $n$. Therefore, Theorem
\ref{thmg6} is valid for every central simple algebra.

Let $F$ be a Euclidean field with divisible $F^{*2}$ (for example a
real closed field). If $\mathcal{Q}$ is the quaternion division
algebra over $F$ then from Theorem \ref{thmg9} it follows that $\mathcal{Q}^*$
has no normal maximal subgroup. However, in the case that
$F=\mathbb{R}$ in \cite{Ma01} it was shown that the subgroup
$\mathbb{C}\cup \mathbb{C}j$ is a (nonnormal) maximal subgroup of
$\mathcal{Q}^*$. Also, a different approach to prove that the real
quaternion has a maximal subgroup was considered in
\cite{AkEbMoSa03}. But the existence of a maximal subgroup in a
quaternion division algebra, in the general case, was demonstrated
in \cite{HW09} by a refinement of the argument given in \cite{Ma01}.
Here we are going to show what the nature of such a subgroup of
$\mathcal{Q}^*$ can be. Since $F$ is assumed to be Euclidean, then
it has a valuation ring $V$ which is determined by the ordering:
$$V=\{b\in F|\ |b|\leq n\ \textrm{for some}\ n\in \mathbb{N}\},$$
with maximal ideal
$$M=\{b\in F|\ |b|\leq 1/n\ \textrm{for every}\ n\in \mathbb{N}\}.$$
Thus, $F\setminus V$ is the set of all elements ``infinitely large''
relative to the rational numbers $\mathbb{Q}\subset F$. Also, $M$ is
the set of all elements of $F$ ``infinitesimal'' relative to
$\mathbb{Q}$. Let $P$ be the ``purely imaginary part'' of
$\mathcal{Q}$, i.e., $P=\{bi+cj+dk|\ b,c,d\in F\}$ and
$$S(P)=\{\alpha\in P|\ \norm{\alpha}=1\}$$
be the unit sphere in $P$ (note that for $x\in \mathcal{Q}$, we have
$\norm{x}=\sqrt{\operatorname{Nrd}_{\mathcal{Q}^*}(x)}$ ). Suppose
that
$$\Delta=\{\alpha\in S(P)|\ \norm{\alpha-i}\in M\},$$
the set of elements ``infinitesimally near'' to $i$. In this setting
one can observe that
$$G=\{x\in \mathcal{Q}^*|\ i^x\in \Delta\cup-\Delta\}$$
is a maximal subgroup of $\mathcal{Q}^*$ (see \cite{HW09}).

By combining the above observation, Theorem \ref{thmg9}, Corollary
\ref{corg7} and Lemma 3 of \cite{May72} which asserts that every
divisible field of characteristic zero contains all primitive roots
of unity, we obtain:
\begin{thm}\label{thmg10}
Let $D$ be an $F$-central division algebra and suppose that $D^*$
has no maximal subgroups. Then,
\begin{enumerate}
  \item If $\deg(D)$ is even, then $D=\mathcal{Q}\otimes_FE$ where
  $E$ is an $F$-central division algebra of odd degree, and $F$ is Euclidean
  (so $\operatorname{char}(F)=0$) with $F^{*2}$ divisible;
  \item If $\deg(D)$ is odd, then $\operatorname{char}(F)>0$, $\operatorname{char}(F)\nmid \deg(D)$
  and $F^*$ is divisible;
  \item In either case, there is an odd prime number $p$ dividing
  $\deg(D)$; for each such $p$ we have $[F(\mu_p):F]\geq 4$ (so $p\geq 5$) and
  ${_p}Br(F)$ is generated by noncyclic algebras of degree $p$.
\end{enumerate}
\end{thm}
Finally, we give some additional results
concerning the group $\overline{G}(D)$ and maximal subgroups of
$D^*$. For the proofs of these theorems see \cite{MaMo12}.

\begin{thm}
If $A$ and $B$ are $F$-central simple algebras of coprime degrees,
then
$\overline{G}(A\otimes_FB)\cong\overline{G}(A)\times\overline{G}(B)$.
\end{thm}
\begin{thm}
If $D$ is an $F$-central division algebra with $\overline{G}(D)=1$,
then the following assertions are equivalent:
\begin{enumerate}
  \item $D$ contains an absolutely irreducible nilpotent subgroup;
  \item $D$ is a nilpotent crossed product division algebra whose exponent in $Br(F)$ is equal to $\deg(D)$;
  \item $D$ is the ordinary quaternion division algebra and $F$ is Euclidean.
\end{enumerate}
\end{thm}
\begin{thm}
Let $D$ be a tame division algebra of index $ n $ over its Henselian
center $F$. If either of the following conditions holds, then we
have $G(D)\cong\overline{G}(D)\cong\Gamma_D/n\Gamma_D$:
\begin{enumerate}
  \item $D$ is unramified and $G(\overline{D})=1$;
  \item $D$ is totally ramified and
  $\overline{F}^*$ is $n$-divisible;
  \item $D$ is semiramified, $\overline{D}/\overline{F}$ is a cyclic
  extension,
and $N_{\overline{D}/\overline{F}}(\overline{D}^*)$ is
$n$-divisible.
\end{enumerate}
\end{thm}
\begin{thm}
Let $D$ be an $F$-central division algebra of index $n\geq1$ such
that $[D]\neq [\mathcal{Q}]$ in $Br(F)$, where $\mathcal{Q}$ is the
ordinary quaternion division algebra. Moreover, assume that
$[D]=\sum_{j=1}^k[D_j]$, where $D_j$'s are $F$-central cyclic
division algebras. If $D^*$ contains a divisible maximal subgroup,
then $D=F$.
\end{thm}
\section{Radicable division algebras}
Let $D$ be a division ring. Recall that $D$ is called \textit{right}
(\textit{left}) \textit{algebraically closed} if every polynomial
equation with coefficients in $D$ has a right (left) root in $D$.
When $D$ is of finite dimension over its center, there is a perfect
description of such a division algebra. By Niven-Jacobson Theorem if
$Z(D)$ is a real closed field then $D$ is right
as well as left algebraically closed (see \cite[p. 255]{Lam01} or 
\cite{Niv41}). Also, by a theorem of Baer if
$D$ is right algebraically closed then $D$ is the ordinary
quaternion division algebra and $Z(D)$ is a real closed field (\cite[p. 255]{Lam01}). Now,
combining these results gives
\begin{thm}\label{thmr1}
The finite dimensional right algebraically closed division algebras
are precisely the division algebras of quaternions over a real
closed field. These division algebras are also left algebraically
closed.
\end{thm}
However, our point of view here is to explore how much the above
classification depends on the group theoretic structure of $D$. To
setup our problem, for a moment we come back to the commutative
case. Let $F$ be an algebraically closed field. Then its unit group
is divisible. More exactly, $F^*$ is isomorphic to a group of one of
the following types:
\begin{itemize}
  \item[(i)] $\mathbb{Q}/\mathbb{Z}\oplus\mathbb{Q}^{\lambda}$;
  \item[(ii)] $\oplus_{q\neq p}\mathbb{Z}_{q^{\infty}}\oplus\mathbb{Q}^{\lambda}$
   for a prime number $p$.
\end{itemize}
for an infinite cardinal $\lambda$. Conversely, all groups of the
above types (i) and (ii) are isomorphic to the unit group of a
suitable algebraically closed field (cf. \cite[pp. 106-107]{Kar88}).
Moreover, by a theorem in \cite[p. 107]{Kar88} the unit group of a
field $F$ is isomorphic to the multiplicative group of a real closed
field if and only if $F^*\cong \mathbb{Q}^{\lambda}\times
\mathbb{Z}_2$ for some infinite cardinal $\lambda$.

On the other hand, a natural generalization of the notion of
divisibility in the nonabelian setting is the notion of
``radicability''. We recall that a (nonabelian) group is called
\textit{radicable} if for each $g\in G$ and each $n\in \mathbb{N}$
there exists an element $h\in G$ such that $h^n=g$.  A division ring
$ D $ is called radicable if $ D^* $ is a radicable group. In
particular, when $ D^* $ is abelian the notions of divisibility and
radicability coincide and $ D $ is called a divisible (otherwise
indivisible) field. From the definition, it follows that the
multiplicative group of each right (left) algebraically closed
division ring is radicable. So, it can be of interest to explore
what happens when the multiplicative group of a division ring is
radicable. In this direction, in \cite{MaMo11} all division algebras
with radicable multiplicative groups with indivisible center have
been classified. More precisely, the authors provided a new version
of Theorem \ref{thmr1} for radicable division algebras. Here, our
aim is to present the main theorem of  \cite{MaMo11} as a
consequence of the results appeared in the previous section. We
proceed by recalling some definitions and facts from \cite{MaMo11}.

Let $F$ be a field and $K/F$ be a finite extension. Recall that $K$
is called a radical extension if there are finite intermediate
fields $F\subseteq K_1\subseteq K_2\subseteq\ldots\subseteq K_n=K$
such that for all $1\leq i\leq n$, $K_{i+1}=K_{i}(\alpha_i)$ with
$\alpha_i^{m_i}\in K_i$ for some $m_i\in \mathbb{N}$. $F$ is called
\textit{radically closed} if it possesses no radical extensions.
\begin{prop}\label{propr2}
Let $F$ be a field. Then there exists a unique, up to isomorphism,
extension $F_{\operatorname{rad}}/F$ which is radically closed and
contained in any radically closed algebraic extension of $F$.
\end{prop}
\begin{proof}
Let $F_{\operatorname{alg}}$ be a fixed algebraic closure of $F$.
Set $L_0= F$ and for $i>0$ let $L_i$ be the splitting field of the
family $\{x^n-a|\ a\in L_{i-1}\ ,\ n\in \mathbb{N}\}$ in
$F_{\operatorname{alg}}$. Put
$F_{\operatorname{rad}}:=\cup_{i=0}^{\infty}L_i$. Clearly,
$F_{\operatorname{rad}}$ is radically closed as well as divisible.
Now, one can easily check that every radical extension of $L$ in
$F_{\operatorname{alg}}$ is contained in $F_{\operatorname{rad}}$.
Also, the uniqueness follows from the standard theorems of field
theory.
\end{proof}
Proposition \ref{propr2} provides us a facility to introduce the
notion of \textit{radical closure} and \textit{radically real closed
field}. For every field $F$, we refer to the extension
$F_{\operatorname{rad}}/F$ as the radical closure of $F$. Note that
Proposition \ref{propr2} ensures that $F_{\operatorname{rad}}$ is
independent of the choice of the algebraic closure of $F$. Moreover,
we say that $F$ is radically real closed field if $\sqrt{-1}\not\in
F$ and $F_{\operatorname{rad}}=F(\sqrt{-1})$.
\begin{exa}\label{exar3}
From field theory, recall that there are some polynomial
equations in $\mathbb{Q}[x]$ that are not soluble by radicals. Using
this fact, one can easily observe that
$\mathbb{Q}_{\operatorname{rad}}\neq\mathbb{Q}_{\operatorname{alg}}$.
Also, it is not hard to check that if $a\in
\mathbb{Q}_{\operatorname{rad}}$, then the complex conjugate of $a$
is contained in $\mathbb{Q}_{\operatorname{rad}}$, i.e.,
$\overline{a}\in \mathbb{Q}_{\operatorname{rad}}$. Now, consider the
automorphism $\sigma: \mathbb{Q}_{\operatorname{rad}}\rightarrow
\mathbb{Q}_{\operatorname{rad}}$, given by $a\mapsto \overline{a}$.
Since $\sqrt{-1}\in \mathbb{Q}_{\operatorname{rad}} $ we conclude
that $\sigma\neq 1$. Moreover, it is clear that
$\sigma^2=1_{\mathbb{Q}_{\operatorname{rad}}}$. Hence $\sigma$ has
order 2 which yields $F(\sqrt{-1})=\mathbb{Q}_{\operatorname{rad}}$
where $F$ is the fixed field of $\sigma$. This shows that $F$ is
radically real closed and is not real closed, because
$\mathbb{Q}_{\operatorname{rad}}$ is not algebraically closed.
\end{exa}
The next theorem can be viewed as an analogue of the classical
Frobenius Theorem
 (cf. \cite[p. 208]{Lam01}).
\begin{thm}\label{thmr4}
If $K$ is a divisible finite field extension of $F$, then the
following statements are equivalent:
 \begin{enumerate}
  \item $F$ is indivisible;
  \item $F$ is radically real closed;
  \item $Br(F)\cong \mathbb{Z}_2$.
 \end{enumerate}
Moreover, if one of the above conditions holds, then $F^*$ is
isomorphic to the multiplicative group of a real closed field and
the ordinary $F$-quaternion division algebra $\mathcal{Q}$ is the
only noncommutative division algebra with center $F$.
\end{thm}
\begin{proof}
(1)$\Rightarrow$(2). Let $E$ be a minimal extension of $F$ contained
in $K$ with divisible unit group. Note that since $F^*$ is not
divisible we have $1<[E:F]<\infty$. Let $L$ be a maximal subfield of
$E$ containing $F$. Clearly $L^*$ is indivisible. Now, we claim that
$L^*=N_{E/L}(E^*)\times C_m$ for some $m\neq 1$ dividing $[E:L]$. To
establish our claim, set $[E:L]=n$. Since $N_{E/L}(E^*)$ is
divisible and $L^*$ is indivisible, it follows that
$L^*=N_{E/L}(E^*)\times N$ for some non-trivial subgroup $N$ of
$F^*$. (Note that since $N_{E/L}(E^*)$ is divisible, it is injective
as a $\mathbb{Z}$-module.) On the other hand, we know that $L^{*n}$
is contained in $N_{E/L}(E^*)$. So $N\cong L^*/N_{E/L}(E^*)$ has
exponent $n$ and thus every element of $N$ is a root of the
polynomial $x^n-1$ in $L$. Therefore, $N$ is a finite cyclic
subgroup of $L^*$ of order dividing $n$, as desired. Now, let $p$ be
a prime divisor of $m$ and $P=\langle \alpha\rangle$ be the
$p$-Sylow subgroup of $C_m$. Clearly $\alpha^{1/p}\in E\setminus L$
and hence $L\subsetneqq L(\alpha^{1/p})\subseteq E$. By the
maximality of $ L $ we obtain $E=L(\alpha^{1/p})$. But $\mu_p\subset
L$ as $p$ divides $m$. Thus, by Kummer Theory $E/L$ is a cyclic
extension with $[E:L]=p$. Since $N_{E/L}(E^*)\neq L^*$, from Lemma
\ref{lemg1}, we conclude that $L$ is Euclidean. Hence
$\sqrt{-1}\not\in L$ and so $\sqrt{-1}\not\in F$. If $F(\sqrt{-1})$
is indivisible, we may apply the above argument to obtain
$\sqrt{-1}\not\in F(\sqrt{-1})$, which is a contradiction.
Therefore, $F(\sqrt{-1})$ is divisible and hence $E=F(\sqrt{-1})$.
Now, since $\operatorname{char}(F)=0$, by Lemma 3 of \cite{May72},
$E$ contains all roots of unity. Thus, for every $a\in E$ the
polynomial $x^n-a$ splits in $E[x]$. Therefore, $E$ has no proper
radical extension. This yields $F\subsetneqq
F_{\operatorname{rad}}\subseteq E=F(\sqrt{-1})$ and thus
$F_{\operatorname{rad}}=F(\sqrt{-1})$.

(2)$\Rightarrow$(3). By a similar argument as above, we have
$F^*=N_{F_{\operatorname{rad}}/F}(F_{\operatorname{rad}}^*)\times
C_2$. This clearly implies that $F^{*}=F^{*2}\times C_2$. Now, Lemma
\ref{lemg1} shows that $F$ is Euclidean and thus
$\operatorname{char}(F)=0$. Here we observe that since the ordinary
quaternion algebra $\mathcal{Q}$ over $F$ is a division algebra we
have $Br(F)\neq 0$. At the other extreme, since $F(\sqrt{-1})$ is
divisible of characteristic zero, it contains all roots of unity and
thus by Merkurjev Theorem $F(\sqrt{-1})$ has a trivial Brauer group.
Therefore, each $F$-central simple algebra splits by $F(\sqrt{-1})$.
Now, if $0\neq[D]\in Br(F)$ then by Theorem 7 of \cite[p. 64]{Dra83}
there exists an $F$-central simple algebra $A$ such that $[D]=[A]$
and $F(\sqrt{-1})$ is a maximal subfield of $A$. Thus, $A$ is a
quaternion division algebra and hence $A=\mathcal{Q}$ because the
only quaternion division algebra over a Euclidean field is the
ordinary one. This implies that $Br(F)=\{[F],[\mathcal{Q}]\}\cong
\mathbb{Z}_2$.

(3)$\Rightarrow$(1). Since $Br(F)\cong \mathbb{Z}_2$, we conclude
that ${_2}Br(F)\neq 0$. But, by Merkurejv Theorem ${_2}Br(F)$ is
generated by division algebras of degree 2. This forces that $F^*$
is indivisible, because every division algebra of degree 2 is
cyclic.

Finally, as we have seen above, we have $F^*=F^{*2}\times
\langle-1\rangle$. Now, by a theorem of \cite[p. 107]{Kar88} it
follows that $F^*$ is isomorphic to the unit group of a suitable
real closed field.
\end{proof}
\begin{thm}
Let $F$ be an indivisible field. If $D$ is an $F$-central division
algebra, then the following statements are equivalent:
 \begin{enumerate}
  \item $ D $ contains a divisible subfield $K$ containing $F$;
  \item $ F $ is radically real closed and $ D $ is the ordinary quaternion division
  algebra $\mathcal{Q}$;
  \item $D$ is radicable.
 \end{enumerate}
Furthermore, if one of the above conditions holds, then $F^*$ is
isomorphic to the multiplicative group of a real closed field.
\end{thm}
\begin{proof}
Clearly (1) and (2) are equivalent by Theorem \ref{thmr4}.

(2)$\Rightarrow$(3). Since $F(\sqrt{-1})^*$ is divisible, for each
$a\in F^*$ and $n\in \mathbb{N}$ there exists a $b\in
F(\sqrt{-1})\subseteq D^*$ such that $b^n=a$. Now, if $\alpha\in
D^*\setminus F^*$, then $F(\alpha)$ is a quadratic subfield of $D$.
But $F(\alpha)$ is separable over $F$ as $\operatorname{char}(F)=0$
(note that every radically real closed field is Euclidean and thus
has characteristic zero). Hence $F(\alpha)/F$ is a cyclic extension.
On the other hand, Kummer Theory asserts that quadratic extensions
of $F$ are in one to one correspondence with the subgroups of
$F^*/F^{*2}\cong\langle-1\rangle$. Therefore, $F(\sqrt{-1})$ is the
only quadratic extension of $F$, up to isomorphism. This forces that
$F(\alpha)\cong F(\sqrt{-1})$ and consequently $F(\alpha)$ is
divisible. So, for each $n\in\mathbb{N}$, there is a $\beta\in
F(\alpha)$ such that $\beta^n=\alpha$. This shows that $D^*$ is
radicable.

(3)$\Rightarrow$(2). By definition $D^*$ is divisible and hence from
Theorem \ref{thmg10} it follows that $D=\mathcal{Q}\otimes_FE$,
where $E$ has odd degree and $F$ is Euclidean. So we must prove that
$\deg(E)=1$. Otherwise, let $s=\deg(E)>1$. First suppose that
$F(\sqrt{-1})$ contains an $s$-th root of unity. So $\mu_p\subset
F(\sqrt{-1})$ for at least one prime dividing $s$ and for such a
prime, $[F(\mu_p):F]$ is equal or less than 3. On the other hand as
in the proof of Theorem \ref{thmr4} we can prove that
$\operatorname{Nrd}_D(D^*)$ is divisible and
$F^*=\operatorname{Nrd}_D(D^*)\times \langle-1\rangle$. Now, since
$\operatorname{Nrd}_D(D^*)$ and $\langle-1\rangle$ are $p$-divisible
we conclude that $F^*$ is $p$-divisible. At this stage, Merkurejv
Theorem yields ${_p}Br(F)=0$ which is a contradiction. Thus, we are
left with the case that $F(\sqrt{-1})$ has no $s$-th root of unity.
Here, we claim that $F(\sqrt{-1})$ is divisible. To prove our claim,
first we note that since $F(\sqrt{-1})^*$ is 2-divisible (see the
proof of Lemma \ref{lemg5}), we must show that $F^*=F^{*p}$ for
every odd prime $p$. Let $u\in F(\sqrt{-1})^*\setminus F^*$ and $p$
be an odd prime number. Put $E=C_D(F(\sqrt{-1}))$. By the
Centralizer Theorem, $E$ is an $F(\sqrt{-1})$-central division
algebra with $\deg(E)=s\neq1$. Since $D^*$ is radicable, there is a
$w\in D^*$ such that $w^{sp}=u$. Because $w\in C_D(u)$ we have $w\in
C_D(F(u))=C_D(F(\sqrt{-1}))=E$. Taking the reduced norm, we obtain
$\operatorname{Nrd}_{E}(w)^{sp}=u^s$. Since $F(\sqrt{-1})$ contains
no $s$-th roots of unity we conclude that
$\operatorname{Nrd}_{E}(w)^p=u$. Thus $u\in F(\sqrt{-1})^{*p}$.
Also, a similar argument as in the proof of (2)$\Rightarrow$(3)
shows that $F^*$ is $p$-divisible. Thus our claim is established.
But, the divisibility of $F(\sqrt{-1})^*$ is in contrast with Lemma
3 of \cite{May72} which asserts that $F(\sqrt{-1})$ contains all
primitive roots of unity. Thus $s=1$ and hence $D=\mathcal{Q}$, as
required.

Finally, for $u$ and $p$ as above there exists $v\in D^*$ such that
$v^p=u$. But $v\in C_D(F(\sqrt{-1}))=F(\sqrt{-1})$ as $F(\sqrt{-1})$
is a maximal subfield. Therefore, $F(\sqrt{-1})$ is radically closed
and hence $F$ is radically real closed.
\end{proof}
\begin{exa}
Consider the ordinary quaternion division algebra over
$\mathbb{Q}_{\operatorname{rad}}$. By the above theorem this
division algebra is radicable. But, it is not algebraically closed,
otherwise by Baer's Theorem $\mathbb{Q}_{\operatorname{rad}}$
would be real closed which is a contradiction.
\end{exa}
\bibliographystyle{plain}
\bibliography{HMMotiee}

\begin{thebibliography}{100}

\bibitem{AkEbMoSa03}
S.~Akbari, R.~Ebrahimian, H.~Momenaee~Kermani, and A.~Salehi~Golsefidy.
\newblock Maximal subgroups of {${\rm GL}_n(D)$}.
\newblock {\em J. Algebra}, 259(1):201--225, 2003.

\bibitem{AkMa00}
S.~Akbari and M.~Mahdavi-Hezavehi.
\newblock Normal subgroups of {${\rm GL}_n(D)$} are not finitely generated.
\newblock {\em Proc. Amer. Math. Soc.}, 128(6):1627--1632, 2000.

\bibitem{AkMa02}
S.~Akbari and M.~Mahdavi-Hezavehi.
\newblock On the existence of normal maximal subgroups in division rings.
\newblock {\em J. Pure Appl. Algebra}, 171:123--131, 2002.

\bibitem{AkMaMah99}
S.~Akbari, M.~Mahdavi-Hezavehi, and M.~G. Mahmudi.
\newblock Maximal subgroups of {${\rm GL}_1(D)$}.
\newblock {\em J. Algebra}, 217(2):422--433, 1999.

\bibitem{Alb61}
A.~Adrian Albert.
\newblock {\em Structure of algebras}.
\newblock Revised printing. American Mathematical Society Colloquium
  Publications, Vol. XXIV. American Mathematical Society, Providence, R.I.,
  1961.

\bibitem{Am55}
S.~A. Amitsur.
\newblock Finite subgroups of division rings.
\newblock {\em Trans. Amer. Math. Soc.}, 80:361--386, 1955.

\bibitem{Am72}
S.~A. Amitsur.
\newblock On central division algebras.
\newblock {\em Israel J. Math.}, 12:408--420, 1972.

\bibitem{bant}
Benjamin Antieau.
\newblock On a theorem of {H}azrat and {H}oobler.
\newblock {\em Proc. Amer. Math. Soc.}, 141(8):2609--2613, 2013.

\bibitem{boulag}
M'hammed Boulagouaz.
\newblock Le gradu\'e d'une alg\`ebre \`a division valu\'ee.
\newblock {\em Comm. Algebra}, 23(11):4275--4300, 1995.

\bibitem{DorFalMah11}
H.~R. Dorbidi, R.~Fallah-Moghaddam, and M.~Mahdavi-Hezavehi.
\newblock Soluble maximal subgroups in {${\rm GL}_n(D)$}.
\newblock {\em J. Algebra Appl.}, 10(6):1371--1382, 2011.

\bibitem{Dra83}
P.~K. Draxl.
\newblock {\em Skew fields}, volume~81 of {\em London Mathematical Society
  Lecture Note Series}.
\newblock Cambridge University Press, Cambridge, 1983.

\bibitem{Ebr04}
R.~Ebrahimian.
\newblock Nilpotent maximal subgroups of {${\rm GL}_n(D)$}.
\newblock {\em J. Algebra}, 280(1):244--248, 2004.

\bibitem{EbKiMa05a}
R.~Ebrahimian, D.~Kiani, and Mahdavi-Hezavehi.
\newblock Irreducible subgroups of {$GL_1(D)$} satisfying group identities.
\newblock {\em Comm. Algebra}, 33(9):3367--3373, 2005.

\bibitem{EbKiMa05b}
R.~Ebrahimian, D.~Kiani, and M.~Mahdavi-Hezavehi.
\newblock Supersoluble crossed product criterion for division algebras.
\newblock {\em Israel J. Math.}, 145:325--331, 2005.

\bibitem{Er83}
Yu.~L. Ershov.
\newblock Henselian valuations of division rings and the group {$SK_{1}$}.
\newblock {\em Math USSR-Sb.}, 45:63--71, 1983.

\bibitem{Fau69}
R.~J. Faudree.
\newblock Locally finite and solvable subgroups of sfields.
\newblock {\em Proc. Amer. Math. Soc.}, 22:407--413, 1969.

\bibitem{FeTh63}
Walter Feit and John~G. Thompson.
\newblock Solvability of groups of odd order.
\newblock {\em Pacific J. Math.}, 13:775--1029, 1963.

\bibitem{GonMan86}
Jairo~Z. Gon{\c{c}}alves and Arnaldo Mandel.
\newblock Are there free groups in division rings?
\newblock {\em Israel J. Math.}, 53(1):69--80, 1986.

\bibitem{Hai11}
Bui~Xuan Hai.
\newblock On locally nilpotent maximal subgroups of the multiplicative group of
  a division ring.
\newblock {\em Acta Math. Vietnam.}, 36(1):113--118, 2011.

\bibitem{HaiHa10}
Bui~Xuan Hai and Dang Vu~Phuong Ha.
\newblock On locally solvable maximal subgroups of the multiplicative group of
  a division ring.
\newblock {\em Vietnam J. Math.}, 38(2):237--247, 2010.

\bibitem{HaiThi09}
Bui~Xuan Hai and Nguyen~Van Thin.
\newblock On locally nilpotent subgroups of {${\rm GL}_1(D)$}.
\newblock {\em Comm. Algebra}, 37(2):712--718, 2009.

\bibitem{Haz01}
R.~Hazrat.
\newblock {$SK_1$}-like functors for division algebras.
\newblock {\em J. Algebra}, 239(2):573--588, 2001.

\bibitem{HaMaHMi99}
R.~Hazrat, M.~Mahdavi-Hezavehi, and B.~Mirzaii.
\newblock Reduced {$K$}-theory and the group {$G(D)=D^*/F^*D'$}.
\newblock In {\em Algebraic {$K$}-theory and its applications ({T}rieste,
  1997)}, pages 403--409. World Sci. Publ., River Edge, NJ, 1999.

\bibitem{HaWa07}
R.~Hazrat and A.~R. Wadsworth.
\newblock Nontriviality of certain quotients of {$K_1$} groups of division
  algebras.
\newblock {\em J. Algebra}, 312(1):354--361, 2007.

\bibitem{HW09}
R.~Hazrat and A.~R. Wadsworth.
\newblock On maximal subgroups of the multiplicative group of a division
  algebra.
\newblock {\em J. Algebra}, 322:2528--2543, 2009.

\bibitem{hwunitary}
R.~Hazrat and A.~R. Wadsworth.
\newblock {$\rm SK_1$} of graded division algebras.
\newblock {\em Israel J. Math.}, 183:117--163, 2011.

\bibitem{hwsk1}
R.~Hazrat and A.~R. Wadsworth.
\newblock Unitary {${\rm SK}_1$} of graded and valued division algebras.
\newblock {\em Proc. Lond. Math. Soc. (3)}, 103(3):508--534, 2011.

\bibitem{Haz02}
Roozbeh Hazrat.
\newblock On central series of the multiplicative group of division rings.
\newblock {\em Algebra Colloq.}, 9(1):99--106, 2002.

\bibitem{hazhoobler}
Roozbeh Hazrat and Raymond~T. Hoobler.
\newblock {$K$}-theory of {A}zumaya algebras over schemes.
\newblock {\em Comm. Algebra}, 41(4):1268--1277, 2013.

\bibitem{HaVi05}
Roozbeh Hazrat and Uzi Vishne.
\newblock Triviality of the functor {${\rm Coker}(K_1(F)\to K_1(D))$} for
  division algebras.
\newblock {\em Comm. Algebra}, 33(5):1427--1435, 2005.

\bibitem{HazWil06}
Roozbeh Hazrat and Robert Wilson.
\newblock A decomposition theorem for {$CK_1$} of central simple algebras.
\newblock {\em Comm. Algebra}, 34(12):4573--4578, 2006.

\bibitem{Her53}
I.~N. Herstein.
\newblock Finite multiplicative subgroups in division rings.
\newblock {\em Pacific J. Math.}, 3:121--126, 1953.

\bibitem{Her56}
I.~N. Herstein.
\newblock Conjugates in division rings.
\newblock {\em Proc. Amer. Math. Soc.}, 7:1021--1022, 1956.

\bibitem{Her78}
I.~N. Herstein.
\newblock Multiplicative commutators in division rings.
\newblock {\em Israel J. Math.}, 31(2):180--188, 1978.

\bibitem{Her80}
I.~N. Herstein.
\newblock Multiplicative commutators in division rings. {II}.
\newblock {\em Rend. Circ. Mat. Palermo (2)}, 29(3):485--489 (1981), 1980.

\bibitem{HerProSch75}
Israel~N. Herstein, Claudio Procesi, and Murray Schacher.
\newblock Algebraic valued functions on noncommutative rings.
\newblock {\em J. Algebra}, 36(1):128--150, 1975.

\bibitem{Hua49}
Loo-Keng Hua.
\newblock Some properties of a sfield.
\newblock {\em Proc. Nat. Acad. Sci. U. S. A.}, 35:533--537, 1949.

\bibitem{Hua50}
Loo-Keng Hua.
\newblock On the multiplicative group of a field.
\newblock {\em Acad. Sinica Science Record}, 3:1--6, 1950.

\bibitem{Huz60}
M.~{\v{S}}. Huzurbazar.
\newblock The multiplicative group of a division ring.
\newblock {\em Soviet Math. Dokl.}, 1:433--435, 1960.

\bibitem{hwalg}
Y.-S. Hwang and A.~R. Wadsworth.
\newblock Algebraic extensions of graded and valued fields.
\newblock {\em Comm. Algebra}, 27(2):821--840, 1999.

\bibitem{hwcor}
Y.-S. Hwang and A.~R. Wadsworth.
\newblock Correspondences between valued division algebras and graded division
  algebras.
\newblock {\em J. Algebra}, 220(1):73--114, 1999.

\bibitem{jw}
Bill Jacob and Adrian Wadsworth.
\newblock Division algebras over {H}enselian fields.
\newblock {\em J. Algebra}, 128(1):126--179, 1990.

\bibitem{Kap51}
Irving Kaplansky.
\newblock A theorem on division rings.
\newblock {\em Canadian J. Math.}, 3:290--292, 1951.

\bibitem{Kar88}
Gregory Karpilovsky.
\newblock {\em Unit groups of classical rings}.
\newblock Oxford Science Publications. The Clarendon Press Oxford University
  Press, New York, 1988.

\bibitem{KeMa07}
T.~Keshavarzipour and M.~Mahdavi-Hezavehi.
\newblock Crossed product conditions for central simple algebras in terms of
  irreducible subgroups.
\newblock {\em J. Algebra}, 315(2):738--744, 2007.

\bibitem{KM05}
T.~Keshavarzipuor and M.~Mahdavi-Hezavehi.
\newblock On the non-triviality of {$G(D)$} and the existence of maximal
  subgrups of {$GL_1(D)$}.
\newblock {\em J. Algebra}, 315(2):738--744, 2005.

\bibitem{KiMa05}
D.~Kiani and M.~Mahdavi-Hezavehi.
\newblock Crossed product conditions for division algebras of prime power
  degree.
\newblock {\em J. Algebra}, 283(1):222--231, 2005.

\bibitem{KiaMa10}
D.~Kiani and M.~Mahdavi-Hezavehi.
\newblock Tits alternative for maximal subgroups of skew linear groups.
\newblock {\em Comm. Algebra}, 38(6):2354--2363, 2010.

\bibitem{KiaRam09}
Dariush Kiani and Mojtaba Ramezan-Nassab.
\newblock Maximal subgroups of {${\rm GL}_n(D)$} with finite conjugacy classes.
\newblock {\em Manuscripta Math.}, 130(3):287--293, 2009.

\bibitem{knus}
Max-Albert Knus.
\newblock {\em Quadratic and {H}ermitian forms over rings}, volume 294 of {\em
  Grundlehren der Mathematischen Wissenschaften [Fundamental Principles of
  Mathematical Sciences]}.
\newblock Springer-Verlag, Berlin, 1991.
\newblock With a foreword by I. Bertuccioni.

\bibitem{Lam01}
T.~Y. Lam.
\newblock {\em A first course in noncommutative rings}, volume 131 of {\em
  Graduate Texts in Mathematics}.
\newblock Springer-Verlag, New York, second edition, 2001.

\bibitem{LT93}
D.~W. Lewis and J.-P. Tignol.
\newblock Square class groups and {W}itt rings of central simple algebras.
\newblock {\em J. Algebra}, 154(2):360--376, 1993.

\bibitem{Lic77}
A.~Lichtman.
\newblock On subgroups of the multiplicative group of skew fields.
\newblock {\em Proc. Amer. Math. Soc.}, 63(1):15--16, 1977.

\bibitem{Lic78}
A.~I. Lichtman.
\newblock Free subgroups of normal subgroups of the multiplicative group of
  skew fields.
\newblock {\em Proc. Amer. Math. Soc.}, 71(2):174--178, 1978.

\bibitem{Lip76}
V.~A. Lipnicki{\u\i}.
\newblock On the {T}annaka-{A}rtin problem over special fields.
\newblock {\em Dokl. Akad. Nauk SSSR}, 228(1):26--29, 1976.

\bibitem{Ma94}
M.~Mahdavi-Hezavehi.
\newblock Extending valuations to algebraic division algebras.
\newblock {\em Comm. Algebra}, 22(11):4373--4378, 1994.

\bibitem{Ma95}
M.~Mahdavi-Hezavehi.
\newblock Extension of valuations on derived groups of division rings.
\newblock {\em Comm. Algebra}, 23(3):913--926, 1995.

\bibitem{ma98}
M.~Mahdavi-Hezavehi.
\newblock Commutator subgroups of finite dimensional division algebras.
\newblock {\em Rev. Roumaine Math. Pures Appl.}, 43(9-10):853--867, 1998.

\bibitem{Ma00}
M.~Mahdavi-Hezavehi.
\newblock Commutators in division rings revisited.
\newblock {\em Bull. Iranian Math. Soc.}, 26(2):7--88, 2000.

\bibitem{Ma01}
M.~Mahdavi-Hezavehi.
\newblock Free subgroups in maximal subgroups of {${\rm GL}_1(D)$}.
\newblock {\em J. Algebra}, 241(2):720--730, 2001.

\bibitem{Ma02}
M.~Mahdavi-Hezavehi.
\newblock Maximal subgroups of skew linear groups.
\newblock {\em Algebra Colloq.}, 9(1):1--6, 2002.

\bibitem{Mah04}
M.~Mahdavi-Hezavehi.
\newblock Tits alternative for maximal subgroups of {${\rm GL}_n(D)$}.
\newblock {\em J. Algebra}, 271(2):518--528, 2004.

\bibitem{MaAk95}
M.~Mahdavi-Hezavehi and S.~Akbari.
\newblock A generalization of {K}aplansky's theorem.
\newblock {\em Bull. Iranian Math. Soc.}, 21(1):1--4, 1995.

\bibitem{MaAk98}
M.~Mahdavi-Hezavehi and S.~Akbari.
\newblock Some special subgroups of {${\rm GL}_n(D)$}.
\newblock {\em Algebra Colloq.}, 5(4):361--370, 1998.

\bibitem{MaAkMeHa95}
M.~Mahdavi-Hezavehi, S.~Akbari-Feyzaabaadi, M.~Mehraabaadi, and
  H.~Hajie-Abolhassan.
\newblock On derived groups of division rings. {II}.
\newblock {\em Comm. Algebra}, 23(8):2881--2887, 1995.

\bibitem{MaMahYa00}
M.~Mahdavi-Hezavehi, M.~G. Mahmudi, and S.~Yasamin.
\newblock Finitely generated subnormal subgroups of {${\rm GL}_n(D)$} are
  central.
\newblock {\em J. Algebra}, 225(2):517--521, 2000.

\bibitem{MaMo97}
M.~Mahdavi-Hezavehi and D.~Mojdeh.
\newblock On derived groups of division ring extensions.
\newblock {\em Sci. Iran.}, 4(1-2):53--56, 1997.

\bibitem{MaMo11}
M.~Mahdavi-Hezavehi and M.~Motiee.
\newblock Division algebras with radicable multiplicative groups.
\newblock {\em Comm. Algebra}, 39(11):4084--4096, 2011.

\bibitem{MaMo12}
M.~Mahdavi-Hezavehi and M.~Motiee.
\newblock A criterion for the triviality of {$G(D)$} and its applications to
  the multiplicative structure of {$D$}.
\newblock {\em Comm. Algebra}, 40(7):2645--2670, 2012.

\bibitem{MaTi03}
M.~Mahdavi-Hezavehi and J.-P. Tignol.
\newblock Cyclicity conditions for division algebras of prime degree.
\newblock {\em Proc. Amer. Math. Soc.}, 131(12):3673--3676 (electronic), 2003.

\bibitem{May72}
Warren May.
\newblock Multiplicative groups of fields.
\newblock {\em Proc. London Math. Soc. (3)}, 24:295--306, 1972.

\bibitem{Mer83}
A.~S. Merkurjev.
\newblock Brauer groups of fields.
\newblock {\em Comm. Algebra}, 11:2611--2624, 1983.

\bibitem{Mer95}
A.~S. Merkurjev.
\newblock {$K$}-theory of simple algebras.
\newblock In {\em {$K$}-theory and algebraic geometry: connections with
  quadratic forms and division algebras ({S}anta {B}arbara, {CA}, 1992)},
  volume~58 of {\em Proc. Sympos. Pure Math.}, pages 65--83. Amer. Math. Soc.,
  Providence, RI, 1995.

\bibitem{millar}
J.~R. Millar.
\newblock {\em K-theory of Azumaya algebras}.
\newblock Ph.D. thesis, Queen's University Belfast, United Kingdom 2010,
  arXiv:1101.1468v1.

\bibitem{mou}
K.~Mounirh.
\newblock Nicely semiramified division algebras over {H}enselian fields.
\newblock {\em Int. J. Math. Math. Sci.}, (4):571--577, 2005.

\bibitem{NaMa43}
T~Nakayama and Y~Matsushima.
\newblock \"{U}ber die multiplikative {G}ruppe einer {$p$}-adischen
  {D}ivisionsalgebra.
\newblock {\em Proc. Imp. Acad. Tokyo}, 19:622--628, 1943.

\bibitem{Niv41}
Ivan Niven.
\newblock Equations in quaternions.
\newblock {\em Amer. Math. Monthly}, 48:654--661, 1941.

\bibitem{Pas77}
Donald~S. Passman.
\newblock {\em The algebraic structure of group rings}.
\newblock Pure and Applied Mathematics. Wiley-Interscience [John Wiley \&
  Sons], New York, 1977.

\bibitem{Pie82}
Richard~S. Pierce.
\newblock {\em Associative algebras}, volume~88 of {\em Graduate Texts in
  Mathematics}.
\newblock Springer-Verlag, New York, 1982.
\newblock Studies in the History of Modern Science, 9.

\bibitem{platonov}
V.~P. Platonov.
\newblock The {T}annaka-{A}rtin problem, and reduced {$K$}-theory.
\newblock {\em Izv. Akad. Nauk SSSR Ser. Mat.}, 40(2):227--261, 469, 1976.

\bibitem{PutYaq74}
Mohan~S. Putcha and Adil Yaqub.
\newblock Multiplicative commutators in division rings.
\newblock {\em Math. Japon.}, 19:111--115, 1974.

\bibitem{RSegSei02}
Segev~Y. Rapinchuk, A.~S. and G.~M. Seitz.
\newblock Finite quotients of the multiplicative group of a finite dimensional
  division algebra are solvable.
\newblock {\em J. Amer. Math. Soc.}, 15:929--978, 2002.

\bibitem{rei}
I.~Reiner.
\newblock {\em Maximal orders}.
\newblock Academic Press [A subsidiary of Harcourt Brace Jovanovich,
  Publishers], London-New York, 1975.
\newblock London Mathematical Society Monographs, No. 5.

\bibitem{Rob82}
Derek John~Scott Robinson.
\newblock {\em A course in the theory of groups}, volume~80 of {\em Graduate
  Texts in Mathematics}.
\newblock Springer-Verlag, New York, 1982.

\bibitem{saltman}
David~J. Saltman.
\newblock {\em Lectures on division algebras}, volume~94 of {\em CBMS Regional
  Conference Series in Mathematics}.
\newblock Published by American Mathematical Society, Providence, RI, 1999.

\bibitem{Schi50}
O.~F.~G. Schilling.
\newblock {\em The {T}heory of {V}aluations}.
\newblock Mathematical Surveys, No. 4. American Mathematical Society, New York,
  N. Y., 1950.

\bibitem{Sco57}
W.~R. Scott.
\newblock On the multiplicative group of a division ring.
\newblock {\em Proc. Amer. Math. Soc.}, 8:303--305, 1957.

\bibitem{Sco87}
W.~R. Scott.
\newblock {\em Group theory}.
\newblock Dover Publications Inc., New York, second edition, 1987.

\bibitem{Sh05}
M.~Shirvani.
\newblock On soluble-by-finite subgroups of division algebras.
\newblock {\em J. Algebra}, 294(1):255--277, 2005.

\bibitem{ShWe86}
M.~Shirvani and B.~A.~F. Wehrfritz.
\newblock {\em Skew linear groups}, volume 118 of {\em London Mathematical
  Society Lecture Note Series}.
\newblock Cambridge University Press, Cambridge, 1986.

\bibitem{Stu64}
Charles~J. Stuth.
\newblock A generalization of the {C}artan-{B}rauer-{H}ua theorem.
\newblock {\em Proc. Amer. Math. Soc.}, 15:211--217, 1964.

\bibitem{tigwad}
J.-P. Tignol and A.~R. Wadsworth.
\newblock Value functions and associated graded rings for semisimple algebras.
\newblock to appear in Trans. Amer. Math. Soc.; preprint available (No. 247)
  at: {\tt http://www.math.uni-bielefeld.de/LAG/} .

\bibitem{TW87}
J.-P. Tignol and A.~R. Wadsworth.
\newblock Totally ramified valuations on finite-dimensional division algebras.
\newblock {\em Trans. Amer. Math. Soc.}, 302(1):223--250, 1987.

\bibitem{Tit72}
J.~Tits.
\newblock Free subgroups in linear groups.
\newblock {\em J. Algebra}, 20:250--270, 1972.

\bibitem{Wad99}
A.~R. Wadsworth.
\newblock Valuation theory on finite dimensional division algebras.
\newblock In {\em Valuation theory and its applications, {V}ol. {I}
  ({S}askatoon, {SK}, 1999)}, volume~32 of {\em Fields Inst. Commun.}, pages
  385--449. Amer. Math. Soc., Providence, RI, 2002.

\bibitem{wyan}
A.~R. Wadsworth and V.~I. Yanchevski{\u\i}.
\newblock Unitary {${\rm SK}_1$} for a graded division ring and its quotient
  division ring.
\newblock {\em J. Algebra}, 352:62--78, 2012.

\bibitem{Wad86}
Adrian~R. Wadsworth.
\newblock Extending valuations to finite-dimensional division algebras.
\newblock {\em Proc. Amer. Math. Soc.}, 98(1):20--22, 1986.

\bibitem{Wad90}
Adrian~R. Wadsworth.
\newblock In a division algebra of even degree {$-1$} is a product of squares.
\newblock {\em Bull. Soc. Math. Belg. S\'er. B}, 42(1):103--104, 1990.

\bibitem{Wa50}
Sh~Wang.
\newblock On the commutator group of a simple algebra.
\newblock {\em Amer. J. Math.}, 72:323--334, 1950.

\bibitem{wedd}
J.~H.~M. Wedderburn.
\newblock On division algebras.
\newblock {\em Trans. Amer. Math. Soc.}, 22(2):129--135, 1921.

\bibitem{Weh71}
B.~A.~F. Wehrfritz.
\newblock {$2$}-generator conditions in linear groups.
\newblock {\em Arch. Math. (Basel)}, 22:237--240, 1971.

\bibitem{We90a}
B.~A.~F. Wehrfritz.
\newblock Crossed product criteria and skew linear groups. {II}.
\newblock {\em Michigan Math. J.}, 37(2):293--303, 1990.

\bibitem{We90b}
B.~A.~F. Wehrfritz.
\newblock Crossed products, control in group algebras and absolutely
  irreducible groups.
\newblock {\em J. London Math. Soc. (2)}, 42(2):209--225, 1990.

\bibitem{We90c}
B.~A.~F. Wehrfritz.
\newblock Some matrix groups over finite-dimensional division algebras.
\newblock {\em Proc. Edinburgh Math. Soc. (2)}, 33(1):97--111, 1990.

\bibitem{We06}
B.~A.~F. Wehrfritz.
\newblock On a recent theorem of {M}.\ {S}hirvani on subgroups of division
  algebras.
\newblock {\em J. Algebra}, 300(1):25--34, 2006.

\bibitem{Wehr07}
B.~A.~F. Wehrfritz.
\newblock Nilpotent subgroups of finite-dimensional division algebras.
\newblock {\em Bull. Lond. Math. Soc.}, 39(3):359--365, 2007.

\bibitem{Weh07}
B.~A.~F. Wehrfritz.
\newblock On soluble skew linear groups over finite-dimensional division
  algebras.
\newblock {\em J. Pure Appl. Algebra}, 209(2):301--309, 2007.

\bibitem{y}
V.~I. Yan{\v{c}}evski{\u\i}.
\newblock Reduced unitary {$K$}-theory and division algebras over {H}enselian
  discretely valued fields.
\newblock {\em Izv. Akad. Nauk SSSR Ser. Mat.}, 42(4):879--918, 1978.
\newblock (in Russian) English trans., Math. USSR Izvestiya, {\bf 13} (1979),
  175--213.

\bibitem{Zal93}
A.~E. Zalesskii.
\newblock Linear groups.
\newblock In {\em Algebra, {IV}}, volume~37 of {\em Encyclopaedia Math. Sci.},
  pages 97--196. Springer, Berlin, 1993.

\end{thebibliography}
\end{document}